\documentclass{article}
\usepackage{lmodern}
\usepackage[utf8]{inputenc}
\usepackage[T1]{fontenc}
\usepackage{amssymb}
\usepackage{enumitem}
\usepackage{stmaryrd}
\usepackage{amsmath}
\usepackage{centernot}
\usepackage[nottoc]{tocbibind}
\usepackage{mathtools}
\usepackage{geometry}
\usepackage{dsfont}
\usepackage[dvipsnames, svgnames, x11names]{xcolor}
\usepackage{pgf}
\usepackage{tikz}
\usepackage[final]{pdfpages} 
\usetikzlibrary{arrows.meta, shapes,positioning, animations, shapes.geometric,decorations.pathmorphing, decorations.shapes, backgrounds}
\usepackage{float}
\usepackage{appendix}

\usepackage{framed}
\usepackage{array}
\usepackage{graphicx}
\usepackage{fullpage}
\usepackage{caption}
\usepackage{subcaption}
\usepackage{lipsum}
\usepackage[most]{tcolorbox}
\usepackage[breaklinks, pdfencoding=auto, unicode]{hyperref}
\usepackage{breakurl}
\usepackage{bbold}
\usepackage{amsthm}
\usepackage{thmtools}
\usepackage{adjustbox}
\usepackage{verbatim}
\usepackage[affil-it]{authblk}

\usepackage[
style=alphabetic,
sorting=nty,
block=ragged,
maxnames=50,
giveninits=true %--- (pour la cohérence, comme on a pas tous les prénoms)
]{biblatex}

\usepackage{cleveref}
\hypersetup{breaklinks=true}

\usepackage{soul}

\DeclareMathOperator{\Cat}{\mathrm{Cat}}
\DeclareMathOperator{\FCat}{\mathrm{FCat}}
\DeclareMathOperator{\Var}{\mathrm{Var}}
\DeclareMathOperator{\E}{\mathds{E}}
\DeclareMathOperator{\1}{\mathds{1}}
\renewcommand{\P}{\mathds{P}}

\renewcommand{\leq}{\leqslant}

\renewcommand{\geq}{\geqslant}
\definecolor{darkspringgreen}{rgb}{0.09, 0.45, 0.27}

\newcommand{\abs}[1]{
\left|#1\right|
}

\hypersetup{colorlinks, allcolors = {RoyalBlue3}}

\newcommand{\norm}[1]{\left\lVert#1\right\rVert}

\newcommand{\setEnvironmentQed}[2]{
  \AtBeginEnvironment{#1}{%
    \pushQED{\qed}\renewcommand{\qedsymbol}{#2}%
  }
  \AtEndEnvironment{#1}{\popQED}
}

\newtheoremstyle{MyTheo}
                {}          % Space above
                {}          % Space below
                {\normalfont}  % Body font
                {}          % Indent amount
                {\bfseries} % Head font
                {.}         % Punctuation after head
                { }         % Space after theorem head
                {}          % Theorem head spec

\theoremstyle{MyTheo}

\newtheorem{Theo}{Theorem}[section]
\newtheorem{Defi}[Theo]{Definition}
\newtheorem{Lemm}[Theo]{Lemma}
\newtheorem{Rema}[Theo]{Remark}
\newtheorem{Prop}[Theo]{Proposition}
\newtheorem{Coro}[Theo]{Corollary}

\newtheorem*{Theonn}{Theorem}
\newtheorem*{Propnn}{Proposition}
\newtheorem*{Coronn}{Corollary}

\setEnvironmentQed{Theo}{\ensuremath{\scalebox{0.5}{$\diamondsuit$}}}
\setEnvironmentQed{Defi}{\ensuremath{\scalebox{0.5}{$\diamondsuit$}}}
\setEnvironmentQed{Lemm}{\ensuremath{\scalebox{0.5}{$\diamondsuit$}}}
\setEnvironmentQed{Rema}{\ensuremath{\scalebox{0.5}{$\diamondsuit$}}}
\setEnvironmentQed{Prop}{\ensuremath{\scalebox{0.5}{$\diamondsuit$}}}
\setEnvironmentQed{Coro}{\ensuremath{\scalebox{0.5}{$\diamondsuit$}}}
\setEnvironmentQed{Exam}{\ensuremath{\scalebox{0.5}{$\diamondsuit$}}}
\setEnvironmentQed{Nota}{\ensuremath{\scalebox{0.5}{$\diamondsuit$}}}
\setEnvironmentQed{Theonn}{\ensuremath{\scalebox{0.5}{$\diamondsuit$}}}
\setEnvironmentQed{Propnn}{\ensuremath{\scalebox{0.5}{$\diamondsuit$}}}
\setEnvironmentQed{Coronn}{\ensuremath{\scalebox{0.5}{$\diamondsuit$}}}

\newtheoremstyle{MyProo}
                {}          % Space above
                {}          % Space below
                {\normalfont}  % Body font
                {}          % Indent amount
                {\sc} % Head font
                {:}         % Punctuation after head
                { }         % Space after theorem head
                {}          % Theorem head spec

\theoremstyle{MyProo}

\newtheorem*{Proo}{Proof}
\setEnvironmentQed{Proo}{\ensuremath{\scalebox{0.8}{$\blacksquare$}}}

\crefname{Theo}{Theorem}{Theorems}
\crefname{Defi}{Definition}{Definitions}
\crefname{Lemm}{Lemma}{Lemmas}
\crefname{Rema}{Remark}{Remarks}
\crefname{Prop}{Proposition}{Propositions}
\crefname{Coro}{Corollary}{Corollaries}
\crefname{Exam}{Example}{Examples}

\tikzset{decorate sep/.style 2 args=
{decorate,decoration={shape backgrounds,shape=circle,shape size=#1,shape sep=#2}}}
\begin{document}

\title{\vspace{-1cm}A parameterized family of balance indices\break for phylogenetic networks}

\author[1]{François Bienvenu}
\author[1]{Jean-Jil Duchamps}
\author[1]{Hadrien Maffioli}

\renewcommand\Affilfont{\itshape\small\renewcommand{\baselinestretch}{1}}
\affil[1]{Univ.\ Marie et Louis Pasteur, CNRS, LmB (UMR 6623), F-25000 Besançon, France.\vspace{1ex}}

\maketitle

\begin{abstract}
We introduce a new family of balance indices for phylogenetic networks:
the $H_\alpha$ indices, where $\alpha$ is a positive real number.
This family includes the $B_2$ index as a special case ($\alpha = 1$) and
provides a natural extension of the Sackin index to phylogenetic networks.
We show that the $H_\alpha$ indices share many structural properties
with the $B_2$ index, most notably a ``grafting property'' that makes it
possible to express the $H_\alpha$ index of a network in terms of the
$H_\alpha$ indices of its biconnected components.
These properties allow us to identify networks that minimize /
maximize $H_\alpha$
for various classes of phylogenetic networks, and to study
its distribution for several models of random trees and networks (in
particular, Galton--Watson trees and binary Markov branching trees, with a
focus on the Yule and PDA models).
Finally, we show how local limits can be used to analyze the asymptotic
behavior of $H_\alpha$ for large trees and networks, and we obtain
general results for the moments of $H_\alpha$ for a broad class of random
phylogenetic networks known as blowups of Galton--Watson trees.
\end{abstract}

\tableofcontents

\section{Introduction}
\subsection{Biological context}

Traditionally, the evolutionary relationships between organisms, species or
genes have been represented using phylogenetic trees.
To study such trees, phylogeneticists have at their disposal various
summary statistics that quantify specific aspects of tree structure.
Among those statistics, the class of \emph{balance indices} stands out as one
of the most important;
see \cite{fischerTreeBalanceIndices2023} for a comprehensive
survey of the diversity of tree balance indices.

Over the last decades, the recognition of phenomena such as hybridization or
horizontal gene transfer (often gathered under the general term
``reticulated evolution'') has challenged the exclusive use of trees
in phylogenetics and motivated the use of \emph{phylogenetic networks}
(see e.g.~\cite{husonPhylogeneticNetworksConcepts2010}
for an introduction to the topic).
Thus, in recent years, efforts have been made to extend
the definition of existing balance indices from trees to networks:
these have included the Sackin index
\cite{zhangSackinIndexSimplex2021,fuchsSackinIndicesLabeled2025}, the $B_2$
index \cite{bienvenuRevisitingShaoSokals2021,bienvenu$B_2$IndexGalled2024} and
the weighted total cophenetic index \cite{KnuverFischerHellmuthWicke2024}.

In this article, we introduce a new family of phylogenetic network
balance indices based on the statistical properties of random walks on
these networks, similarly as a possible definition of the $B_2$ index.
Formally, these indices $(H_\alpha)_{\alpha \geqslant 0}$ can be seen
as $q$-analogues of the $B_2$ index that use the structural
$\alpha$-entropy (also known as the Tsallis entropy, see below)
instead of the Shannon entropy. They have the advantage of admitting a
clear-cut interpretation as balance indices -- unlike, for instance,
existing extensions of the Sackin index (which are mathematically
natural but whose interpretation as balance indices is not intuitive).
The family $(H_\alpha)_{\alpha \geqslant 0}$ includes $B_2$, for
$\alpha = 1$, and, for $\alpha=2$, an analogue of the Gini--Simpson
index commonly used in biology to quantify
biodiversity~\cite{Simpson1949,Hurlbert1971}. Finally, the Sackin
index of a tree $T$ can be recovered from the behavior of
$H_\alpha(T)$ as $\alpha$ approaches 0. This provides a new, fairly
tractable extension of the Sackin index to phylogenetic networks.

We start our presentation by formally defining $H_\alpha$ and observing some
elementary properties of the functions $G \mapsto H_\alpha(G)$ and $\alpha
\mapsto H_\alpha(G)$; properties that we then use to study the range of
$H_\alpha$ over different classes of phylogenetic networks.
We continue with a detailed study of the distribution of~$H_\alpha$ for several
classic models of random trees, such as the Yule, PDA and Galton--Watson
models.
Next, we show how local limits can be used to study the asymptotic behavior of
$H_\alpha(G_n)$ as the size of $G_n$ goes to infinity.
Our primary focus is a class of models of random phylogenetic networks
known as blowups of Galton--Watson trees.
These include, among others, the PDA model, uniform leaf-labeled galled trees
and, more generally, uniform leaf-labeled level-$k$ networks
\cite{stuflerBranchingProcessApproach2022}.
Finally, we conclude with a proof-of-concept illustration
of the statistical power of $H_\alpha$ to distinguish between different
models of random trees.

\subsection{Setting and notation}

Let us start by recalling, informally, the idea behind the $B_2$ index.
Following \cite{bienvenuRevisitingShaoSokals2021}, 
picture water flowing from the root of a phylogenetic network
through its edges, splitting evenly among the children of each vertex it
reaches and accumulating in the leaves of the network.
The $B_2$ index measures the balance of the network as the uniformity
of the resulting distribution of water among the leaves~-- specifically, as its
Shannon entropy.

There are, however, other measures of uniformity that can be used.
For instance, in biology one such well-known
measure is the Gini--Simpson index, rooted in Gini's work on
statistical variability \cite{cerianiOriginsGiniIndex2012} and later adopted
as a standard measure of biodiversity~\cite{Simpson1949,Hurlbert1971}.
Both the Shannon entropy and the Gini--Simpson index are -- up to some
normalization -- special cases
of a family of measures of uncertainty/evenness known as
\emph{structural $\alpha$-entropies}.
First introduced by Havrda and Charvát
\cite{havrdaQuantificationMethodClassification1967a} in the context of
information theory, the structural $\alpha$-entropy is also known as the
\emph{Tsallis entropy}, following its reintroduction by Tsallis
in statistical physics, with a slightly different normalizing constant
\cite{tsallisPossibleGeneralizationBoltzmannGibbs1988}.

The definition of the $H_\alpha$ index is the same as that of the $B_2$ index,
but with the Shannon entropy replaced by the more general structural
$\alpha$-entropy.
To introduce it formally, we must first set up some terminology.

\begin{Defi}
A \emph{phylogenetic network} is a directed acyclic graph $G$ such that:
\begin{enumerate}
    \item the vertex set of $G$ is finite or countably infinite;
    \item $G$ is \emph{locally finite} (i.e.\ each vertex has a finite number
      of neighbors);
    \item $G$ is \emph{rooted} (i.e.\ there exists a unique vertex
      $\rho$ with in-degree 0, called the root) and every vertex can
      be reached from the root. \qedhere
\end{enumerate} 
\end{Defi}

Given a phylogenetic network $G$, the \emph{simple random walk on $G$} is
the stochastic process defined as follows:
start at the root at time $t=0$; then, conditional on the random walk being on
vertex~$v$ at time $t$, choose one of the children of $v$ uniformly at random
and move to that vertex at time $t+1$.
The process stops upon reaching a leaf.
    
On a finite network, the random walk ends in a leaf after a finite number of
steps.
Thus, if we denote by $\mathcal{L}_G$ the leaf set of $G$, the simple random
walk induces a probability distribution $(p_\ell)_{\ell \in \mathcal{L}_G}$ on
$\mathcal{L}_G$.
For any nonnegative real number $\alpha \neq 1$, we
define the $H_\alpha$ index of $G$ as
\begin{equation}
    H_\alpha(G) = \frac{1}{1-2^{1-\alpha}}
    \sum_{\ell \in \mathcal{L}_G} p_\ell \big(1 - p_\ell^{\alpha -1}\big)
    = \frac{1}{1-2^{1-\alpha}}
    \left(1 - \sum_{\ell \in \mathcal{L}_G} p_\ell^{\alpha}\right) .
    \label{defHalpha}
\end{equation}
It is straightforward to see that, as $\alpha \to 1$,
\begin{equation*}
  \lim_{\alpha \to 1} H_\alpha(G) =
    - \sum_{\ell \in \mathcal{L}_G}\! p_\ell \log_2(p_\ell) = B_2(G)\,, 
\end{equation*}
where $\log_2$ denotes the base-2 logarithm. Thus, we extend the
definition of $H_\alpha$ to $\alpha = 1$ by continuity and set $H_1 = B_2$. 
Also note that for $\alpha = 0$ we get
$H_0(G) = |\mathcal{L}_G| - 1$ and that for $\alpha = 2$ we recover twice the
Gini--Simpson index of $(p_\ell)_{\ell \in \mathcal{L}_G}$.

Although infinite networks may not seem relevant to phylogenetics, they can be
useful to study the properties of large networks through local
limits~\cite{stuflerBranchingProcessApproach2022,
bienvenuMathematicallyTractableModels2025}.
We thus extend the definition of $H_\alpha$ to infinite phylogenetic networks
using the approach of \cite{bienvenu$B_2$IndexGalled2024}, which
can be summarized as follows: the first step is to introduce a suitable
notion of \emph{boundary} for the network.
Informally, this boundary consists of the leaves (that is, the vertices in
which the random walk can get trapped) together with ``points at infinity'' to
which the random walk can escape (that is, the different directions in which it
can go to infinity).
The simple random walk induces a probability distribution $\mu$ on this
boundary, and we can therefore define the $H_\alpha$ index of an infinite
phylogenetic network as the structural $\alpha$-entropy of $\mu$.
We refer the reader to Appendix~\ref{AppInfiniteNetworks} for the technical
details.

\subsection[Basic properties of the \texorpdfstring{$H_\alpha$}{H-alpha} index]{%
Basic properties of the \texorpdfstring{$\boldsymbol{H_\alpha}$}{H-alpha} index} \label{basicProperties}

In this section we present a few straightforward but useful
properties of the $H_\alpha$ index.
Let us start with the announced relationship
with the Sackin index in the context of
binary trees:
for any finite phylogenetic network~$G$, the computation of the
derivative of $\alpha \mapsto H_\alpha(G)$ yields the following expansion
as $\alpha \downarrow 0$ :
\begin{equation*}
     H_\alpha(G) = (|\mathcal{L}_G| -1) + \left( 2\log(2)(|\mathcal{L}_G|-1) +
       \sum_{\ell \in \mathcal{L}_G} \log(p_\ell) \right) \alpha + o(\alpha),
\end{equation*}
where $\log$ is the natural logarithm.
Now note that, for binary trees, $p_\ell = 2^{-\delta_\ell}$ where
$\delta_\ell$ is the depth of the leaf $\ell$ (the number of edges
separating it from the root).
Thus, for a binary tree~$T$, $\sum_{\ell \in \mathcal{L}_T} \log(p_\ell) =
-\log(2)\, \mathrm{Sackin}(T)$, where
$\mathrm{Sackin}(T) = \sum_{\ell\in \mathcal{L}_T} \delta_\ell$
is the Sackin index of~$T$ \cite{ShaoSokal1990}.
Thus, by setting, for any phylogenetic network $G$, 
\begin{equation*}
    \mathrm{Sackin}(G) = 2(|\mathcal{L}_G| -1) -\log(2)^{-1} \cdot
       \frac{\partial H_\alpha(G)}{\partial \alpha}(0), 
\end{equation*}
we get an extension of the Sackin index to phylogenetic networks
that agrees with its usual definition for binary trees (note however that this
extension does not agree with the usual definition of the Sackin index for
multifurcating trees).

The next proposition provides a comparison between the $H_\alpha$ indices for
different values of $\alpha$.

\begin{Prop}
For any finite phylogenetic network $G$ such that $p_\ell \leqslant 1/2$ for
all $\ell \in \mathcal{L}_G$, $H_\alpha(G)$ is nonincreasing in $\alpha$.
Moreover, for any finite phylogenetic network $G$ and any $1 < \alpha < \beta$,
we have
\begin{equation*}
  H_\alpha(G) \leqslant \frac{1-2^{1-\beta}}{1-2^{1-\alpha}} H_\beta(G).
  \qedhere
\end{equation*}
\end{Prop}

\begin{Proo}
To show that $\alpha \mapsto H_\alpha(G)$ is nonincreasing
for any finite phylogenetic network $G$ such that $p_\ell \leqslant 1/2$ for
all leaf $\ell$, it suffices to note that
$H_\alpha(G) = \sum_{\ell \in \mathcal{L}_G} p_\ell  f_{p_\ell}(\alpha)$,
where 
\begin{equation*}
    f_{p_\ell} : \alpha \in \mathbb{R}_+^* \mapsto \frac{1-p_\ell^{\alpha
    -1}}{1-(1/2)^{\alpha-1}}
\end{equation*}
is nonincreasing whenever $p_\ell \leqslant 1/2$.
To prove the second inequality, note that
\begin{equation*}
 \sum_{\ell \in \mathcal{L}_G}\! p_\ell^\alpha
  \, \geqslant
 \sum_{\ell \in \mathcal{L}_G}\! p_\ell^\beta
\end{equation*}
whenever $\beta > \alpha$, because
$x \mapsto p_\ell^x$ is decreasing on $\mathbb{R}_+$
for all $p_\ell \in\;]0; 1[$; and that the normalizing
factor $1-2^{1-x}$ is positive for $x > 1$.
\end{Proo}

\begin{Rema}
The condition ``$p_\ell \leqslant 1/2$ for all $\ell \in \mathcal{L}_G$''
would not be necessary if we had used the
Tsallis normalization factor
$(\alpha -1)^{-1}$, instead of the
structural $\alpha$-entropy normalization factor
$(1-2^{1-\alpha})^{-1}$.
However, as will become apparent, normalizing by
$(1-2^{1-\alpha})^{-1}$ is particularly convenient when working with
binary trees, which are central in phylogenetics -- hence our choice.
Also note that the condition $p_\ell \leqslant 1/2$ for all $\ell \in
\mathcal{L}_G$ is always verified for \emph{trees} with more than one leaf,
but not necessarily for networks.
\end{Rema}

\begin{Rema} \label{remaClassement}
One of the main goals of balance indices is to compare
phylogenetic trees/networks. 
Formally, each balance index defined on a class of phylogenetic networks
induces a total preorder on that class.
As the following example demonstrates, different values of $\alpha$ can induce
different preorders; that is, for $\alpha \neq \beta$,
\begin{equation*}
    H_\alpha(G) \leqslant H_\alpha(G') \centernot \Longrightarrow H_\beta(G) \leqslant H_\beta(G').
\end{equation*}
To see this, consider the trees $T$ and $T'$ depicted in Fig.~\ref{contreEx}.
One can check that, for $\alpha = 0.2$ and $\beta = 0.9$, we have
$H_\alpha(T) \approx 5.37 > H_\alpha(T') \approx 5.23$
and
$H_\beta(T) \approx 2.57 < H_\beta(T') \approx 2.64$.
Thus, this example shows that, even when we restrict ourselves to binary trees,
the $H_\alpha$ indices do \emph{not} induce the same preorder.
This means that  different values of $\alpha$ capture different aspects of
the structure of phylogenetic networks, and that
there is genuinely more information about
a network $G$ in the family $(H_\alpha(G))_{\alpha \geqslant 0}$ than 
in $H_\alpha(G)$ for a single parameter $\alpha$.

Nevertheless, there is a simple sufficient condition to ensure that two
phylogenetic networks $G$ and $G'$ satisfy $H_\alpha(G) \leq H_\alpha(G')$ for
all values of $\alpha$: indeed, it is well-known that the structural
$\alpha$-entropy is a Schur-concave function.
As a result, if $G$ and $G'$ are two phylogenetic networks with $n$ leaves
such that $(p'_\ell)_{\ell\in \mathcal{L}_{G'}}$
is \emph{majorized} by $(p_\ell)_{\ell\in \mathcal{L}_G}$, then
$H_\alpha(G) \leqslant H_\alpha(G')$ for all $\alpha > 0$.
See e.g.\ \cite{ostrowskiQuelquesApplicationsFonctions1984a,
marshallSchurConvexFunctions2011} for more on Schur-convexity and majorization.
\end{Rema}

\begin{figure}[H]
\centering
\begin{subfigure}[t]{.4\textwidth}
\begin{tikzpicture}[main/.style = {draw, circle, minimum size = 7pt, inner sep =1pt}]
    \tikzstyle{dot}=[draw,circle,gray,
                  top color= white, text=gray,minimum size=7pt, inner sep = 1]
    \tikzstyle{phantom}=[opacity = 0, minimum size = 7pt]
    \node[main, label ={[font=\footnotesize] \o} ] (root) at (0, 0) {};
    \node[main] (1) at (-1, -1) {};
    \node[main] (2) at (1, -1) {};
    \node[main] (21) at (0, -2) {};
    \node[main] (22) at (2, -2) {};
    \node[main] (211) at (-0.5, -3) {};
    \node[main] (212) at (0.5, -3) {};
    \node[main] (221) at (1.5, -3) {};
    \node[main] (222) at (2.5, -3) {};
    \node[main] (2111) at (-0.75, -4) {};
    \node[main] (2112) at (-0.25, -4) {};
    \node[main] (2121) at (0.25, -4) {};
    \node[main] (2122) at (0.75, -4) {};
    \node[main] (2211) at (1.25, -4) {};
    \node[main] (2212) at (1.75, -4) {};
    \draw[->] (root) -- (1);
    \draw[->] (root) -- (2);
    \draw[->] (2) -- (21);
    \draw[->] (2) -- (22);
    \draw[->] (21) -- (212);
    \draw[->] (21) -- (211);
    \draw[->] (22) -- (222);
    \draw[->] (22) -- (221);
    \draw[->] (211) -- (2111);
    \draw[->] (211) -- (2112);
    \draw[->] (212) -- (2121);
    \draw[->] (212) -- (2122);
    \draw[->] (221) -- (2211);
    \draw[->] (221) -- (2212);
\end{tikzpicture}
    \caption{Representation of the tree $T$.}
    \label{T}
\end{subfigure}
~
\begin{subfigure}[t]{.4\textwidth}
\begin{tikzpicture}[main/.style = {draw, circle, minimum size = 7pt, inner sep =1pt}]
    \tikzstyle{dot}=[draw,circle,gray,
                  top color= white, text=gray,minimum size=7pt, inner sep = 1]
    \tikzstyle{phantom}=[opacity = 0, minimum size = 7pt]
    \node[main, label ={[font=\footnotesize] \o} ] (root) at (0, 0) {};
    \node[main] (1) at (-1, -1) {};
    \node[main] (2) at (1, -1) {};
    \node[main] (11) at (-0.5, -2) {};
    \node[main] (12) at (-1.5, -2) {};
    \node[main] (21) at (0.5, -2) {};
    \node[main] (22) at (1.5, -2) {};
    \node[main] (221) at (1, -3) {};
    \node[main] (222) at (2, -3) {};
    \node[main] (2221) at (1.5, -4) {};
    \node[main] (2222) at (2.5, -4) {};
    \node[main] (22221) at (2, -5) {};
    \node[main] (22222) at (3, -5) {};
    \node[main] (222221) at (2.5, -6) {};
    \node[main] (222222) at (3.5, -6) {};
    \draw[->] (root) -- (1);
    \draw[->] (root) -- (2);
    \draw[->] (2) -- (21);
    \draw[->] (2) -- (22);
    \draw[->] (1) -- (11);
    \draw[->] (1) -- (12);
    \draw[->] (22) -- (222);
    \draw[->] (22) -- (221);
    \draw[->] (222) -- (2221);
    \draw[->] (222) -- (2222);
    \draw[->] (2222) -- (22221);
    \draw[->] (2222) -- (22222);
    \draw[->] (22222) -- (222221);
    \draw[->] (22222) -- (222222);
\end{tikzpicture}
\caption{Representation of the tree $T'$.}
\label{T'}
\end{subfigure}
\caption{Two binary trees $T$ and $T'$ such that $H_\alpha(T)-H_\alpha(T')$ changes sign as $\alpha$ varies.}
\label{contreEx}
\end{figure}
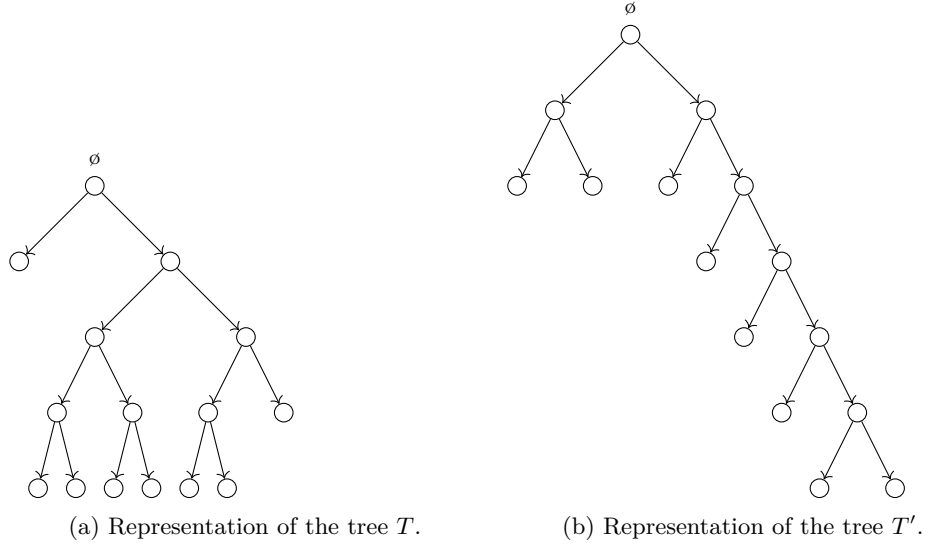

The following proposition is the key to our study of the $H_\alpha$ index,
as it makes it possible to express the $H_\alpha$ index of a network in terms
of the $H_\alpha$ indices of its biconnected components, thereby
opening the door to recursive computations.

\begin{Prop}[Grafting property] \label{graftingproperty}
Let $G_1$ and $G_2$ be two phylogenetic networks. Let $G$ be the network
obtained by grafting $G_2$ on a leaf $\ell$ of $G_1$ (that is, by
identifying $\ell$ with the root of $G_2$).
Then, 
    \begin{equation*}
        H_\alpha(G) = H_\alpha(G_1) + p_{\ell, G_1}^\alpha H_\alpha(G_2),
    \end{equation*}
    where $p_{\ell, G_1}$ denotes the probability that the simple random walk on $G_1$ ends in $\ell$.
\end{Prop}

This proposition is a direct reformulation of the corresponding property for the
structural $\alpha$-entropy
(\cref{graftinggen} in \Cref{secStructuralEntropy});
we thus refer the reader to the appendix for its proof.

In particular, this proposition implies that $H_\alpha$ is a \textit{recursive tree-shape statistic} on the space of binary trees in the sense of F. Matsen in \cite{matsenOptimizationClassTree2007}.
The grafting property also shows that the structure of a phylogenetic network becomes progressively less important as one moves far away from the root. Moreover, the parameter $\alpha$ controls the speed at which this effect occurs : the smaller $\alpha$, the more importance is given to the structure far from the root.
This explains the inversion of the ordering of $T$ and $T'$ observed in
\cref{remaClassement}: indeed, the tree~$T$ is very imbalanced close to the
root and very balanced far away from it; whereas $T'$
is very balanced close to the root but very imbalanced far away from it.

Noting that the $H_\alpha$ index of a cherry (i.e.\ the binary
tree with two leaves) is equal to~1 for all $\alpha \geqslant 0$, we
immediately get the following two useful corollaries of \cref{graftingproperty}.

\begin{Coro} \label{graftingcherry}
Let $G$ be a phylogenetic network and let $G'$ be the network obtained by
grafting a cherry onto a leaf $\ell$ of $G$. Then,
    \begin{equation*}
        H_\alpha(G') = H_\alpha (G) + p_\ell^\alpha. \qedhere
    \end{equation*}
\end{Coro}

\begin{Coro} \label{graftingtwophylogeny}
Let $G'$ and $G''$ be two phylogenetic networks, and let $G$ be the
network obtained by grafting $G'$ and $G''$ onto the two leaves of the cherry,
one on each leaf.
Then, 
\begin{equation*}
    H_\alpha(G)= 1 + 2^{-\alpha} \left( H_\alpha(G') + H_\alpha(G'') \right).
    \qedhere
\end{equation*}
\end{Coro}

We close this section with some examples of computations of $H_\alpha$ for a few
networks of interest.
First, let $\mathrm{CB}(h)$ be the \emph{complete binary tree with height $h$},
i.e.\ the tree depicted in Fig.~\ref{symtree}. Then, letting $n = 2^h$
denote the number of leaves of $\mathrm{CB}(h)$, 
\begin{equation*}
        H_\alpha(\mathrm{CB}(h)) = 
        \frac{1 - n^{1-\alpha}}{1-2^{1-\alpha}}.
    \end{equation*}

Second, let $\Cat(n)$ be the \emph{caterpillar with $n$ leaves} (sometimes also
known as the \emph{comb}), i.e.\ the tree depicted in Fig.~\ref{chenille}.
Since $\Cat(n)$ is obtained by grafting a cherry onto the deepest leaf
of $\Cat(n-1)$, one can use \cref{graftingcherry} to compute the
$H_\alpha$ index recursively.
This yields
\begin{equation*}
    H_\alpha(\Cat(n)) = \frac{1 - 2^{-\alpha (n-1)}}{1-2^{-\alpha}}.
\end{equation*}

Finally, let $\FCat(n)$ be the so-called \textit{fat caterpillar with $n$
leaves} (Fig.~\ref{fcat}).
As before, one can notice that $\FCat(n)$ is obtained by
grafting $\FCat(2)$ onto the $(n-1)$-th leaf of $\FCat(n -1)$.
We can thus
compute its $H_\alpha$ index recursively using \cref{graftingproperty} and
the fact that $H_\alpha(\FCat(2))= \frac{1}{1-2^{1-\alpha}}
\left(1-\frac{3^\alpha +1}{4^\alpha} \right)$; this yields, after
some simplifications:
\begin{equation*} 
H_\alpha(\FCat(n)) = H_\alpha(\FCat(2)) \cdot \left( \frac{1-4^{-\alpha (n-1)}}{1-4^{-\alpha}} \right).
\end{equation*}

\begin{figure}[H]
\centering
\begin{subfigure}[t]{.3\textwidth}
    \begin{tikzpicture}[main/.style = {draw, circle, minimum size = 7pt, inner sep =1pt}]
    \tikzstyle{dot}=[draw,circle,gray,
                  top color= white, text=gray,minimum size=7pt, inner sep = 1]
    \tikzstyle{phantom}=[opacity = 0, minimum size = 7pt]
    \node[main, label ={[font=\footnotesize] \o} ] (root) at (0, 0) {};
    \node[main] (11) at (-1, -1) {};
    \node[main] (12) at (1, -1) {};
    \node[dot] (21) at (-1.25, -2) {};
    \node[dot] (22) at (-0.75, -2) {};
    \node[dot] (23) at (0.75, -2) {};
    \node[dot] (24) at (1.25, -2) {};
    \node[phantom] (p1) at (-1.5, -3) {};
    \node[phantom] (p8) at (1.5, -3) {};
    \node[main, label = {[font=\scriptsize]below:1}] (1) at (-1.75, -4) {};
    \node[main,label = {[font=\scriptsize]below:2}] (2) at (-1, -4) {};
    \node[main,label ={[font=\scriptsize]below:$2^h-1$}] (n-1) at (1, -4) {};
    \node[main,label = {[font=\scriptsize]below:$2^h$}] (n) at (1.75, -4) {};
    \node[phantom] (3) at (-0.5, -4) {};
    \node[phantom] (n-2) at (0.5, -4) {};
    \node[phantom] (p2) at (-1.15, -3) {};
    \node[phantom] (p7) at (1.15, -3) {};
    \node[phantom] (p3) at (-0.8, -3) {};
    \node[phantom] (p6) at (0.8, -3) {};
    \node[phantom] (p4) at (-0.45, -3) {};
    \node[phantom] (p5) at (0.45, -3) {};
    \draw[dashed, ->] (21) -- (p2);
    \draw[dashed, ->] (22) -- (p3);
    \draw[dashed, ->] (22) -- (p4);
    \draw[dashed, ->] (23) -- (p5);
    \draw[dashed, ->] (23) -- (p6);
    \draw[dashed, ->] (24) -- (p7);
    \draw[->] (root) -- (11);
    \draw[->] (root) -- (12);
    \draw[->] (11) -- (21);
    \draw[->] (11) -- (22);
    \draw[->] (12) -- (23);
    \draw[->] (12) -- (24);
    \draw[dashed, ->] (21) -- (p1);
    \draw[dashed, ->] (24) -- (p8);
    \draw[->] (p1) -- (1);
    \draw[->] (p1) -- (2);
    \draw[->] (p8) -- (n-1);
    \draw[->] (p8) -- (n);
    \draw[decorate sep={0.5mm}{1mm},fill] (3) -- (n-2);
\end{tikzpicture}
    \caption{$\mathrm{CB}(h)$, the complete binary tree with $2^h$ leaves.}
    \label{symtree}
\end{subfigure}
~ 
\begin{subfigure}[t]{.3\textwidth}
    \begin{tikzpicture}[main/.style = {draw, circle, minimum size = 7pt, inner sep =1pt}]
    \tikzstyle{dot}=[draw,circle,gray,
                  top color= white, text=gray,minimum size=7pt, inner sep = 1]
    \tikzstyle{phantom}=[opacity = 0, minimum size = 7pt]
    \node[main, label = {[font=\footnotesize]\o}] (root) at (0, 0) {};
    \node[main] (1) at (0.5, -1) {};
    \node[main] (2) at (1, -2) {};
    \node[phantom] (3) at (1.5, -3) {};
    \node[main] (n-1) at (2, -4) {};
    \node[main, label = {[font=\scriptsize]below:$n$}] (n) at (2.5, -5) {};
    \node[main, label = {[font=\scriptsize]below left:1}] (l1) at (-0.5, -1) {};
    \node[main, label = {[font=\scriptsize]below left:2}] (l2) at (0, -2) {};
    \node[main, label = {[font=\scriptsize]below left:3}] (l3) at (0.5, -3) {};
    \node[dot] (4) at (1, -4) {};
    \node[main, label = {[font=\scriptsize]below:$n-1$}] (ln-1) at (1.5, -5) {};
    \draw[->] (root) -- (1);
    \draw[->] (1) -- (2);
    \draw[->] (root) -- (l1);
    \draw[->] (1) -- (l2);
    \draw[->] (2) -- (l3);
    \draw[->] (n-1) -- (ln-1);
    \draw[->] (n-1) -- (n);
    \draw[dashed, ->] (2) -- (n-1);
    \draw[dashed, ->] (3) -- (4);
\end{tikzpicture}
    \caption{$\Cat(n)$, the caterpillar tree with $n$ leaves.}
    \label{chenille}
\end{subfigure}
~
\begin{subfigure}[t]{.3\textwidth}
    \begin{tikzpicture}[main/.style = {draw, circle, minimum size = 7pt, inner sep =1pt}]
    \tikzstyle{dot}=[draw,circle,gray,
                  top color= white, text=gray,minimum size=7pt, inner sep = 1]
    \tikzstyle{phantom}=[opacity = 0, minimum size = 7pt]
    \node[main, label = {[font=\footnotesize]above left:\o}] (root) at (0, 0) {};
    \node[main] (b1) at (1, -0.25) {};
    \node[main] (b2) at (1.75, -0.5) {};
    \node[main] (b3) at (2.75, -0.75) {};
    \node[main] (b4) at (3.75, -1) {};
    \node[main] (b5) at (4.75, -1.25) {};
    \node[main, label = {[font=\scriptsize]below:$n$}] (n) at (5.75, -1.5) {};
    \node[main] (sb1) at (0.5, -1.5) {};
    \node[main] (sb2) at (2, -1.75) {};
    \node[main] (sb3) at (4, -2) {};
    \node[main, label = {[font=\scriptsize]below:1}] (1) at (0, -2.25) {};
    \node[main,label = {[font=\scriptsize]below:2}] (2) at (1.5, -2.5) {};
    \node[main,label = {[font=\scriptsize]below:$n-1$}] (n-1) at (3.5, -2.75) {};
    \draw[->] (root) -- (b1);
    \draw[->] (root) -- (sb1);
    \draw[->] (sb1) -- (1);
    \draw[->] (b1) -- (sb1);
    \draw[->] (b1) -- (b2);
    \draw[->] (b2) -- (b3);
    \draw[dashed] (b3) -- (b4);
    \draw[->] (b4) -- (b5);
    \draw[->] (b5) -- (n);
    \draw[->] (b2) -- (sb2);
    \draw[->] (b3) -- (sb2);
    \draw[->] (b4) -- (sb3);
    \draw[->] (b5) -- (sb3);
    \draw[->] (sb2) -- (2);
    \draw[->] (sb3) -- (n-1);
    \node[phantom] (p1) at (2.25, -2.175) {};
    \node[phantom] (p2) at (3.25, -2.275) {};
    \draw[decorate sep={0.5mm}{1mm},fill] (p1) -- (p2);
\end{tikzpicture}
    \caption{$\FCat(n)$, the fat caterpillar with $n$ leaves.}
    \label{fcat}
\end{subfigure}
\caption{Examples of phylogenetic trees and networks of special interest.}
\end{figure}
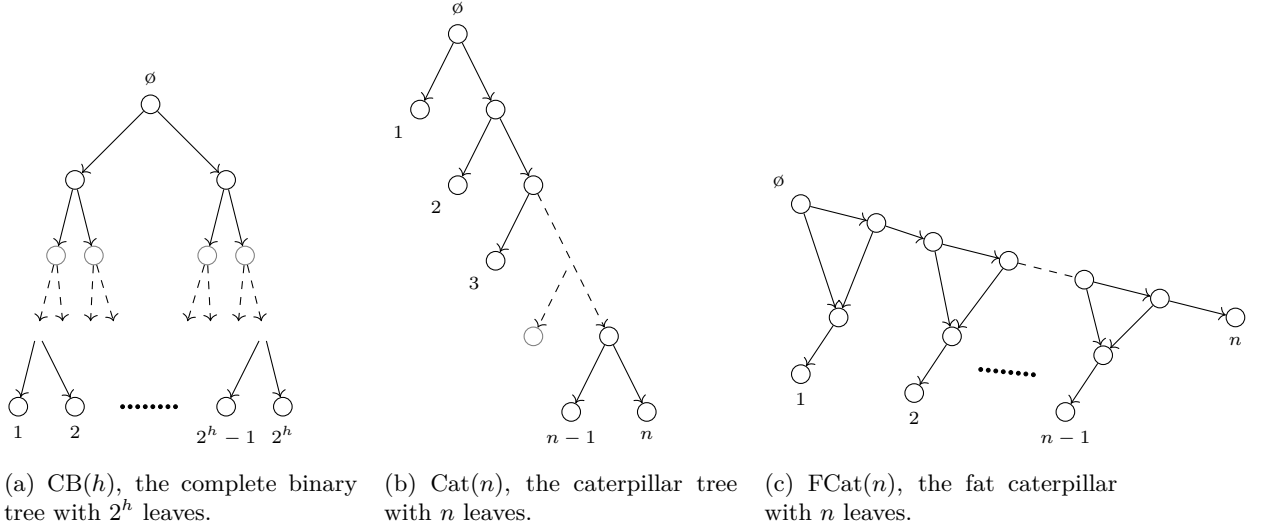

We close this section with a useful formula for the $H_\alpha$ index of binary
trees.

\begin{Prop} \label{HalphaBinaire}
For any binary tree $T$, 
\begin{equation*}
  H_\alpha(T) \;=\; 2^{\alpha -1}\!\!\!\!
  \sum_{v\in T\setminus\{ \rho\}}\!\!\!\! 2^{-\alpha \delta_v}
    \;=\; 2^{\alpha -1} \sum_{k \geqslant 1} 2^{-\alpha k} Z_k(T), 
\end{equation*}
where the first sum runs on all vertices of $T$ except its root;
$\delta_v$ denotes the depth of $v$, i.e.\ its distance to the root; and
$Z_k(T) = \#\{v \in T : \delta_v = k\}$.
\end{Prop}

\begin{Proo}
Let us proceed by induction on the height of $T$. If $T$ has
height~0, i.e.\ if $T$ is reduced to its root, then 
$H_\alpha(T) = 0 = \sum_{\emptyset}$.
Assume now that the result holds for any tree with height at most $n$, and
consider a binary tree $T$ with height $n+1$.
Since $T$ is binary, it is a cherry with a tree of height at most $n$ attached
to each leaf.
Denote these trees by $T_1$ and $T_2$.
By \cref{graftingproperty}, 
\begin{equation*}
    H_\alpha(T)= 1 + 2^{-\alpha} \big( H_\alpha(T_1) + H_\alpha(T_2) \big).
\end{equation*}
So, by the induction hypothesis,
\begin{align*}
    H_\alpha(T) &= 1 + 2^{-\alpha}
    \bigg( 2^{\alpha -1} \sum_{k \geqslant 1} 2^{-\alpha k}
    \big(Z_k(T_1) + Z_k(T_2)\big) \bigg) \\
    &= 1 + 2^{\alpha -1} \sum_{k \geqslant 1} 2^{-\alpha(k+1)} Z_{k+1}(T) \\
    &= 2^{\alpha -1} \sum_{k \geqslant 1} 2^{-\alpha k} Z_k(T),
\end{align*}
where the second equality comes from the fact that vertices at
depth $k+1$ in $T$ are exactly vertices at depth $k$ in $T_1$ or $T_2$, 
and the last equality comes from the fact that $Z_1(T) = 2$.
The result follows by induction for any binary tree with height at most
$n$; and by taking the limit $n\to\infty$ for infinite binary trees
(see \Cref{Stubsequence} in the Appendices).
\end{Proo}

\begin{Rema} \label{rem-d-ary-trees}
For $d$-ary trees (i.e.\ trees where each non-leaf vertex has exactly
$d$ children), one has the following generalization of \Cref{HalphaBinaire}:
\begin{equation*}
    H_\alpha(T) \;=\;
    \frac{d^{\alpha -1}-1}{1-2^{1-\alpha}} \!\!\!
    \sum_{v\in T\setminus\{ \rho\}}\!\!\!\! d^{-\alpha \delta_v}
    \;=\; \frac{d^{\alpha -1}-1}{1-2^{1-\alpha}}
    \sum_{k \geqslant 1} d^{-\alpha k}Z_k(T). \qedhere
\end{equation*}
\end{Rema}

\subsection[Range of the \texorpdfstring{$H_\alpha$}{H-alpha} index]{%
Range of the \texorpdfstring{$\boldsymbol{H_\alpha}$}{H-alpha} index} \label{secRange}

In this section, we derive bounds for the $H_\alpha$ indices over various 
classes of phylogenetic networks and, when possible, 
characterize the networks that attain those bounds.
Since the proofs of our results are straightforward adaptations of the
corresponding propositions in \cite{bienvenuRevisitingShaoSokals2021}, we
only state these bounds and refer the reader to \Cref{appRange} for the
proofs.

\begin{Prop} \label{generalBounds}
Let $\mathcal{G}_n$ be the class of phylogenetic networks with $n$ leaves. Then,
for $\alpha > 0$,
\begin{equation*}
    \inf_{G \in \mathcal{G}_n} H_\alpha (G) = 0
    \qquad \text{and} \qquad
    \max_{G \in \mathcal{G}_n} H_\alpha (G) =
    \frac{1-n^{1-\alpha}}{1-2^{1-\alpha}}.
\end{equation*}
Moreover, a network $G$ maximizes $H_\alpha$ if and only if
$p_\ell = p_{\ell'}$ for all leaves $\ell$ and $\ell '$ of $G$.
\end{Prop}

In practice, the class of all phylogenetic networks is often too broad to be
relevant.
We therefore provide results for two particularly relevant subclasses
of phylogenetic networks: binary trees and temporal tree-child networks.

\begin{Theo} \label{binaryBound}
    Let $T$ be a rooted binary tree with $n$ leaves. Then,
    \begin{equation*}
        H_\alpha(\Cat(n)) \leqslant H_\alpha(T) \leqslant H_\alpha(\mathrm{CB}(\lfloor \log_2(n) \rfloor)) + \left(n-2^{\lfloor \log_2(n)\rfloor} \right) \cdot 2^{-\alpha \lfloor \log_2(n) \rfloor},
    \end{equation*}
    where we recall that $\Cat(n)$ denotes the caterpillar tree with $n$ leaves
    and $\mathrm{CB}(h)$ denotes the complete binary tree with height $h$.
    Moreover, for $\alpha > 0$ these bounds are sharp and:
    \begin{enumerate}[label=(\roman*)]
        \item The caterpillar tree $\Cat(n)$ is the only rooted binary tree
          with $n$ leaves that minimizes $H_\alpha$.
        \item The rooted binary trees with $n$ leaves that maximize
          $H_\alpha$ are exactly the trees such that the difference
          between the height of any two leaves is at most $1$. \qedhere
    \end{enumerate}
\end{Theo}

In particular, \cref{binaryBound} shows that
the $H_\alpha$ indices are indeed balance indices in the sense of
\cite[Def.~2]{fischerTreeBalanceIndices2023}.

\begin{Theo} \label{boundsTemporal}
Assume that $\alpha \leqslant 1$.  
For every temporal tree-child network $N$ with $n$ leaves,
\begin{equation*}
        H_\alpha(\Cat(n)) \leqslant H_\alpha(N) \leqslant
        H_\alpha(\mathrm{CB}(\lfloor \log_2(n) \rfloor)) 
        + \left(n-2^{\lfloor \log_2(n)\rfloor} \right) \cdot 
        2^{-\alpha \lfloor \log_2(n) \rfloor}. \qedhere
\end{equation*}
\end{Theo}

\begin{Rema}
As for $\alpha > 1$, note that, by Proposition~\ref{generalBounds},
$H_\alpha(N)$ is bounded above by $(1-n^{1-\alpha})/(1-2^{1-\alpha})$, which
is uniformly bounded in $n$ when $\alpha > 1$.
\end{Rema}

\begin{Theo} \label{boundTreeChild}
Let $N$ be a tree-child network with $n$ leaves. Then,
\begin{equation*}
H_\alpha(N) \geqslant H_\alpha(\FCat(n)),
\end{equation*}
where $\FCat(n)$ is the fat caterpillar with $n$ leaves (see
Section~\ref{basicProperties}).
\end{Theo}

\section{The \texorpdfstring{$\boldsymbol{H_\alpha}$}{H-alpha} index of various models of random trees}

In this section, we study the distribution $H_\alpha$ (mostly through its first
moments) for two broad models of random trees: binary Markov branching trees
and Galton--Watson trees.
We then specialize these results to two phylogenetically important examples of
binary Markov branching trees, namely the Yule / ERM model and the PDA model.

\subsection{Binary Markov branching trees}

Binary Markov branching trees are a class of models of random trees
introduced by Aldous in \cite{aldousProbabilityDistributionsCladograms1996a}.
Since then, they have become classical in phylogenetics --
in great part thanks to the $\beta$-splitting model, a parameterized
family of Markov branching trees introduced by Aldous in that
same article that includes the Yule / ERM model and the
PDA model (for $\beta = 0$ and $\beta = -3/2$, respectively).

A Markov branching tree model is parameterized  by a familly
$\mathbf{q} = (q_n)_{n \geqslant 2}$
of probability distributions  such that $q_n$ is supported on
$\{1, \cdots, n-1\}$ and is symmetric
(meaning that $q_n(i) = q_n(n-i)$ for all~$i$).
These distributions are known as the \emph{root-split distributions}
of the model.
The model produces an ordered binary tree with $n$ leaves -- say $1, \ldots, n$
-- as follows:
start by choosing a random integer~$i$
according to $q_n$, then choose $i$ leaves uniformly at random from
$\{1, \ldots, n\}$.
Let the root have descendant subtrees (left and right) whose leaf sets are
the chosen $i$ leaves and the remaining $n-i$ leaves, respectively.  Repeat
this operation independently in the two subtrees, until every subtree has one
leaf.

Throughout this section, $T_n$ will denote a binary Markov branching tree with
$n$ leaves and root-split distributions~$\mathbf{q}$.
Let $(T'_k)_{k\geqslant 1}$ and $(T''_k)_{k\geqslant 1}$ be two independent
families of Markov branching trees with root-split
distribution $\mathbf{q}$, such that $T'_k$ and $T''_k$ have $k$ leaves,
and let $K_n$ be a random variable distributed according to $q_n$ and
independent of $(T'_k)_{k\geqslant 1}$ and $(T''_k)_{k\geqslant 1}$ .
Then, by definition, 
\begin{equation*}
  T_n \;\overset{d}{=}\; T'_{K_n} \oplus T''_{n-K_n},
\end{equation*}
where $\tau_1 \oplus \tau_2$ denotes the tree obtained
by grafting the trees $\tau_1$ and $\tau_2$ onto a cherry, i.e.\ by identifying
their roots with one leaf of the cherry each.
\Cref{graftingcherry} thus implies
\begin{equation} \label{graftingmarkov}
  H_\alpha (T_n) \;\overset{d}{=}\;
  1 + 2^{-\alpha} (H_\alpha(T_{K_n}') + H_\alpha(T_{n-K_n}'')).
\end{equation}
The independence of $K_n$, $(T'_k)_{k\geqslant 1}$ and $(T''_k)_{k\geqslant 1}$
and the fact that $K_n$ is distributed as $n-K_n$ readily imply
the following recurrence relations for the first and second moments of
$H_\alpha(T_n)$.

\begin{Theo} \label{reccurence}
Let $T_n$ be a Markov branching tree with $n$ leaves and root-split
distribution $\mathbf{q} = (q_n)$. Let $\mu_n = \mathds{E}[H_\alpha(T_n)]$,
$s_n = \mathds{E}[H_\alpha(T_n)^2]$ and $v_n = \mathrm{Var}(H_\alpha(T_n))$.
Then letting $K_n \sim q_n$, we have the following recurrence relations:
\begin{enumerate}[label=(\roman*)]
\item $\mu_n = 1 + 2^{1-\alpha} \mathds{E}[\mu_{K_n}]$,
\item $s_n = 1 + 2^{1-2\alpha} \left( \mathds{E}[s_{K_n}] + 
  \mathds{E}[\mu_{K_n} \mu_{n-K_n}] \right) +2^{2-\alpha} \E[\mu_{K_n}]$,
\item $v_n = 4^{-\alpha} \left(\mathrm{Var}(\mu_{K_n} +
  \mu_{n-K_n}) + 2 \mathds{E}[v_{K_n}] \right)$. \qedhere
\end{enumerate}
\end{Theo}

\begin{Rema}
Point (i) in \cref{reccurence} can also be written as
\begin{equation*}
    \mu_n = 1 + b \cdot \sum_{k=1}^{n-1}q_n(k)\, \mu_k,
\end{equation*}
where $b = 2^{1-\alpha}$ and $\mu_1 = 0$.
This recursion can be iterated to yield
\begin{equation*}
    \mu_n = \sum_{j=0}^{n-2} c_{j, n} b^j,
\end{equation*}
where the coefficients $c_{j,n}$ are defined by $c_{0, n} = 1$ and
\begin{equation}
    c_{j,n} = 
    \sum_{k=j+1}^{n-1} q_n(k)\, c_{j-1, k} =
    \sum_{k_1 = j+1}^{n-1}q_n(k_1) \sum_{k_2 = j}^{k_1-1}q_{k_1}(k_2)
    \, \cdots\!\!
    \sum_{k_j = 2}^{k_{j-1}-1} q_{k_{j-1}}(k_j). \label{expressionofCjn}
\end{equation}
The coefficient $c_{j, n}$ has a simple interpretation as
the probability that the leftmost leaf in the tree has height $j+1$.
Indeed, by construction, the left subtree must have $k \geqslant j+1$ leaves
for one of them to be at height $j+1$; and then the left subtree of that
subtree with $k$ leaves must have its leftmost leaf at height $j$.
\end{Rema}

%\begin{Rema}
%    \begin{itemize}
%        \item It is easily checked that $c_{1, n} = 1-q_n(1)$ and since when
%          $j=n-2$ only one term remains in each sum in \cref{expressionofCjn}
%          (namely $k_1 = n-1,~ k_2=n-2, \cdots , k_{n-2} = 2$), one also has
%          $c_{n-2, n} = \displaystyle \prod_{k=2}^{n-1}q_{k+1}(k)$.
%        \item For instance, for Aldous model (i.e.\ the $\beta$-splitting model
%          when $\beta = -1$) one has
%        \begin{equation*}
%            c_{1, n} = 1- \frac{n}{2(n-1)h_{n-1}} \qquad
%            \text{and by telescoping} \qquad c_{n-2, n} =
%            \frac{n}{2^{n-1}\prod_{k=2}^{n-1}h_k},
%        \end{equation*}
%        where $h_k = \sum_{i = 1}^k \frac{1}{i}$ is the truncated harmonic
%        series. Though, the other coefficients doesn't seem to have such simple
%        expressions for Aldous model.
%    \end{itemize}
%\end{Rema}

Since the Yule / ERM and PDA models are special instances of Markov branching
trees, \cref{reccurence} can be used to compute the expectation and variance of
$H_\alpha$ for these models.
However, in these specific cases one can also use other,
less computation-oriented methods, as explained in \Cref{secYule,secPDA}.

\subsection{Galton--Watson trees} \label{secGW}

In this section, we provide necessary and sufficient conditions for the
existence of integer moments of the $H_\alpha$ index of Galton--Watson trees.
We also give explicit expressions for the expected value and variance.

Let us start by setting some notation.
For all $n \in \mathbb{N}$, let $\xi_n$ be an integer-valued random variable % denote by $\mathcal{D}_n$ its distribution,
and let $(\xi_n(i))_i$ be a family of independent copies of $\xi_n$.
The \emph{time-inhomogeneous Galton--Watson process with offspring distribution $\xi_n$ % $\mathcal{D}_n$
at generation $n$} is the stochastic process $(Z_{n})_{n\geqslant 0}$ defined by
\begin{equation*}
    \begin{cases}
    Z_0 = 1 \\
    Z_{n+1} = \displaystyle \sum_{i = 1}^{Z_n} \xi_n(i).
    \end{cases}
\end{equation*}
We denote by $T$ the random tree describing the genealogy associated
with this process.

\begin{Defi} \label{defTruncation}
Let $G$ be a phylogenetic network. 
The \emph{height} of a vertex $v$ is the number of edges of a shortest
directed path from the root of $G$ to $v$.
We refer to the subnetwork of $G$ induced by its vertices with height at most
 $k$ as the \emph{truncation at height $k$ of $G$} and we denote it by
$[G]_k$.
\end{Defi}

\begin{Prop} \label{inhomoGW}
Let $T$ be a time-inhomogeneous Galton--Watson tree with offspring distribution
$\xi_n$ at generation $n$.
Define
\begin{equation*}
    \kappa_{\alpha, n} = \E \left[ \xi_n^{1-\alpha}
    \mathds{1}_{\{\xi_n \neq 0 \}} \right]
    \qquad \text{and} \qquad
    \eta_{\alpha, n} =
    \frac{1}{1-2^{1-\alpha}} \E \left[ 
    (1-\xi^{1-\alpha}_n)\1_{\{\xi_n \neq 0\}}\right].
\end{equation*}
Then, 
\begin{equation*}
    \mathds{E} [H_\alpha([T]_k)] = \sum_{j=0}^{k-1} \eta_{\alpha, j}
    \prod_{m=0}^{j-1} \kappa_{\alpha, m}.\qedhere
\end{equation*}
\end{Prop}

\begin{Proo}
We can construct a Galton--Watson tree
truncated at generation $k$ by the following procedure: first,
sample $\xi_0$. If $\xi_0 = 0$, then the tree is reduced
to its root and its $H_\alpha$ index is 0.
If $\xi_0 > 0$, then build a star tree with $\xi_0$ leaves and
sample $\xi_0$ independent Galton--Watson trees truncated at generation $k-1$.
Finally, graft those trees to the leaves of the star tree (one on each leaf).
Noting that the $H_\alpha$ index of a star tree with $n$
leaves is $\frac{1}{1-2^{1-\alpha}}(1-n^{1-\alpha})$,
\cref{graftingproperty} yields the following equality in distribution:
\begin{equation} \label{recurrenceGWdistri}
    \left( H_\alpha([T]_k)~|~ \xi_0 = i \right) \overset{d}{=} 
    \left( \frac{1}{1-2^{1-\alpha}}(1-i^{1-\alpha}) + i^{-\alpha}
    \sum_{j = 1}^{i} H_\alpha([\tilde{T}(j)]_{k-1}) \right),
\end{equation}
where $i$ is any positive integer and $(\tilde{T}(j))_{j \geqslant 1}$ are independent Galton--Watson trees
with offspring distributions
$(\xi_{n+1})_{n \geqslant 0}$ that
are also independent of $\xi_0$.
Thus, conditioning on $\xi_0$, taking expectation and using the independence of
$\xi_0$ and $(\tilde{T}(j))_{j \geqslant 1}$, one has
\begin{equation}
    \E[H_\alpha([T]_k)] = \frac{1}{1-2^{1-\alpha}} \E[(1-\xi_0^{1-\alpha})\1_{\{\xi_0 \neq 0\}}] + \E[\xi_0^{1-\alpha}\1_{\{\xi_0 \neq 0\}}] \cdot \E[H_\alpha([\tilde{T}]_{k-1})]. \label{recurrenceGW}
\end{equation}
Since $\E[H_\alpha([T]_0)] = 0$, we conclude by induction.
\end{Proo}

From now on, we only consider time-homogeneous Galton--Watson trees --
meaning that the random variables $\xi_n$ are identically distributed as $\xi$
for all $n\geqslant 0$.

\begin{Coro} \label{expecGW}
Let $T$ be a Galton--Watson tree whose offspring distribution $\xi$ 
satisfies $\mathds{P}(\xi = 1) < 1$.
Then,  $\mathds{E}[H_\alpha(T)]$ is finite if and only if
$\mathds{E}[\xi^{1-\alpha}\mathds{1}_{\{\xi \neq 0\}}] < 1$. In this case,
\begin{equation*}
  \mathds{E}[H_\alpha(T)] =
  \frac{1}{1-2^{1-\alpha}}\cdot
  \frac{ \mathds{E}[(1-\xi^{1-\alpha})
  \mathds{1}_{\{\xi \neq 0\}}]}{1-\mathds{E}[\xi^{1-\alpha}
  \mathds{1}_{\{\xi \neq 0\}}]}.\qedhere
\end{equation*}
\end{Coro}

\begin{Rema}
Note that if $\alpha$ is greater than $1$, then we always have
$\mathds{E}[\xi^{1-\alpha}\mathds{1}_{\{\xi \neq 0\}}] < 1$ because $\xi$
is $\mathbb{N}$-valued and we have excluded the trivial case 
$\mathds{P}(\xi = 1) = 1$ (in which case $T$ is an infinite
straight line and $\E[H_\alpha(T)] = 0$).
This is consistent with the fact that for $\alpha > 1$ the $H_\alpha$ index is
bounded above by $\frac{1}{1-2^{1-\alpha}}$.
\end{Rema}

\begin{Proo}
Let $u_k = \E[H_\alpha([T]_k)]$, and note that the recursion
Eq.~\eqref{recurrenceGW} is of the form
\begin{equation*}
  u_k =  a + b\,u_{k-1}, \quad
  \text{where }
    \begin{dcases}
    a =  \frac{1}{1-2^{1-\alpha}} \E[(1-\xi^{1-\alpha}) \1_{\{\xi \neq0\}}]  \\
    b = \E[\xi^{1-\alpha}\1_{\{\xi \neq 0\}}].
    \end{dcases}
\end{equation*}
Thus, the sequence $(u_k)$ converges if and only if
$b = \mathds{E}[\xi^{1-\alpha}\mathds{1}_{\{\xi \neq 0\}}] < 1$, in
which case its limit is the expression given in the corollary.
To conclude the proof, it remains to show that
$\mathds{E}[H_\alpha([T]_k)] \to \mathds{E}[H_\alpha(T)]$.
Anticipating our results on local limits, this follows from
\Cref{Stubsequence};
alternatively, this can be proved from \Cref{defentropy02}
(note that, in particular, \Cref{monotonicity} implies that
the sequence $(\mathds{E}[H_\alpha([T]_k)])_{k\geqslant 0}$ is nondecreasing).
\end{Proo}

\begin{Theo} \label{VarGW}
Let $T$ be a Galton--Watson tree with offspring distribution $\xi$, where
$\mathds{P}(\xi = 1) < 1$.
Then, $\Var(H_\alpha(T))$ is finite if and only if
$\mathds{E}[\xi^{1-\alpha}\mathds{1}_{\{\xi \neq 0\}}] < 1$ and
$\mathds{E}[\xi^{2(1-\alpha)}\mathds{1}_{\{\xi \neq 0\}}] < + \infty$.
In this case, 
\begin{equation*}
    \mathrm{Var}(H_\alpha(T)) = 
    \frac{\mathds{P}(\xi = 0)}{(1-2^{1-\alpha})^2
    \left(1-\mathds{E}\left[ \xi^{1-2\alpha}
    \mathds{1}_{\{\xi \neq 0\}} \right] \right)}
    \left(1+ \frac{\left( \mathds{E}\left[ \xi^{2(1-\alpha)}
    \mathds{1}_{\{\xi \neq 0\}} \right] -1 \right)
    \mathds{P}(\xi = 0)}{\left(1- \mathds{E}\left[ \xi^{1-\alpha}
    \mathds{1}_{\{\xi \neq 0\}} \right] \right)^2} \right). \qedhere
\end{equation*}
\end{Theo}

The proof of this result consists in squaring both sides of
Eq.~\eqref{recurrenceGWdistri}, taking expectations and performing
routine computations to solve the resulting recursion.
The details can be found in \Cref{appGW}. 
In fact, the same method can be used to study higher moments;
however the computations quickly become tedious, so for $m\geq 3$
we content ourselves with a necessary and sufficient condition for the
existence of the $m$-th moment of $H_\alpha(T)$.
Again, the details are in \Cref{appGW}.

\begin{Theo} \label{momentoforderm}
Let $T$ be a Galton--Watson tree with offspring distribution $\xi$, where
$\mathds{P}(\xi = 1) < 1$.
Then, for any positive integer $m$,
$\mathds{E}[(H_\alpha(T))^m] < +\infty$ if and only if
$\mathds{E}[\xi^{(1-\alpha)m}\mathds{1}_{\{\xi \neq 0\}}] < + \infty$
and $\mathds{E}[\xi^{1-\alpha}\mathds{1}_{\{\xi \neq 0 \}}] < 1$.
\end{Theo}

We now give a simple characterization of the existence of
exponential moments of \emph{binary} Galton--Watson trees.
Because its proof is slightly technical, it is again relegated to \Cref{appGW}.
Note that, as the proof relies on the formula for binary trees given in
\Cref{HalphaBinaire}, we do not see an immediate way to obtain
similar results for arbitrary Galton--Watson trees.

\begin{Prop} \label{exponentialMoment}
Let $T$ be a Galton--Watson tree with offspring distribution
$\xi \sim 2\,\mathrm{Ber}(p)$. The following are equivalent:
\begin{enumerate}[label=(\roman*)]
  \item $\E[H_\alpha(T)] < +\infty$.
  \item $p < 2^{\alpha -1}$.
  \item $H_\alpha(T)$ has exponential moments. \qedhere
\end{enumerate}
\end{Prop}

Finally, we close this section by comparing the $H_\alpha$ index of two
Galton--Watson trees $T$ and $T'$ such that vertices ``tend to have more
children'' in $T'$ than in $T$.
As we will see, the comparison is subtle, and the function that maps an
offspring distribution $\xi$ to the distribution of the $H_\alpha$ index of
$T\sim \mathrm{GW}(\xi)$ is not monotone in the usual sense.

\begin{Defi}
Let $X$ and $Y$ be two nonnegative random variables. We say that:
\begin{itemize}
  \item \emph{$X$ is stochastically dominated by $Y$}, which we denote
    by $X \preccurlyeq Y$, if
    $\mathds{P}(X > t) \leqslant \mathds{P}(Y > t)$
    for all $t$.
    Equivalently, if
    $\E[f(X)] \leqslant \E[f(Y)]$ for every nondecreasing
    function $f$.  
  \item \emph{$X$ is smaller than $Y$ in the Laplace transform order}, 
    which we denote by $X \preccurlyeq_{\mathrm{LT}} Y$, if
    for all $t \geqslant 0$,
    $\E[\exp(-t Y)] \leqslant \E[\exp(-t X)]$ or equivalently, if
    $\E[f(Y)] \leqslant \E[f(X)]$ for every completely
    monotone function $f$. \qedhere
\end{itemize}
\end{Defi}

Stochastic domination is also known as the \emph{usual stochastic order}.
It is a very strong notion of order, with an intuitive interpretation:
indeed, $X \preccurlyeq Y$ if and only if it is possible to couple
$X$ and $Y$ such that $X\leqslant Y$ almost surely.
In comparison, the Laplace transform order -- which is also one of the
classic stochastic orders -- is a much weaker order; see
\cite{ShakedShanthikumar2007} for more on this.

\begin{Prop}
Let $T$ and $T'$ be two Galton--Watson trees with offspring distribution
$\xi$ and $\xi'$, respectively. Then, for all $\alpha \in \;]0; 1[$, 
\[
  \xi \preccurlyeq \xi'
  \implies
  H_\alpha(T) \preccurlyeq_{\mathrm{LT}} H_\alpha(T'). \qedhere
\]
\end{Prop}

\begin{Proo}
Let $T$ be a Galton--Watson tree with offspring distribution $\xi$, and as in
Definition~\ref{defTruncation}
let $[T]_k$ denote the truncation of $T$ at height~$k$
(thus, $[T]_0$ is reduced to its root).
Letting $X_k = H_\alpha([T]_k)$, as we have already seen
$X_k \uparrow H_\alpha(T)$ as $k\to\infty$ and
$(X_k)_{k\geqslant 0}$ satisfies the recursive distributional equation
given by Eq.~\eqref{recurrenceGWdistri}. In particular, for $\xi$ independent of everything else,
\begin{align} \label{eqRec}
  X_k  \;\overset{d}{=}\;
  f(\xi) + \xi^{-\alpha} \sum_{i = 1}^\xi X_{k-1}^{(i)}, 
  \quad \text{ where }
  f\colon i \mapsto \frac{1 - i^{1-\alpha}}{1 - 2^{1-\alpha}}\,\1_{\{i > 0\}},
\end{align}
$X_{k-1}^{(i)}$, $i \geqslant 0$ are i.i.d.\ copies of $X_{k-1}$ and the usual convention $\xi^{-\alpha} \sum_{i=1}^\xi X_{k-1}^{(i)} = 0$ if $\xi = 0$.
With the initial condition $X_0 = 0$, this equation uniquely
determines the distribution of $(X_k)_{k\geqslant 0}$ -- and therefore of
$H_\alpha(T)$.
Let now $L_k \colon t \mapsto \E[\exp(-t X_k)]$ denote the Laplace 
transform of $X_k$. With this notation, Eq.~\eqref{eqRec} is equivalent to
\begin{equation} \label{eqOperator}
  L_{k+1} = \mathcal{T}  L_k\,, 
  \quad\text{where }\;
  \mathcal{T} \colon L \;\longmapsto\;
  \big(t\mapsto \E[e^{-t f(\xi)}L(t \xi^{-\alpha}\1_{\{\xi > 0\}})^\xi]\big),
\end{equation}
is defined on the space of Laplace transforms. Thus, from now on, $L$ will denote an arbitrary Laplace transform and so, with $\psi = -\log L$, the operator $\mathcal{T}$ can also be written
\[
  \mathcal{T}\colon L \;\longmapsto\; \big(t\mapsto \E[e^{- h_t(\xi)}]\big),
  \quad\text{where }
  h_t(x) = t f(x) +  x \psi(t x^{-\alpha}) ,
\]
with $h_t(0) = 0$ for all $t > 0$.

Let us show that $x \mapsto \exp(-h_t(x))$ is nonincreasing, i.e.\
that $h_t$ is nondecreasing: first, note that $f$ is nondecreasing.
Second, recall that, as a Laplace exponent, $\psi$ is
nondecreasing, concave, nonnegative and $\mathcal{C}^\infty$.
As a result, for all $u, v > 0$,
\[
  \psi(v) \leqslant \psi(u) + \psi'(u) (v-u).
\]
Taking $v = 0$ and noting that $\psi(0) = 0$, this yields:
$\forall u > 0$, $u\,\psi'(u)\leqslant \psi(u)$.
Therefore, letting $A_t\colon x \mapsto x\,\psi(t x^{-\alpha})$, since
$0 < \alpha < 1$ we have, for all $t > 0$, 
\[
  A_t'(x) = \psi(tx^{-\alpha}) - \alpha t x^{-\alpha} \psi'(t x^{-\alpha})
  \geqslant \psi(tx^{-\alpha}) - t x^{-\alpha} \psi'(t x^{-\alpha})
  \geqslant 0.
\]
This shows that $A_t$ is nondecreasing, concluding the proof
of the fact $x \mapsto \exp(-h_t(x))$ is nondecreasing.

Now, let $\xi$ and $\xi'$ be two offspring distributions such that
$\xi \preccurlyeq \xi'$, and let
\[
  \mathcal{T} \colon L \;\longmapsto\;
  \big(t\mapsto \E[e^{-h_t(\xi)}]\big)
  \quad\text{and}\quad
  \mathcal{T}' \colon L \;\longmapsto\;
  \big(t \mapsto \E[e^{-h_t(\xi')}]\big) \,.
\]
First, note that $\mathcal{T}'$ and $\mathcal{T}$ are nondecreasing.
Moreover, since $x \mapsto \exp(-h_t(x))$ is nonincreasing and $\xi \preccurlyeq \xi'$,
we have $\mathcal{T}'\leqslant \mathcal{T}$, in the sense that, for
all $L$ and $t$, 
\begin{equation}
    (\mathcal{T'}L)(t) \;\leqslant\;
  (\mathcal{T}L)(t) \,. \label{eqInequalityOperator}
\end{equation}
By induction, using the nondecrease of $\mathcal{T'}$ and \cref{eqInequalityOperator}, we get
\[
  \forall t\geqslant 0, \quad
  L'_k\;=\; \mathcal{T}'^{(k)}(\bar{1}) \;\leqslant\;
  \mathcal{T}^{(k)}(\bar{1}) \;= L_k\; \,,
\]
where $\bar{1} \colon x \mapsto 1$ denotes the Laplace
transform of $X_0 = X'_0 = 0$, and $L_k$ (resp.\ $L'_k$) denote the Laplace
transforms of $H_\alpha([T]_k)$ (resp.\ $H_\alpha([T']_k)$).
The proposition follows by taking the limit $k\to\infty$ in this inequality.
\end{Proo}

\begin{Rema}
Note that for any nonnegative random variables,
$X\preccurlyeq_{\mathrm{LT}} Y \implies \E[X] \leqslant \E[Y]$.
As a result, $\xi \preccurlyeq \xi' \implies
\E[H_\alpha(T)] \leqslant \E[H_\alpha(T')]$.
However, this comparison does not extend to higher moments, because
$x \mapsto x^p$ is not completely monotone for $p > 1$.
\end{Rema}

\subsection{The Yule / ERM model} \label{secYule}

In this section, we study the $H_\alpha$ index of a tree $T_n$ with $n$
leaves sampled according to the Yule model (also known as the ERM model).
We provide explicit expressions for the expectation and variance of
$H_\alpha(T_n)$ and describe its asymptotic behavior as $n\to\infty$.

The Yule model is the pure-birth process defined as follows:
starting from one lineage, each lineage lives for an exponential time with
mean~1, independently of the others, and then splits into two lineages.
The Yule tree with $n$ leaves is the tree describing
the genealogy of the lineages when we stop the process upon reaching $n$
lineages \cite{Yule1925,Harris1963}.
With this construction, the Yule tree has a natural embedding in time, i.e.\
branch lengths.
If we discard those lengths and see the tree as a purely combinatorial object,
a nested sequence $(T_n)_{n \geqslant 1}$ of Yule trees can be obtained
iteratively as follows: $T_1$ is the tree reduced to a single vertex, and $T_n$
is obtained from $T_{n-1}$ by sampling one of its leaves uniformly at random and
grafting a cherry onto it.
Also note that Yule trees are Markov branching trees with uniform root-split
distribution \cite{Harding1971,aldousStochasticModelsDescriptive2001};
in this context, the corresponding model is often referred to as the ERM model.
However, we will not use this viewpoint here.

The following proposition introduces a martingale on which the next results
are based.

\begin{Prop}
Let $(T_n)_{n\geqslant 1}$ be a sequence of Yule trees built by iterated
cherry grafting, as described above, and let
$\mathcal{F} = (\mathcal{F}_n)_{n\geqslant 1}$ be the associated natural
filtration -- that is, $\mathcal{F}_n = \sigma(T_n)$.
Let $X_{n, \alpha} = \sum_{\ell \in \mathcal{L}_{T_n}} p_\ell^\alpha $ and
\begin{equation*}
  Y_{n, \alpha} =
  X_{n, \alpha}\, \prod_{k=1}^n \left(1 + \frac{2^{1-\alpha} - 1}{k}\right)^{-1}
\end{equation*}
Then, $(Y_{n, \alpha})_{n \geqslant 1}$ is a $\mathcal{F}$-martingale.
\end{Prop}

\begin{Proo}
It suffices to note that, by construction of the sequence
$(T_n)_{n \geqslant 1}$, letting $L$ denote the leaf of~$T_n$ onto which
a cherry is grafted to obtain $T_{n+1}$ and
$\mathcal{L}_n^* = \mathcal{L}_{T_n} \setminus \{L\}$ the set
of all other leaves of~$T_n$, we have
\begin{align}
    \E[X_{n+1, \alpha} ~|~ T_n] &= 
    \E \left[\sum_{\ell \in \mathcal{L}_n^*}
    p_\ell^\alpha \,+\, 2^{1-\alpha} p_L^\alpha ~\Bigg|~ T_n \right]
    \nonumber \\
    &= X_{n, \alpha} + (2^{1-\alpha} -1) \E \left[p_L^\alpha ~|~ T_n
    \right] \nonumber \\
    &= X_{n, \alpha} \left( 1+ \frac{2^{1-\alpha} -1}{n} \right). \label{cond}
\end{align}
Thus, $\E[Y_{n+1, \alpha} | \mathcal{F}_n] = Y_{n, \alpha}$, concluding
the proof.
\end{Proo} 

Taking expectations in Eq.~\eqref{cond} and noticing that $H_\alpha(T_n) =
\frac{1}{1-2^{1-\alpha}} (1-X_{n, \alpha})$ immediately yields the following
expression for the expected value of $H_\alpha(T_n)$ under the Yule model.
Note that the expression for $\alpha = 1$ can be obtained by taking
the limit as $\alpha \to 1$, and matches that of
\cite{bienvenuRevisitingShaoSokals2021}.

\begin{Prop} \label{expectYule}
Let $T_n$ be a Yule tree with $n$ leaves. Then, for $\alpha \neq 1$,
\begin{equation*}
    \E [H_\alpha(T_n)] =
    \frac{1}{1-2^{1-\alpha}} \left(1-\prod_{k=1}^{n-1}
    \left( 1 + \frac{2^{1-\alpha} -1}{k} \right) \right),
\end{equation*}
and, for $\alpha = 1$, $\E[H_1(T_n)] = \sum_{k = 1}^{n-1} k^{-1}$.
\end{Prop}

\begin{Coro} \label{asymptoticExpectYule}
Let $T_n$ be a Yule tree with $n$ leaves and assume that $\alpha \neq 1$.
Then, as $n\to\infty$,
\begin{equation*}
    \E[H_\alpha(T_n)] =
    \frac{1}{1 - 2^{1-\alpha}} \,+\, \frac{n^{2^{1-\alpha} -1}}{\Gamma(2^{1-\alpha})} \left(
    \frac{1}{(2^{1-\alpha} - 1)}  \,+\,
    2^{-\alpha}\, n^{-1} \,+\,
    O \Big( n^{-2} \Big) \right).
    \qedhere
\end{equation*}
\end{Coro}

\begin{Proo}
Recall (see e.g.\ \cite{erdélyiAsymptoticExpansionRatio1951a}) that,
for any fixed $z\in \mathbb{C}$, as $n\to\infty$,
\begin{equation}
    \prod_{k=1}^n \left( 1 + \frac{z}{k} \right) =
    \frac{\Gamma(z+n+1)}{\Gamma(n+1)\Gamma(z+1)} =
    \frac{n^z}{\Gamma(z+1)}
    \left( 1+\frac{z(z+1)}{2n} + O \left(n^{-2} \right) \right). \label{asymp}
\end{equation}
The result then follows from  straightforward calculations.
\end{Proo}

\begin{Rema} \label{remAsymExpanYule}
Note that, in the asymptotic expansion of $\E[H_\alpha(T_n)]$,
the leading term is:
\begin{itemize}
  \item $((2^{1-\alpha} - 1)\, \Gamma(2^{1-\alpha}))^{-1}\,n^{2^{1-\alpha}-1} $ for $\alpha < 1$;
  \item $\log(n)$, where $\log$ denotes the natural logarithm, for $\alpha = 1$;
  \item $(1 - 2^{1-\alpha})^{-1}$ for $\alpha > 1$.
\end{itemize}
Recalling that, for $\alpha > 1$, the $H_\alpha$ index is bounded above by
$(1-2^{1-\alpha})^{-1}$, we see that the Yule model produces very balanced
trees -- in fact, so much so that, for $\alpha > 1$ and large $n$,
the $H_\alpha$ index cannot distinguish them from complete binary trees.
This is consistent with our remark that large values of~$\alpha$ give more
weight to the structure next to the root: close to its root, the structure of a
large Yule tree is indeed the same as that of a complete binary tree.
However, the difference between Yule trees and complete binary trees is
already apparent for $\alpha = 1$ -- since the former is asymptotically
$\log(n)$ whereas the latter is $\log_2(n)$ -- and becomes more pronounced
as $\alpha$ decreases.
\end{Rema}

We now give an explicit expression for the variance of the $H_\alpha$ index of
a Yule tree with $n$ leaves.

\begin{Prop} \label{varianceYule}
Let $T_n$ be a Yule tree with $n$ leaves. Then, for $\alpha \neq 1$,
\begin{equation*}
    \Var(H_\alpha(T_n)) =
    \frac{1}{(1-2^{1-\alpha})^2} \left( \prod_{k=1}^{n-1} f_k \cdot
    \left( 1+ \sum_{m=1}^{n-1} \frac{g_m}{\prod_{k=1}^m f_k}  \right)
    - \prod_{k=1}^{n-1} \left( 1 + \frac{2^{1-\alpha}-1}{k} \right)^2 \right),
\end{equation*}
where $f_n = 1 + \frac{2(2^{1-\alpha}-1)}{n}$ and
$g_n = \frac{(2^{1-\alpha}-1)^2}{n} \cdot \prod_{k=1}^{n-1}
\left( 1+ \frac{2^{1-2\alpha}-1}{k} \right)$.
\end{Prop}

\begin{Proo}
First, note that
$\Var (H_\alpha(T_n)) = \frac{1}{(1-2^{1-\alpha})^2} \Var(X_{n, \alpha})$.
Thus, we only need to compute the second moment of $X_{n, \alpha}$.
Writing $\beta = 1 - 2^{1-\alpha}$,
\begin{align}
    \E[X_{n+1, \alpha}^2 ~|~T_n] &= 
    \E\big[(X_{n, \alpha} - \beta\, p_L^\alpha)^2 ~\big|~T_n\big]
    \nonumber \\
    &= X_{n, \alpha}^2 - 2\beta\,X_{n, \alpha}\,
    \E[p_L^\alpha ~|~T_n] + \beta^2 \E[p_L^{2\alpha} ~|~T_n] \nonumber \\
    &= X_{n, \alpha}^2 \left(1 - \frac{2\beta}{n} \right) +
    \frac{\beta^2}{n} X_{n, 2\alpha}. \label{reccurence2moment}
\end{align}
Thus,
\begin{equation*}
    \E \left[X_{n+1, \alpha}^2 \right] =
    \E \left[ X_{n, \alpha}^2 \right]
    \left(1 - \frac{2\beta}{n} \right) +
    \frac{\beta^2}{n} \E \left[ X_{n, 2\alpha} \right].
\end{equation*}
We recognize a first-order linear recurrence.
Therefore, since $X_{1, \alpha} = 1$ almost surely, we have:
\begin{equation}
    \E \left[ X_{n, \alpha}^2 \right] =
    \prod_{k=1}^{n-1} f_k \cdot
    \left(1 + \sum_{m=1}^{n-1} \frac{g_m}{\prod_{k=1}^m f_k} \right),
    \label{moment2}
\end{equation}
with $f_n = 1 - \frac{2\beta}{n}$ and
$g_n = \frac{\beta^2}{n} \prod_{k=1}^{n-1}
\left( 1- \frac{\beta}{k} \right)$.
This concludes the proof.
\end{Proo}

While the explicit expression in \Cref{varianceYule} is simple,
it is not immediately clear from it what the asymptotic behavior of the
variance is -- hence the following corollary.
The proof of this result being technical, we relegate it to \Cref{appYule}.

\begin{Coro} \label{asymptoticVariance}
Let $T_n$ be a Yule tree with $n$ leaves, and let
$c = -\log_{2}(1-\frac{\sqrt{2}}{2})\approx 1.7716$.
Then, for $\alpha \neq 1$, as $n\to\infty$,
\begin{equation*}
  \Var(H_{\alpha}(T_n)) \sim_{n \rightarrow + \infty}
    \begin{cases}
    Q_\alpha\, n^{2(2^{1-\alpha}-1)} & \text{if } \alpha < c, \\
    \dfrac{1}{\Gamma(2^{1-2\alpha})}\,\log(n)\, n^{\,2(2^{1-\alpha}-1)}
    & \text{if } \alpha = c, \\
    \dfrac{1}{\Gamma(2^{1-2\alpha})(1+2^{1-2\alpha}-2^{2-\alpha})}\;
      n^{\,2^{1-2\alpha}-1}
    & \text{if } \alpha > c.
    \end{cases}
\end{equation*}
where $Q_\alpha > 0$ is a constant.
\end{Coro}

We close this section with a more precise description of the limiting behavior
of $H_\alpha(T_n)$.

\begin{Prop}
For any positive $\alpha \neq 1$,
there exists a random variable $Y_{\infty, \alpha}$ such
that, almost surely,
\begin{equation*}
    H_\alpha(T_n) = \begin{cases}
    \dfrac{Y_{\infty, \alpha}}{(2^{1-\alpha}-1)\Gamma(2^{1-\alpha})} n^{2^{1-\alpha} -1} + o\left(n^{2^{1-\alpha} -1} \right) & \text{if } \alpha < 1, \\
    \dfrac{1}{1-2^{1-\alpha}} + \dfrac{Y_{\infty, \alpha}}{(2^{1-\alpha}-1)\Gamma(2^{1-\alpha})} n^{2^{1-\alpha} -1} +o\left(n^{2^{1-\alpha} -1} \right) & \text{if } \alpha > 1.
    \end{cases}
\end{equation*}
\end{Prop}

\begin{Proo}
Recall that $Y_{n, \alpha}$ is a positive martingale.
Thus, there exists a
random variable $Y_{\infty, \alpha}$ such that $Y_{n, \alpha}$ converges
to $Y_{\infty, \alpha}$ almost surely.
Since 
\begin{equation*}
    H_\alpha(T_n) =
    \frac{1}{\beta} \left( 1 - Y_{n, \alpha}
    \prod_{k=1}^n \left( 1 - \frac{\beta}{k} \right)  \right) \,, 
\end{equation*}
the announced result readily follows from Eq.~\eqref{asymp}.
\end{Proo}

%\begin{Rema}
%\begin{itemize}
%    \item For $\alpha > 1$, the $H_\alpha$ index converges almost surely
%    towards $\frac{1}{1-2^{1-\alpha}}$ which is the maximal theoretical value,
%    so the Yule tree is maximally balanced at the limit. Since the $H_\alpha$
%    index is bounded in this case, the convergence also hold in $L^p$ for all
%    $p \geqslant 1$.
%    \item For $\alpha < -\log_2 \left( 1-\frac{\sqrt{2}}{2} \right)$, we saw
%    that $Y_{n, \alpha}$ is bounded in $L^2$ so $Y_{\infty, \alpha}$ is in
%    $L^2$ and the convergence of $Y_{n, \alpha}$ also holds in this space.
%    \qedhere
%\end{itemize}
%\end{Rema}

\subsection{The PDA model} \label{secPDA}

First considered in phylogenetics by
\cite{rosenVicariantPatternsHistorical1978}, the PDA model corresponds to the
uniform distribution on the set of rooted binary trees with $n$ labeled leaves. 
The name PDA, which stands for ``proportional to distinguishable arrangements'',
can be traced back to \cite{GuyerSlowinski1991}.
This model is a special case of Aldous's $\beta$-splitting model, for
$\beta = -3/2$, and as such is a binary Markov branching tree model
\cite{aldousProbabilityDistributionsCladograms1996a}.

In this section, we give an explicit expression for the expectation of the
$H_\alpha$ index of a tree $T_n$ with $n$ leaves sampled from the PDA model.
This expression can be obtained from point (i) of \cref{reccurence},
but at the price of rather heavy computations.
Here, we instead use the definition of the PDA model as the
uniform distribution on the set of rooted binary trees with $n$ labeled leaves
to provide a simple bijective proof that involves almost no calculations.
We also use the fact that $T_n$ can be seen as a size-conditioned
Galton--Watson tree to study the asymptotic behavior of $H_\alpha(T_n)$
using branching process techniques, again with very few calculations.

\begin{Lemm} \label{lemGrandDyck}
Let $F(n, k)$ be the number of ordered pairs $(\tau, \ell)$, where
$\tau$ is an ordered binary tree with $n$ leaves and $\ell$ is a leaf of $\tau$
at distance $k$ from its root.
Then, 
\[
  F(n, k) = \frac{k\,2^k}{2(n-1)-k} \binom{2(n-1)-k}{n-1}  \,.
\]
This is also the number of grand Dyck paths of semilength $n-1$ with $k$
returns to the $x$-axis
(see sequence A108747 of the On-Line Encyclopedia of Integer Sequences
\cite{OEISA108747}).
\end{Lemm}

The definition of grand Dyck paths and the proof of \Cref{lemGrandDyck} can be
found in \Cref{appPDA}.
With this lemma, we readily get the following theorem.

\begin{Theo} \label{expecPDA}
Let $T_n$ be a PDA tree with $n \geqslant 2$. Then,
\begin{equation*}
    \E[H_\alpha(T_n)] =
    \frac{1}{1-2^{1-\alpha}} \left( 1- \frac{1}{C_{n-1}}
    \sum_{k=1}^{n-1} F(n, k)\, 2^{-\alpha k} \right),
\end{equation*}
where $C_n = \frac{1}{n+1} \binom{2n}{n}$ is the $n$-th Catalan number and
$F(n, k)$ is defined as in \Cref{lemGrandDyck}.
\end{Theo}

\begin{Proo}
First, recall that the $H_\alpha$ index of a binary tree $T$ is
\[
    H_\alpha(T) = 
    \frac{1}{1-2^{1-\alpha}} \left(1 -
    \sum_{\ell \in \mathcal{L}_T}\! 2^{-\alpha \delta_\ell} \right),
\]
where $\delta_\ell$ is the depth of leaf~$\ell$, i.e.\ its distance to the
root of $T$. Thus, to conclude the proof it suffices to show that if
$T_n$ is a PDA tree with $n$ leaves, then
\[
  \E\left(\sum_{\ell \in \mathcal{L}_{T_n}}\!\! 2^{-\alpha \delta_\ell} \right)
  = \frac{1}{C_{n-1}} \sum_{k=1}^{n-1} F(n, k)\, 2^{-\alpha k}.
\]

Let $\mathcal{T}_n^l$ denote the set of rooted binary trees with
$n$ leaves labeled with the integers $\{1, \ldots, n\}$; $\mathcal{P}_n$ the
set of Catalan trees (i.e.\ ordered binary trees) with $n$ unlabeled leaves;
and $\mathcal{P}_n^l$ the set of Catalan trees with $n$ leaves labeled
$1, \ldots, n$.
Let $\pi_1\colon\mathcal{P}_n^l\to\mathcal{T}_n^l$ (resp.\
$\pi_2\colon \mathcal{P}_n^l\to\mathcal{P}_n$) be the canonical
projection obtained by discarding the ordering of the tree
(resp.\ the labels of the leaves).
Then, for all $\tau_1 \in \mathcal{T}_n^l$ and all $\tau_2 \in \mathcal{P}_n$,
\begin{equation*}
      |\pi_1^{-1}(\tau_1)| = 2^{n-1} \qquad \text{and} \qquad
      |\pi_2^{-1}(\tau_2)| = n!
\end{equation*} 
As a result, for any
$f \colon (\mathcal{P}_n^l\cup\mathcal{T}_n^l\cup\mathcal{P}_n) \to \mathbb{R}$
that depends neither on the ordering of trees 
%(that is, such that $f(\tau) = f(\pi_1(\tau))$ for all $\tau\in\mathcal{P}_n^l$)
nor on the labels of their leaves, 
%($f(\tau) = f(\pi_2(\tau))$ for all $\tau\in\mathcal{P}_n^l$), we have
\begin{equation*}
    \E_{\mathrm{Unif}(\mathcal{P}_n^l)}[f(T)] =
    \E_{\mathrm{Unif}(\mathcal{T}_n^l)}[f(T)] =
    \E_{\mathrm{Unif}(\mathcal{P}_n)}[f(T)] \,.
\end{equation*}
In particular, for $T_n \sim \mathrm{PDA}_n$,
\[
  \E\left(\sum_{\ell \in \mathcal{L}_{T_n}}\!\! 2^{-\alpha \delta_\ell} \right)
  =  \frac{1}{|\mathcal{P}_n|} \sum_{\tau \in \mathcal{P}_n}
  \sum_{\ell \in \mathcal{L}_\tau} 2^{-\alpha \delta_\ell} .
\]
Moreover,
\begin{align*}
  \frac{1}{|\mathcal{P}_n|} \sum_{\tau \in \mathcal{P}_n}
  \sum_{\ell \in \mathcal{L}_\tau} 2^{-\alpha \delta_\ell}
  &=  \frac{1}{|\mathcal{P}_n|} \sum_{\tau \in \mathcal{P}_n}
  \sum_{k = 1}^{n-1}
  |\{\ell \in \mathcal{L}_\tau : \delta_\ell = k\}| \, 2^{-\alpha k} \\
  &=  \frac{1}{|\mathcal{P}_n|} \sum_{k = 1}^{n-1} 
  |\{(\tau, \ell) : \tau \in \mathcal{P}_n, \ell\in \mathcal{L}_\tau,
  \delta_\ell = k\}| \, 2^{-\alpha k} \,.
\end{align*}
Since $|\mathcal{P}_n| = C_{n-1}$ and, by definition,
$F(n, k) = |\{(\tau, \ell) : 
\tau \in \mathcal{P}_n, \ell\in \mathcal{L}_\tau, \delta_\ell = k\}|$, 
we conclude the proof by \Cref{lemGrandDyck}.
\end{Proo}

We now study the asymptotic behavior of $H_\alpha(T_n)$ as $n \to\infty$.
For this, we use the connection between the PDA model and
conditioned Galton--Watson trees (see e.g.\
\cite[Prop.~1.3.4]{lambertProbabilisticModelsSubtrees2017}):
after discarding the labels of the leaves, a PDA tree with $n$ leaves is
distributed as a Galton--Watson tree with offspring
distribution $2\,\mathrm{Ber}(1/2)$ conditioned to have $n$ leaves
(after discarding the ordering of children of the vertices).
As a result, a special case of \cref{theoprincipal} from \Cref{secBlowup} -- where we recall the definition of Kesten trees -- yields the following proposition.

\begin{Prop} \label{asympPDA}
Let $T_\star$ be the Kesten tree associated with a Galton--Watson tree having a
$2\,\mathrm{Ber}(1/2)$ offspring distribution and let $T_n$ be a PDA tree
with $n$ leaves. 
We have the following convergence in distribution:
\begin{equation*}
  H_\alpha(T_n) \xrightarrow[n\to\infty]{d}
    H_\alpha(T_\star). \qedhere
\end{equation*}
\end{Prop}

When $\xi\sim 2\,\mathrm{Ber}(1/2)$, the corresponding size-biased offspring
distribution $\hat{\xi}$ is equal to 2 almost surely.
As a result, the associated Kesten tree $T^\star$ has a very simple recursive
structure:
it is a cherry with a Galton--Watson tree grafted on one of the leaves and an
independent copy of $T^\star$ grafted on the other one.
Thus, \cref{graftingtwophylogeny} yields the following
equality in distribution:
\begin{equation}
    H_\alpha(T_\star) \;\overset{d}{=}\;
    \sum_{k \geqslant 0} 2^{-\alpha k} \big(1 + 2^{-\alpha} H_\alpha(T_k)\big),
    \label{recKesten}
\end{equation}
where $(T_k)_{k\geqslant 0}$ are i.i.d.\ Galton--Watson trees with
offspring distribution $\xi \sim 2\,\mathrm{Ber}(1/2)$.
Moreover, by \cref{expecGW},
\begin{equation} \label{expecFreePDA}
    \E[H_\alpha(T_0)] = \frac{1}{2(1-2^{-\alpha})}.
\end{equation}
After straightforward calculations, combining Eqs.~\eqref{recKesten}
and~\eqref{expecFreePDA} yields the following expressions for the limit of the
expected value / variance of $H_\alpha(T_n)$;
note that the expression of the expectation is a special case of \cref{expecKesten},
which gives the corresponding expressions for arbitrary Kesten trees.

\begin{Coro} \label{asympexpecPDA}
Let $T_n$ be a PDA tree with $n$ leaves.
Then, for all $\alpha > 1/2$,
\begin{equation*}
    \E[H_\alpha(T_n)]
    \xrightarrow[n\to\infty]{}
    \frac{1-2^{-(\alpha +1)}}{(1-2^{-\alpha})^2}. \qedhere
\end{equation*}
\end{Coro}

\begin{Coro} \label{asympvarPDA}
Let $T_n$ be a PDA tree with $n$ leaves.
Then, for all $\alpha > 3/4$, 
\begin{equation*}
    \Var(H_\alpha(T_n))
    \xrightarrow[n\to\infty]{}
    \frac{2^{-(2\alpha +1)}}{(1-2^{-2\alpha})^2(1-2^{1-\alpha})^2}
    \left( 1 + \frac{2^{1-2\alpha}-1}{2(1-2^{-\alpha})^2} \right).
    \qedhere
\end{equation*}
\end{Coro}

\section{The \texorpdfstring{$\boldsymbol{H_\alpha}$}{H-alpha} index of blowups of Galton--Watson trees} \label{secBlowup}

In this section,  we study a class of random phylogenetic networks known as
\emph{blowups of Galton--Watson trees}, whose definition we recall below.
The two main interests of these models are:

\begin{enumerate}
    \item Blowups of Galton--Watson trees encompass models that are
      biologically relevant.
    Indeed, Stufler showed in \cite{stuflerBranchingProcessApproach2022} that
% level-$k$ networks are in bijection with a class of decorated trees and he deduces that
    the uniform distribution on leaf-labeled level-$k$ networks can be seen as
    specific blowups of Galton--Watson trees conditioned on their number of
    leaves.
    Other natural models of networks generated by a branching process with
    coalescence and mutation have also been identified as blowups of Galton--Watson
    trees \cite{bienvenuBranchingProcessCoalescence2024}.
    \item As we saw in \Cref{secGW}, the recursive structure of
      Galton--Watson trees makes their $H_\alpha$ indices tractable.
      Because the ``blowup'' procedure preserves this recursive structure,
      the $H_\alpha$ index of blowups of Galton--Watson trees can
      be studied using similar methods to those used for Galton--Watson trees.
\end{enumerate}
We study the asymptotic behavior of $H_\alpha$ for blowups of Galton--Watson
trees as their number of leaves goes to infinity.
Building on the local limit framework developed for the $B_2$ index
in~\cite{bienvenu$B_2$IndexGalled2024}, we extend the approach to the present
setting and derive new estimates tailored to this model.
Let us start by recalling the blowup construction and
introducing some notation.

\subsection{Blowups of trees and local limits: setting and
  notation} \label{sec:blowup-notation}

First, let us define random blowups of locally finite rooted trees.
Blowups are most conveniently defined for ordered trees, that is, trees $T$
in which, for each vertex $v\in T$, the (finite) set of children of~$v$ is
equipped with a total order; we can thus speak of the $i$-th child of~$v$.
If $T$ is an unordered tree, we call an ordered tree $T'$ an
\emph{ordering of $T$} if we recover $T$ by forgetting the order on the
children of each vertex of $T'$.

Throughout the rest of this section, for all $k \geq 1$ let $\nu_k$ be a
probability distribution on the space of finite phylogenetic networks with $k$
leaves labeled from $1$ to $k$, and assume that $\nu_k$ is invariant under
relabeling of the leaves of the network. Write $\nu=(\nu_k)_{k\geq 1}$.

\begin{Defi}
Let $T$ be a locally finite ordered tree, and let $\nu$ be a family
of distributions defined as above.
The \emph{blowup of} $T$ \emph{with respect to $\nu$} is the random phylogenetic
network obtained by (1) sampling an independent family of networks
$(\Gamma_v)_{v\in T\setminus \mathcal{L}_T}$,
with ${\Gamma_v \sim \nu_{d_+(v)}}$, where $d_+(v)$ is the number of children
of $v$; and (2) replacing each internal vertex
$v \in T\setminus \mathcal{L}_T$ with the network $\Gamma_v$, identifying
the $k$-th leaf of $\Gamma_v$ with the $k$-th child of $v$.
This procedure is illustrated in \Cref{Blowup_Procedure}.
\end{Defi}

If $T$ is a locally finite unordered tree, choose an arbitrary ordering $T'$
of $T$ and define the blowup of $T$ with respect to $\nu$ as the phylogenetic
network obtained by forgetting the ordering of the blowup of $T'$ with
respect to $\nu$ (note that the distribution of the resulting network
does not depend on the ordering of $T$ used in the construction).
When $T$ is a random tree, the blowup of $T$ with respect to $\nu$ is defined
using the procedure above \emph{conditional on} $T$.
This can be formalized without difficulty,
see~\cite{bienvenu$B_2$IndexGalled2024} for details.

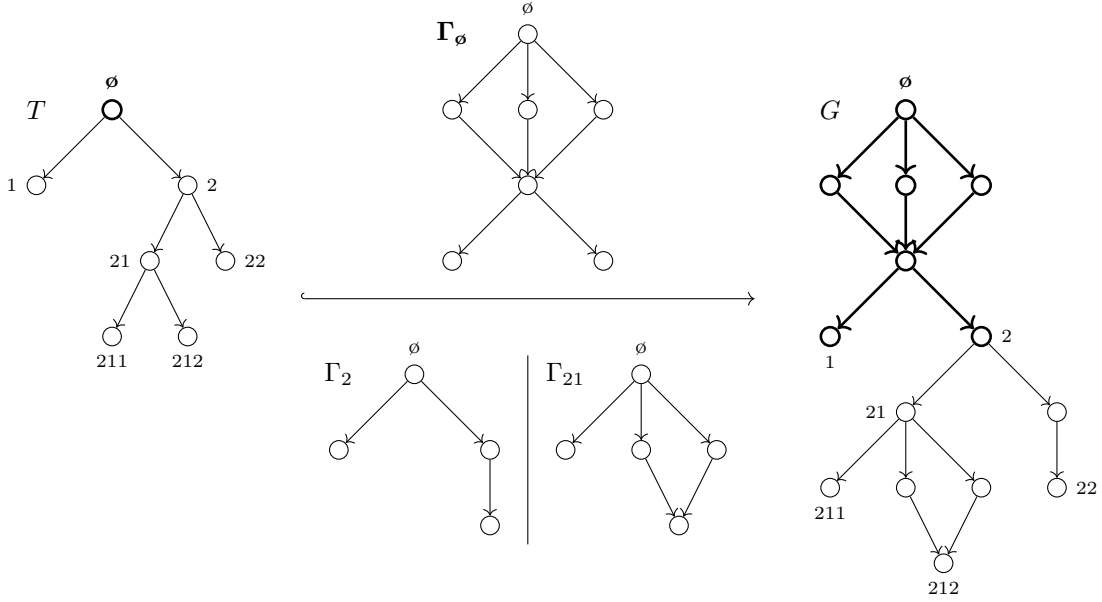
\begin{figure}[H]

  \centering

  \begin{tikzpicture}[main/.style = {draw, circle, minimum size = 7pt, inner sep =1pt}]
    \tikzstyle{dot}=[draw,circle,gray,
                  top color= white, text=gray,minimum size=7pt, inner sep = 1]
    \tikzstyle{phantom}=[opacity = 0, minimum size = 7pt]
    \tikzstyle{leaf}=[draw,rounded corners=2pt,red,
                  top color= white, text=red,minimum size=7pt]
    \tikzstyle{naming}=[text = black, minimum size=7pt]
    \tikzstyle{thicknode}=[draw,line width=1pt, circle, minimum size=7pt, inner sep = 1]
    \node[thicknode, label ={[font=\footnotesize] \textbf{\o}}] (root) at (-1, -1) {};
    \node[main, label ={[font=\scriptsize]left :1}] (1) at (-2, -2) {};
    \node[main, label ={[font=\scriptsize]right :2}] (2) at (0, -2) {};
    \node[main, label ={[font=\scriptsize]left :21}] (21) at (-0.5, -3) {};
    \node[main, label ={[font=\scriptsize]right :22}] (22) at (0.5, -3) {};
    \node[main, label ={[font=\scriptsize]below :211}] (211) at (-1, -4) {};
    \node[main, label ={[font=\scriptsize]below :212}] (212) at (0, -4) {};
    \draw[->] (root) -- (1);
    \draw[->] (root) -- (2);
    \draw[->] (2) -- (21);
    \draw[->] (2) -- (22);
    \draw[->] (21) -- (212);
    \draw[->] (21) -- (211);
    \draw[{Hooks[right]}->] (1.5,-3.5) -- (7.5,-3.5);
    \node[thicknode, label ={[font=\footnotesize] \textbf{\o}}] (root2tree) at (9.5, -1) {};
    \node[thicknode] (1') at (8.5,-2) {};
    \node[thicknode] (2') at (9.5,-2) {};
    \node[thicknode] (3') at (10.5,-2) {};
    \node[thicknode] (21') at (9.5,-3) {};
    \node[thicknode, label ={[font=\scriptsize]below :1}] (211') at (8.5,-4) {};
    \node[thicknode, label ={[font=\scriptsize]right :2}] (212') at (10.5,-4) {};
    \node[main, label ={[font=\scriptsize]left :21}] (2121') at (9.5,-5) {};
    \node[main] (2122') at (11.5,-5) {};
    \node[main, label ={[font=\scriptsize]right :22}] (21221') at (11.5,-6) {};
    \node[main, label ={[font=\scriptsize]below :211}] (21211') at (8.5,-6) {};
    \node[main] (21212') at (9.5,-6) {};
    \node[main] (21213') at (10.5,-6) {};
    \node[main, label ={[font=\scriptsize]below :212}] (212121') at (10,-7) {};
    \draw[->, line width=1pt] (root2tree) -- (1');
    \draw[->, line width=1pt] (root2tree) -- (2');
    \draw[->, line width=1pt] (root2tree) -- (3');
    \draw[->, line width=1pt] (1') -- (21');
    \draw[->, line width=1pt] (2') -- (21');
    \draw[->, line width=1pt] (3') -- (21');
    \draw[->, line width=1pt] (21') -- (211');
    \draw[->, line width=1pt] (21') -- (212');
    \draw[->] (212') -- (2121');
    \draw[->] (212') -- (2122');
    \draw[->] (2122') -- (21221');
    \draw[->] (2121') -- (21211');
    \draw[->] (2121') -- (21212');
    \draw[->] (2121') -- (21213');
    \draw[->] (21212') -- (212121');
    \draw[->] (21213') -- (212121');
    \node[main, label ={[font=\footnotesize] \o}] (rootGamma2) at (3, -4.5) {};
    \node[main] (1Gamma2) at (2, -5.5) {};
    \node[main] (2Gamma2) at (4, -5.5) {};
    \node[main] (21Gamma2) at (4, -6.5) {};
    \draw[->] (rootGamma2) -- (1Gamma2);
    \draw[->] (rootGamma2) -- (2Gamma2);
    \draw[->] (2Gamma2) -- (21Gamma2);
    \draw[] (4.5, -4.25) -- (4.5, -6.75);
    \node[main, label ={[font=\footnotesize] \o}] (rootGamma21) at (6, -4.5) {};
    \node[main] (1Gamma21) at (5, -5.5) {};
    \node[main] (2Gamma21) at (6, -5.5) {};
    \node[main] (3Gamma21) at (7, -5.5) {};
    \node[main] (21Gamma21) at (6.5, -6.5) {};
    \draw[->] (rootGamma21) -- (1Gamma21);
    \draw[->] (rootGamma21) -- (2Gamma21);
    \draw[->] (rootGamma21) -- (3Gamma21);
    \draw[->] (2Gamma21) -- (21Gamma21);
    \draw[->] (3Gamma21) -- (21Gamma21);
    \node[main, label ={[font=\footnotesize] \o}] (rootGammaRoot) at (4.5, 0) {};
    \node[main] (1GammaRoot) at (3.5, -1) {};
    \node[main] (2GammaRoot) at (4.5, -1) {};
    \node[main] (3GammaRoot) at (5.5, -1) {};
    \node[main] (21GammaRoot) at (4.5, -2) {};
    \node[main] (211GammaRoot) at (3.5, -3) {};
    \node[main] (212GammaRoot) at (5.5, -3) {};
    \draw[->] (rootGammaRoot) -- (1GammaRoot);
    \draw[->] (rootGammaRoot) -- (2GammaRoot);
    \draw[->] (rootGammaRoot) -- (3GammaRoot);
    \draw[->] (1GammaRoot) -- (21GammaRoot);
    \draw[->] (2GammaRoot) -- (21GammaRoot);
    \draw[->] (3GammaRoot) -- (21GammaRoot);
    \draw[->] (21GammaRoot) -- (211GammaRoot);
    \draw[->] (21GammaRoot) -- (212GammaRoot);
    \node[naming] (nameGammaRoot) at (3.5,0) {$\boldsymbol{\Gamma_{\textbf{\o}}}$};
    \node[naming] (nameGamma2) at (2,-4.5) {$\Gamma_{2}$};
    \node[naming] (nameGamma21) at (5,-4.5) {$\Gamma_{21}$};
    \node[naming] (nameT) at (-2,-1) {$T$};
    \node[naming] (nameG) at (8.5,-1) {$G$};
  \end{tikzpicture}
  \caption{Representation of the blowup procedure. On the left: $T$, an ordered tree labeled according to the Neveu--Ulam--Harris notation. In the center: $\Gamma_{\text{\o}}, \Gamma_2$ and $\Gamma_{21}$, the random networks associated to the internal vertices of $T$. On the right: $G$, the network resulting from the blowup procedure where the corresponding vertices of $T$ can be traced via the Neveu--Ulam--Harris labeling. For instance the root \textbf{\o} of $T$ has been replaced in $G$ by the network shown in bold, corresponding to $\boldsymbol{\Gamma_{\textbf{\o}}}$, where the two children of \textbf{\o}, namely $1$ and $2$, have been identified with the two leaves of $\boldsymbol{\Gamma_{\textbf{\o}}}$.}
  \label{Blowup_Procedure}
\end{figure}

A key property of the blowup construction is that it only involves
\emph{local} modifications of the underlying tree $T$.
As a result, it plays nicely with the notion of
local limits of graphs (also known as the
\textit{Benjamini--Schramm} limit
\cite{benjaminiRecurrenceDistributionalLimits2011}).
In our setting, this notion can be defined as follows.

\begin{Defi}
Given a sequence of phylogenetic networks $(G_n)_{n\geq 1}$, if there exists
a network $G$ such that for each fixed $k\geq 1$,  $[G_n]_k = [G]_k$ for all
$n$ large enough we say that $G$ is the \emph{local limit} of the sequence
$(G_n)_{n\geq 1}$.
\end{Defi}

It is standard (we refer to~\cite{vanderhofstadRandomGraphsComplex}) that this
notion of convergence is associated to a topology known as the \emph{local
topology}, and that the space of locally finite rooted graphs equipped with
the local topology is Polish. This makes it possible to speak of
\emph{convergence in distribution for the local topology} (or simply
\emph{local convergence in distribution}) for random phylogenetic networks.

Informally, \cref{graftingproperty} suggests that the $H_\alpha$ index
of a network depends in great part on its structure close to the root. Thus,
it is a natural question to ask whether $H_\alpha$ is continuous for
the local topology. In \Cref{appNonContinuity}, we will see that this is
not the case: if $G$ is an infinite network such that
$H_\alpha(G)<\infty$, then $H_\alpha$ is discontinuous at $G$.
However, we identify some
conditions for continuity of $H_\alpha$ along specific sequences, see
\Cref{appCondContinuity}.

The local limits as $n\to\infty$ of Galton--Watson trees conditioned to have $n$
leaves are well understood.
When the offspring distribution has a finite mean, we obtain the so-called
Kesten trees as limits and we refer the reader to
e.g.~\cite{abrahamLocalLimitsConditioned2014} for an overview of the subject. 
If, moreover, the offspring distribution has mean 1, the structure of the associated Kesten tree can easily be described.
\begin{Defi}
  Let $T$ be a critical Galton--Watson tree with offspring
  distribution $\xi$. The Kesten tree associated with $\xi$ and
  denoted by $T_\star$ is a multi-type Galton--Watson tree, that is a
  Galton--Watson tree where vertices carry extra information
  referred to as a \emph{type}, on which the offspring distribution
  may depend. More precisely, $T_\star$ is a two-type (spine/normal)
  Galton--Watson tree distributed as follows:
    \begin{itemize}
        \item The root of $T_\star$ is a \emph{spine} individual.
        \item \emph{Normal} individuals reproduce according to the
          offspring distribution $\xi$ and all of their children are
          \emph{normal}.
        \item \emph{Spine} individuals reproduce according to the
          size--biaised distribution $\hat{\xi}$ (that is
          $\mathds{P}(\hat{\xi} = k) = k \mathds{P}(\xi = k)$ for all
          $k \geqslant 0$) and exactly one of their children, chosen
          uniformly at random, is a \emph{spine} individual, while the
          others are \emph{normal} individuals. \qedhere
    \end{itemize}
\end{Defi}

The blowup procedure only changes the structure of a network on a local scale.
Thus, it is not hard to see that if a sequence of rooted trees $(T_n)_{n\geq
1}$ converges locally to $T$, then the associated sequence of blowups
$(G_n)_{n\geq 1}$ converges locally to $G$, where $G$ is the blowup of $T$.
As for Galton--Watson trees, one can study of the $H_\alpha$ indices of a Kesten tree using its simple recursive structure.
For instance, letting $T_\star$ be the Kesten tree associated with a critical
Galton--Watson tree with offspring distribution $\xi$, and letting $\hat{\xi}$
be the associated size-biased distribution, one has the following result, 
whose proof is deferred to~\Cref{appHalphaKesten}.

\begin{Theo} \label{expecKesten}
  Assume that $\mathds{E}[\xi]=1$ and $\mathds{P}(\xi = 1) < 1$.
  Then, $\mathds{E}[H_\alpha(T_\star)]$ is finite if and only if
  $\mathds{E}[\hat{\xi}^{1-\alpha}] < + \infty$. In this case, we have
  the following expression:
\begin{equation*}
\mathds{E}[H_\alpha(T_\star)] = \frac{\frac{1}{1-2^{1-\alpha}}(1-\mathds{E}[\hat{\xi}^{1-\alpha}])+\mathds{E}[\hat{\xi}^{-\alpha}(\hat{\xi}-1)] \cdot \mathds{E}[H_\alpha(T)]}{1- \mathds{E}[\hat{\xi}^{-\alpha}]},
\end{equation*}
where $\mathds{E}[H_\alpha(T)]$ is explicit in view of \cref{expecGW}.
\end{Theo}

In the next section, we use the structure of the Kesten tree to study the
associated blowups and retrieve information about the limiting behavior of
blowups of Galton--Watson trees.

\subsection[Continuity of the \texorpdfstring{$H_\alpha$}{H-alpha} index for blowups of Galton--Watson trees]{%
Continuity of the \texorpdfstring{$\boldsymbol{H_\alpha}$}{H-alpha} index for blowups of Galton--Watson trees}

Let us start with some notation.
Let $T$ be a critical Galton--Watson tree with offspring distribution $\xi$.
For all $n$ such that the event $A_n = \{\abs{\mathcal{L}_T} = n\}$ has a
non-zero probability, let $T_n$ be a tree distributed as $T$ conditioned
on~$A_n$.
Let $T_\star$ be the local limit of $T_n$ i.e.\ the Kesten tree associated
with~$T$.
Finally, consider $(G_n)_{n\geq 1}$ the sequence of blowups of
$(T_n)_{n\geq 1}$ with respect to some fixed family of distribution~$\nu$, and
$G_*$ the blowup of $T_\star$ with respect to $\nu$.

We now have all the necessary tools to
study the limiting behavior of $(H_\alpha(G_n))_{n\geq 1}$.
The following theorem, which is a generalization of
\cite[Thm~3.7]{bienvenu$B_2$IndexGalled2024}, allows one to compute the limit
of moments of the $H_\alpha$ index of a Galton–Watson tree blowup using its
local limit.
In particular, this theorem is useful when studying the asymptotic behavior of
the $H_\alpha$ index of a tree built under the PDA model; see \Cref{secPDA}.

\begin{Theo} \label{theoprincipal}
  With the notation above, assuming
  that $\mathds{E}[\xi]=1$ and $\mathds{P}(\xi = 1) < 1$, we have:
\begin{itemize}[label = \textbullet]
\item For all $\alpha > 1$,
\begin{enumerate}[label = (\roman*)]
\item $H_\alpha(G_n) \rightarrow H_\alpha(G_\star)$ in distribution and,
\item for all $ m \geq 1, ~ \mathds{E}[H_\alpha(G_n)^m] \rightarrow \mathds{E}[H_\alpha(G_\star)^m]$, and all these moments are finite.
\end{enumerate}
\item For all $\alpha \leq 1$, if $\xi$ has a finite third moment,
\begin{enumerate}[label = (\roman*)]
\item $H_\alpha(G_n) \rightarrow H_\alpha(G_\star)$ in distribution and,
\item for all $m \geq 1$ such that $m ( 1-\alpha) < 1/2, ~\mathds{E}[H_\alpha(G_n)^m] \rightarrow \mathds{E}[H_\alpha(G_\star)^m]$, and all these moments are finite. \qedhere
\end{enumerate}
\end{itemize}
\end{Theo}

We prove this result in \Cref{appBlowups}. Note that in the $B_2$ case
(when $\alpha=1$ in our setting) this result was already proven
in~\cite{bienvenu$B_2$IndexGalled2024}; the condition
$m(1-\alpha)<1/2$ is then always satisfied, so all moments converge.

\section{Simulations: statistical power of the \texorpdfstring{$\boldsymbol{H_\alpha}$}{H-alpha} index}

To conclude this article, let us return briefly to the biological motivations.
One of the main uses of balance indices in practice is
to compare random tree models with one another or against real data;
see e.g.\
\cite{aldousStochasticModelsDescriptive2001,
agapowPowerEightTree2002, blumWhichRandomProcesses2006, kirkpatrickSearchingEvolutionaryPatterns1993,
mooersInferringEvolutionaryProcess1997}.
In this context, a ``good'' balance index is a balance index that discriminates
effectively between trees generated by different models.

Recently, Kersting, Wicke and Fischer published an overview providing
a systematic comparison of the statistical power of a multitude of
balance indices \cite{kerstingTreeBalancePhylogenetic2025}. Here, the
term \emph{power} refers to the ability of an index to distinguish
alternative models from a null model. In practice, it results from the following test: given a balance index $I$, build an acceptance region at a
chosen confidence level for $I$ under the null model (such regions are typically built numerically using simulations). Then,
compute $I$ for a large number of trees generated independently under
an alternative model. For each such tree $T$, if $I(T)$ falls outside
of the acceptance region, we say that $T$ is rejected. The power of
$I$ is the proportion of trees rejected by the test. The higher the
power, the better the balance index is at distinguishing the
alternative model from the null model. Of course, this statistical
power is not an intrinsic property of a balance index, as it depends
on the specific choice of both the null and alternative models, on the
confidence level and on the chosen construction of the
acceptance region; nevertheless, it is possible to compare these
powers for different models to look for general trends.

Our goal here is not to provide an exhaustive comparison between
the $H_\alpha$ indices and all existing tree balance indices, as
this would constitute a research project of its own.
Rather, we compile a few illustrative results about this family and describe
how it fits into the current landscape of tree balance indices (recall, however,
that $H_\alpha$ is not merely a \emph{tree} balance index).
To do so, we use the R packages \texttt{poweRbal} and \texttt{treebalance}
\cite{kerstingPoweRbal2026, fischerTreebalance2021}.
In particular, this means that the acceptance region used to compute the following results is the one built in \cite{kerstingTreeBalancePhylogenetic2025} for the confidence level of 5\%.

In what follows, we compare the $H_\alpha$ indices to the Sackin index, the
$\hat{s}$-Shape index \cite{blumWhichRandomProcesses2006}, and the number
of cherries \cite{mckenzieDistributionsCherriesTwo2000}.
This choice is somewhat arbitrary. It is, however, motivated by
the fact that, in the systematic study
\cite{kerstingTreeBalancePhylogenetic2025}, the $\hat{s}$-Shape index was
usually among the best-performing indices -- if not the best -- often
followed by the Sackin index; by contrast, the number of cherries was
frequently one of the worst-performing indices.
We thus chose these three indices to see how the $H_\alpha$ family would
compare to one of the ``best'' indices (the $\hat{s}$-shape index);
to a well-established index with good performance (the Sackin index); and to
one of the ``worst'' indices (the number of cherries).
Note however that the terms ``best'' and ``worst'' should be interpreted with
caution, as the conclusion of \cite{kerstingTreeBalancePhylogenetic2025} is
precisely that no index consistently outperforms all others.
    
Regarding the $H_\alpha$ indices, we will limit ourselves to five values
of $\alpha$: $\{0.01, 0.25, 0.5, 1, 2 \}$.
This is motivated by the following two points.
\begin{enumerate}
    \item As we saw throughout this paper, the behavior of
      $H_\alpha$ is very different for $\alpha < 1$ and $\alpha > 1$.

      When $\alpha$ is lower than 1, the $H_\alpha$ index is not bounded
      unlike when $\alpha$ is greater than 1 and, as we saw throughout this
      paper, this dichotomy has numerous consequences.
      Thus, it seems wise to consider parameters both lower and greater than
      one.

      For the one greater than one, we might as well look at $\alpha= 2$ to
      provide results about the Gini-Simpson balance index we defined.  We will
      also see that $H_\alpha$ indices for $\alpha$ lower than one seem to
      perform better on average (see \Cref{CompAlpha,OptimalAlphaNotConstant})
      so we include the values $0.25$ and $0.5$ in our study.
      We will also provide results for $\alpha = 1$ as it corresponds to the
      $B_2$ index.
    \item As we saw in \Cref{basicProperties}, when $\alpha$ is close to zero,
      the variations of the $H_\alpha$ index of a binary tree can be described
      using only the number of leaves and the Sackin index.
      Thus, when the number of leaves is fixed (hypothesis that we are destined
      to make regardless) the Sackin index solely explains those variations.
      This observation is confirmed by simulations : the power of the
      $H_\alpha$ index when $\alpha$ is close to zero is always very close to
      the power of the Sackin index regardless of the considered null and
      alternative models (see for instance the results obtained for $H_{0.01}$
      and Sackin in \Cref{YulevsMundo} when the null model is the Yule model). 
\end{enumerate}

In the literature (see e.g.\ \cite{agapowPowerEightTree2002, blumWhichRandomProcesses2006,
  kirkpatrickSearchingEvolutionaryPatterns1993}), the Yule model is
often chosen as null model and the first alternative model to be
considered is the PDA model. For this classic comparison, the
$H_\alpha$ indices are not as good as the best index in this situation
: the $\hat{s}$-Shape statistic nor even as good as the Sackin index
when we consider trees with 30 leaves (see \Cref{YulevsMundo}).
However, as soon as the number of leaves increases $H_{0.25}$ and
$H_{0.5}$ perform nearly optimally in the sense that they come close
to Sackin and $\hat{s}$-Shape (see \Cref{FunctionOfLeaves}). When
comparing the Yule model to some of the various alternative models
presented in \cite{kerstingTreeBalancePhylogenetic2025}, we observe
that the $H_\alpha$ performances for values of $\alpha$ lower than one
half are on average less good than the $\hat{s}$-Shape statistic
performances and as good as the ones of the Sackin index. Once again,
this affirmation is to be nuanced since the results are largely
influenced by the number of leaves of the trees we consider. The
alternative models we consider are again taken from
\cite{kerstingTreeBalancePhylogenetic2025}; the reader is therefore
referred to that article for a description of these models. Our
selection includes the most widespread models, such as the
\emph{Yule}, \emph{PDA} and the \emph{$\beta$-splitting with
  $\beta = -1$} models as well as the following less
popular models which nevertheless provide useful points of comparison:
\emph{Equiprobable-types-models}; \emph{alternative
  birth-death} with parameters $\lambda_0 = 1$ and
$\mu_0 = 0.5$, \emph{Simple Brownian}
with parameter $\sigma = 1$; \emph{Linear Brownian}
with parameters $\sigma_x = 1$ and
$\sigma_\lambda = 0.5$; \emph{Direct-children-only}
with parameter $\zeta = 1$ and \emph{Inherited
  fertility} with parameter $\zeta = 1$.

\begin{figure}[H]
    \centering
    \includegraphics[width=0.9\linewidth]{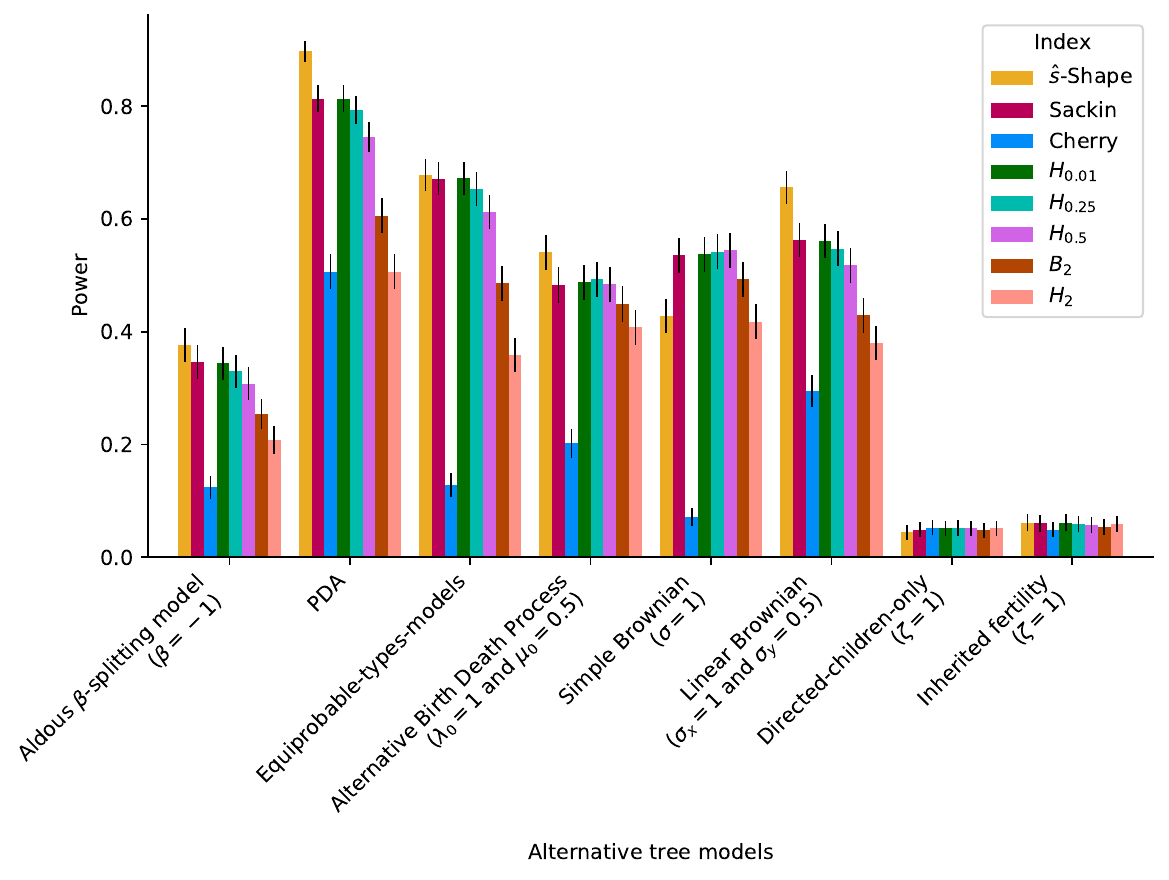}
    \caption{Powers of the selected balance indices, when comparing
      each alternative model to the Yule null model. The number of
      leaves of each simulated tree is 30 for these simulations.}
    \label{YulevsMundo}
\end{figure}

However, if one changes the model used in the null hypothesis of the test, the
order induced by the indices performances can be totally different.
The $H_\alpha$ indices for $\alpha$ lower than one will still be comparable to
the Sackin index in terms of power but, when taking the \emph{Simple Brownian} model with parameter 1, the $\hat{s}$-Shape statistic
performs less well than the $H_\alpha$ indices when considering for instance
the alternative models \emph{Yule}, \emph{Direct-children-only} with
parameter 1 or \emph{Inherited fertility} with parameter 1 (see
\Cref{SBSvsMundo}).

\begin{figure}[H]
    \centering
    \includegraphics[width=0.9\linewidth]{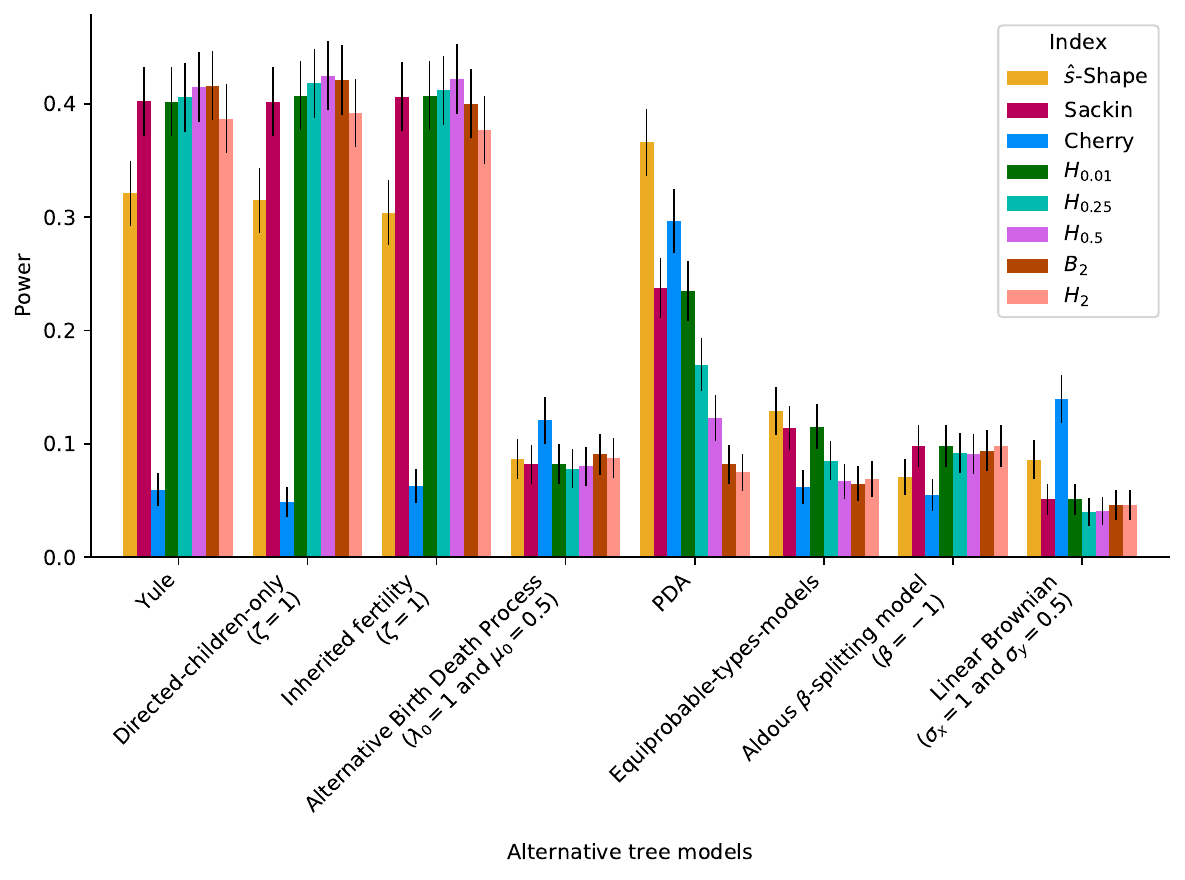}
    \caption{Powers of the selected balance indices, when comparing
      each alternative model to the symmetric simple Brownian (with
      parameter 1) null model. The number of leaves of each simulated
      tree is 30 for these simulations.}
    \label{SBSvsMundo}
\end{figure} 

To compare the power of these different indices as a function of the number of
leaves, we plotted the corresponding curves in a few situations (see
\Cref{FunctionOfLeaves}).
Since the $H_{0.01}$ power is almost identical to the
one of the Sackin index and the $H_{0.5}$ power is often close to the power of
$H_{0.25}$ we chose not to present them to emphasize the following results.
First, we see that the more leaves are considered, the better is the power of
any index, which is consistent with the observations in
\cite{kerstingTreeBalancePhylogenetic2025}.
Moreover, the speed at which this improvement occurs depends on the considered
models.
Finally it seems important to notice that some curves can intersect meaning
that the order induced by the indices power can depend on the number of leaves
(see for instance the graphs of the Cherry index and the $B_2$ index in
\Cref{YulevsPDAnbLeaves}).

\begin{figure}[H]
    \centering
    \begin{subfigure}{0.48\linewidth}
        \centering
        \includegraphics[width=\linewidth]{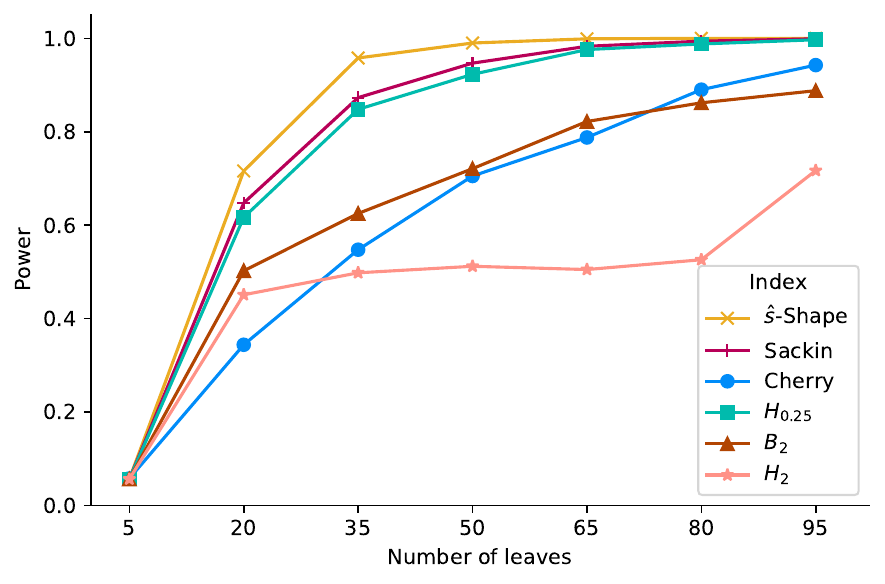}
        \caption{The null model is Yule and the alternative is PDA.}
        \label{YulevsPDAnbLeaves}
        \end{subfigure} \hfill
    \begin{subfigure}{0.48\linewidth}
    \centering
        \includegraphics[width=\linewidth]{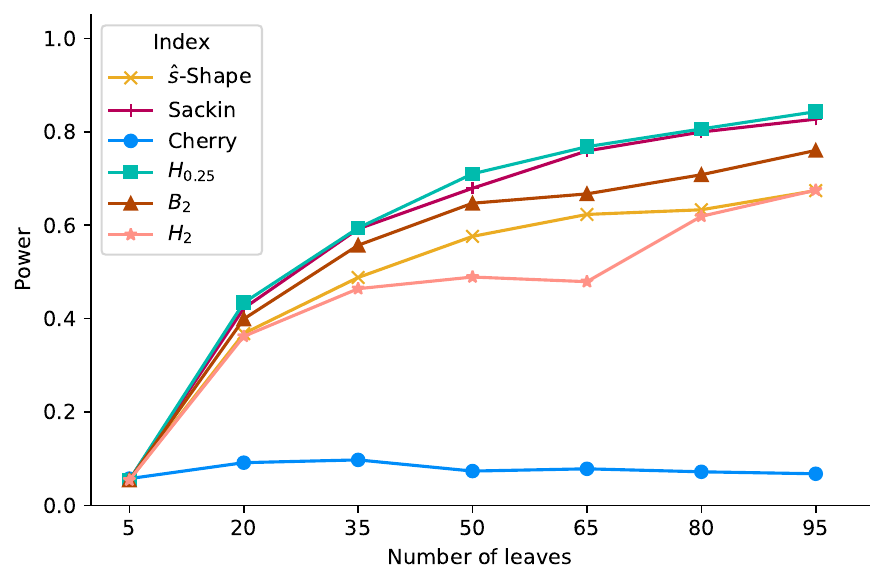}
        \caption{The null model is Yule and the alternative is Simple Brownian (with parameter $1$).}
    \end{subfigure}
    \caption{Comparison between the power of different indices as a function of the number of leaves.}
    \label{FunctionOfLeaves}
\end{figure}

To end this study, let us take an interest in the power of the $H_\alpha$
indices as a function of $\alpha$.
In many situations, the corresponding curve is non-increasing in $\alpha$ (see
\Cref{decreaseInAlpha}) meaning that in such cases the power of the $H_\alpha$
indices is always lower than the power of the Sackin index (but, as we saw,
comes close for $\alpha$ small enough).
For other models however, we observe a slight improvement up to some $\alpha$
between 0 and 1 before the decrease. In such situations, the maximum is represented by an orange dot (see \Cref{increaseInAlpha,OptimalAlphaNotConstant}).
For now, we cannot predict which pair of models falls into one situation or the
other nor how to estimate the value of the optimal $\alpha$ parameter.
Moreover, this optimal value can also depend on the number of leaves we
consider (see \Cref{OptimalAlphaNotConstant}).
Finally, to be exhaustive we also observe on rare occasions (see e.g
\Cref{OptimalAlpha2}) that the power curve can be non-decreasing in $\alpha$.
Note that in this type of situation, the power of any index is relatively low but the $H_\alpha$ indices seem to perform best for larger $\alpha$ (see for instance $H_2$ in \Cref{SBSvsMundo}).

\begin{figure}[H]
    \centering
     \begin{subfigure}{0.48\linewidth}
        \centering
        \includegraphics[width = \linewidth]{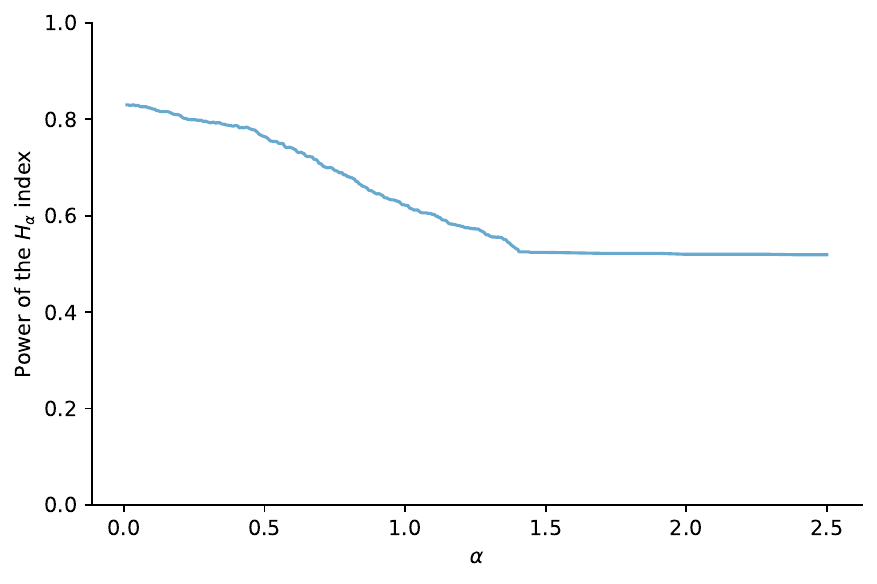}
        \caption{The null model is Yule and the alternative is PDA. The number of leaves of each simulated tree is 30.}
        \label{decreaseInAlpha}
    \end{subfigure} \hfill
    \begin{subfigure}{0.48\linewidth}
        \centering
        \includegraphics[width=\linewidth]{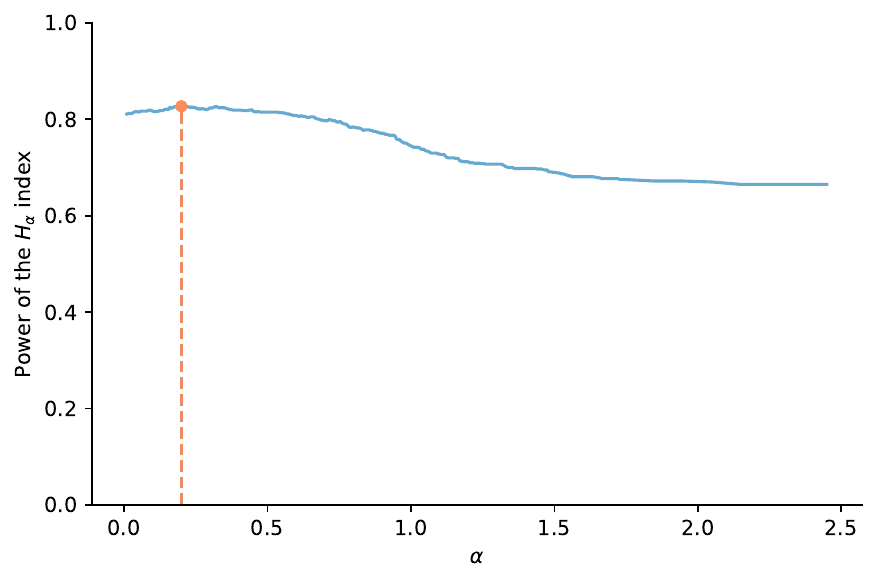}
        \caption{The null model is Yule and the alternative is Simple Brownian (with parameter 1). The number of leaves of each simulated tree is 100.}
        \label{increaseInAlpha}
    \end{subfigure}
    \caption{A case where the power decreases in $\alpha$ and a case where it first increases.}
    \label{CompAlpha}
\end{figure} 
\begin{figure}[H]
    \centering
    \begin{subfigure}{0.48\linewidth}
    \centering
        \includegraphics[width=\linewidth]{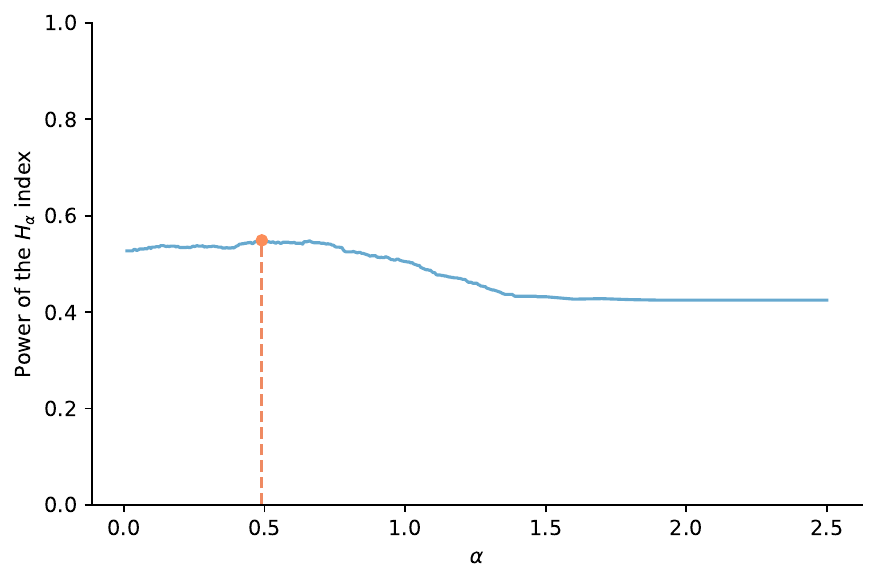}
        \caption{The null model is Inherited fertility and the alternative is Simple Brownian (both parameters equal~1). The number of leaves of each simulated tree is 30.}
    \end{subfigure} \hfill
    \begin{subfigure}{0.48\linewidth}
        \centering
        \includegraphics[width=\linewidth]{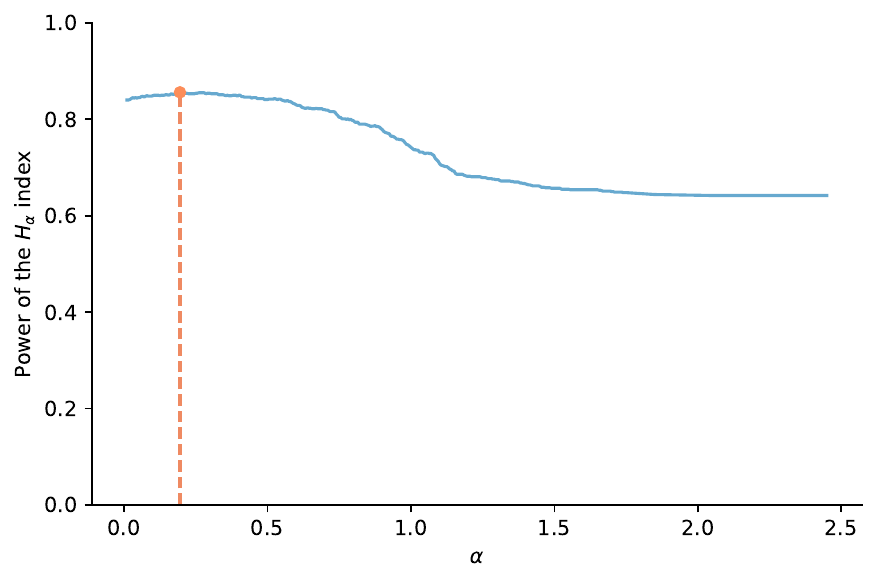}
        \caption{Same models but the simulated trees now have 100 leaves.}
    \end{subfigure}
    \caption{Some cases where the optimal $\alpha$ is not constant for the number of leaves.}
    \label{OptimalAlphaNotConstant}
\end{figure}
\begin{figure}[H]
    \centering
    \includegraphics[width=0.5\linewidth]{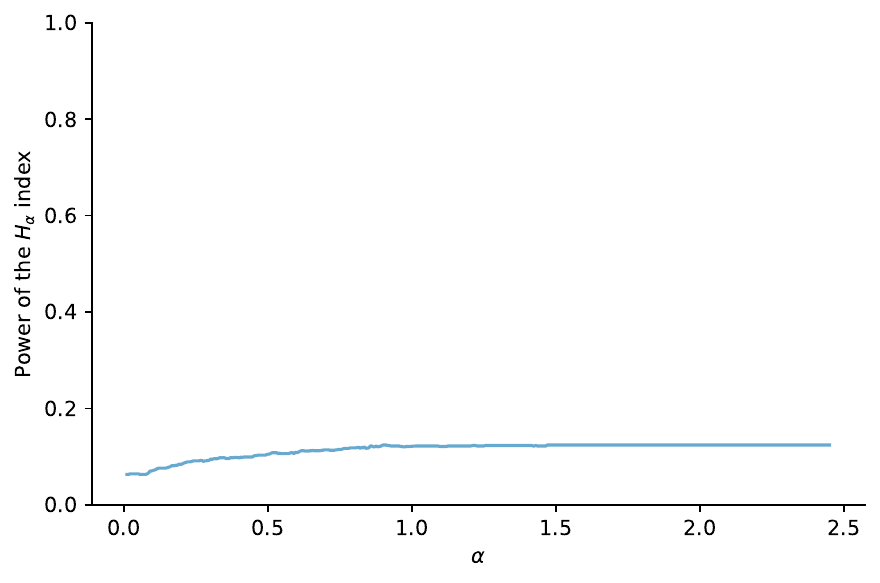}
    \caption{The null model is $\beta$-splitting with $\beta =-1$ and the alternative is Simple Brownian (with parameter 1). The number of leaves of each simulated tree is 30. The power increases with $\alpha$.}
    \label{OptimalAlpha2}
\end{figure}

To summarize, the results obtained here are in the continuity of the
conclusions of Kersting et al.\ in \cite{kerstingTreeBalancePhylogenetic2025}:
it seems that there is no ''best'' index generally speaking and that
performance depends on the choice of the null and alternative models.
Even the Cherry index, originally chosen for its rather poor performances, can
outperform every other index in some situations (e.g.\ with null model \emph{Simple
Brownian} and alternative models \emph{Linear Brownian} or \emph{Alternative
birth-death}).
Nevertheless, some indices perform better than other on average and it appears
that choosing $\alpha$ around 0.25 often produces near optimal results
regarding the $H_\alpha$ indices. With this choice of parameter, the power of
$H_{\alpha}$ is on average close to the power of the Sackin index  but usually
lower than the one of the $\hat{s}$-Shape statistic; although once again, in
some situations $H_{0.25}$ can outperform $\hat{s}$-Shape (for instance in the
comparison \emph{Yule vs Simple Brownian}).

Note that since we adopted the methodology presented in
\cite{kerstingTreeBalancePhylogenetic2025} we compared toy models between them
and never with real data.
However, the many differences among available data
sets seem to affect quite largely and for now in an unpredictable manner, the
indices' power (see \cite{bienvenuRevisitingShaoSokals2021}).
For instance the $B_2$ index seems to be one of the best index in this type of
situation (see e.g.\ \cite{khuranaLimitsConstantrateBirth2024,
kerstingTreeBalancePhylogenetic2025}).
    
To go one step further and differentiate at best trees produced by different
models, one could consider pairs of indices.
In this scenario, the power of a pair could be obtained by looking at $2D$
acceptance regions as in
\cite{bienvenuRevisitingShaoSokals2021}.
Thus, it could be interesting to see by carrying out a systematic study as in
\cite{kerstingTreeBalancePhylogenetic2025} which pair of indices presents the
best power on average, but this goes beyond our scope.

\paragraph{Acknowledgments}
We acknowledge funding through the ANR project GARP (ANR-24-CE40-7154).

\newpage

\phantomsection{}
\addcontentsline{toc}{section}{References}
\printbibliography[title={REFERENCES}]

\newpage

\appendix
\appendixpage
\addappheadtotoc
\crefalias{section}{appendix}
\crefalias{subsection}{appendix}

\section{The \texorpdfstring{$\boldsymbol{H_\alpha}$}{H-alpha} index of infinite phylogenetic networks}
\label{AppInfiniteNetworks}

The goal of this appendix is to extend the definition of the $H_\alpha$
index to infinite networks.
This extension is necessary for the study of
Galton--Watson trees and their blowups using local limits.
It involves two steps: first, defining a suitable
notion of boundary of a phylogenetic network; and then
showing that the simple random walk induces a probability
distribution on this boundary.
The $H_\alpha$ index of a (possibly infinite) network can then be defined as
the structural $\alpha$-entropy of this probability distribution.
In order to make this article self-contained (and because many sources
only define the structural $\alpha$-entropy for probability distributions over
a finite set), we start by recalling the definition of the structural
$\alpha$-entropy and the properties that are relevant to this article.

\subsection[Properties of the structural \texorpdfstring{$\alpha$}{alpha}-entropy]{%
Properties of the structural \texorpdfstring{$\boldsymbol{\alpha}$}{alpha}-entropy} \label{secStructuralEntropy}

The \textit{structural $\alpha$-entropy} was first
introduced in the context of information theory by Havrda and Charvát
\cite{havrdaQuantificationMethodClassification1967a}.
For a discrete probability distribution $\{p_i\}$
(meaning that $p_i \in [0; 1]$ and $\sum_i p_i = 1$), it is defined as  
\begin{equation}
    H_\alpha(\{p_i\}) =
    \frac{1}{1-2^{1-\alpha}} \left( 1 - \sum_i p_i^\alpha \right).
    \label{originalentropy}
\end{equation}
To show that this definition admits a natural extension to arbitrary
probability measures, we need to introduce the notion of $\alpha$-entropy
with respect to a countable partition.

\begin{Defi} \label{defentropy}
Let $(\Omega, \mathcal{A}, \mu)$ be a probability space, and let $\pi$ be a
countable measurable partition of $\Omega$.
For any positive real number $\alpha \neq 1$, the \emph{$\alpha$-entropy of
$\mu$ with respect to $\pi$} is defined as
\begin{equation*}
H_\alpha(\mu ~|~ \pi) = \frac{1}{1-2^{1-\alpha}} 
  \left(1 - \sum_{A \in \pi} \mu(A)^{\alpha}\right).
\end{equation*}
The function $\alpha \mapsto H_\alpha(\mu ~|~ \pi)$ can be
continuously extended at $\alpha = 1$, where it corresponds to the
Shannon entropy of $\mu$ with respect to $\pi$, and so we extend it by
continuity and set
\begin{equation*}
H_1(\mu~|~\pi) = -\sum_{A \in \pi} \mu(A) \log_2 \mu(A). \qedhere
\end{equation*}
\end{Defi}

Definition~\ref{defentropy} has a natural extension
to arbitrary measurable partitions.
Recall that a partition $\pi$ is said to be \emph{finer} than a
partition $\pi'$ if for all $B \in \pi$ there exists $B' \in \pi'$ such that $B
\subset B'$.
We write $\pi \preccurlyeq \pi'$ to indicate that the partition $\pi$
is finer than the partition $\pi'$.

\begin{Defi} \label{defentropy02}
Let $(\Omega, \mathcal{A}, \mu)$ be a probability space, and let $\pi$ be a
measurable partition of $\Omega$. The $\alpha$-entropy of $\mu$ with respect
to $\pi$ is
\begin{equation*}
  H_\alpha(\mu~|~ \pi) = 
  \sup\{H_\alpha(\mu~|~ \pi') : \pi'
  \text{ countable measurable partition of } 
  \Omega \text{ s.t. } \pi \preccurlyeq \pi' \}.
\end{equation*}
The $\alpha$-entropy of the probability distribution $\mu$ is then defined as 
\begin{equation*}
  H_\alpha(\mu) = 
  \sup \{ H_\alpha (\mu~|~ \pi) : \pi 
  \text{ measurable partition of } \Omega \}. \qedhere
\end{equation*}
\end{Defi}

Definition~\ref{defentropy02} is indeed a legitimate extension of
the structural $\alpha$-entropy of a discrete probability distribution, because 
(1) if the partition is countable, this definition is equivalent to
\Cref{defentropy}; and (2) in view of
Proposition~\ref{monotonicity} below, if the support of the probability distribution
is countable, it is equivalent to Eq.~\eqref{originalentropy}.

\begin{Prop}[Monotonicity of the structural $\alpha$-entropy] \label{monotonicity}
Let $\pi$ and $\pi'$ be two countable measurable partitions. 
If $\pi$ is finer than $\pi'$, then
$H_\alpha(\mu~|~\pi)~\geqslant~H_\alpha(\mu~|~\pi')$.
\end{Prop}

\begin{proof}
Decompose each block of $\pi'$ into blocks of $\pi$; the result then
follows from Jensen's inequality applied to $x \mapsto x^\alpha$, combined
with the sign of the normalizing factor $1 - 2^{1-\alpha}$.
\end{proof}

Although Definition~\ref{defentropy02} presents the natural way to extend
the definition of the structural $\alpha$-entropy to arbitrary probability
distributions, in practice this definition takes a much simpler form, as the
next proposition shows.

\begin{Prop} \label{atomicDecomposition}
Let $(\Omega, \mathcal{A}, \mu)$ be a probability space. Then,
\begin{enumerate}[label=(\roman*)]
\item If $\alpha > 1$, writing $\Omega_\infty$ for the set of atoms of $\mu$,
  we have
\begin{equation*}
  H_\alpha(\mu) = 
  \frac{1}{1-2^{1-\alpha}} \left( 1 - \sum_{\omega \in \Omega_\infty}
  \mu(\omega)^\alpha \right) .
\end{equation*}
\item If $0 < \alpha \leqslant 1$ and $\mu$ has a non-atomic part
  of positive mass, then $H_\alpha(\mu) = + \infty$. \qedhere
\end{enumerate}
\end{Prop}

\begin{Proo}
  For the first point, by the decomposition of a probability measure
  into its atomic and diffuse parts, we may write
  $\mu = \nu_{\text{diffuse}} + \nu_{\text{atoms}}$, where
  $\nu_{\text{atoms}} = \sum_{\omega \in \Omega_\infty} \mu(\omega)
  \delta_\omega $ and $\nu_{\text{diffuse}}$ has no $\mu$-atoms. Note
  that $\Omega_\infty$ is a countable (possibly finite, possibly
  empty) set. Next, consider a partition~$\pi_n$ of $\Omega$ whose
  restriction to $\Omega_\infty$ is the partition into singletons, and
  whose restriction to $\Omega \setminus \Omega_\infty$ consists of
  $n$ measurable subsets $A_1, \cdots, A_n$ such that, for all
  $i \leqslant n$,
  $\mu(A_i) = \nu_{\text{diffuse}}(A_i) \leqslant 1/n$ -- note that it
  is possible to find such a family by Lyapunov's theorem on nonatomic
  measures (see e.g.~\cite{Art90}). This partition $\pi_n$ yields
\begin{equation*}
  H_\alpha(\mu~|~ \pi_n) \geqslant 
  \frac{1}{1-2^{1-\alpha}} \left(1- \left( \sum_{\omega \in \Omega_\infty} 
  \!\! \mu(\omega)^\alpha + n^{1- \alpha} \right) \right).
\end{equation*}
Since $\alpha > 1$, this implies 
\begin{equation*}
  H_\alpha(\mu) \geqslant \frac{1}{1-2^{1-\alpha}}
  \left( 1 - \sum_{\omega \in \Omega_\infty}\!\! \mu(\omega)^\alpha \right).
\end{equation*}
Moreover, $f_\alpha : x \mapsto x^\alpha$ is a convex function
such that $f_\alpha(0) = 0$, so it is superadditive on the positive
reals.
Thus, for every measurable subset $A \in \mathcal{A}$,
\begin{equation*}
  \mu(A)^\alpha \geqslant \left( \sum_{\omega \in A \cap \Omega_\infty} \!\!\!\!
  \mu(\omega) \right)^\alpha \geqslant \sum_{\omega \in A \cap \Omega_\infty}
  \!\!\!\!\! \mu(\omega)^\alpha,
\end{equation*}
and, therefore, for all measurable partition $\pi$, 
\begin{equation*}
  H_\alpha(\mu~|~ \pi) \leqslant \frac{1}{1-2^{1-\alpha}} \left( 1 -
  \sum_{\omega \in \Omega_\infty}\!\! \mu(\omega)^\alpha \right),
\end{equation*}
hence the equality.

The second point is a straightforward adaptation of
\cite[Prop.~A.11]{bienvenu$B_2$IndexGalled2024}, which
covers the case $\alpha = 1$. Let $c>0$ be the mass of the non-atomic
part. Using again Lyapunov's theorem, for each $n\geq 1$, partition
$\Omega \setminus \Omega_\infty$ into $n$ measurable subsets
$A_1,\ldots,A_n$ of equal mass $c/n$. Then consider the partition
$\pi_n=\{A_1,\ldots,A_n,\Omega_\infty\}$. For $\alpha<1$,
\[
  H_\alpha(\mu\mid\pi_n) =
  \frac{\left(c^\alpha n^{1-\alpha}+(1-c)^\alpha\right) -
    1}{2^{1-\alpha} - 1}.
\]
Since $1-\alpha>0$, this tends to $+\infty$ as $n\to\infty$, and this
concludes the proof.
\end{Proo}

We now list several properties of the structural $\alpha$-entropy that
are used throughout this document.
The next proposition is the key to the tractability of $H_\alpha$.

\begin{Prop} \label{graftinggen}
Let $(\Omega, \mathcal{A}, \mu)$ be a probability space, and let $\pi$ be a
measurable partition of $\Omega$. Assume that $\pi'$ is obtained from $\pi$
by fragmenting one of its blocks $B$ such that $\mu(B) > 0$, and let $\pi_B$
denote the corresponding partition of $B$. Then, 
\begin{equation*}
H_\alpha(\mu~|~\pi') = H_\alpha(\mu~|~\pi) + \mu(B)^\alpha H_\alpha(\mu_B~|~\pi_B),
\end{equation*}
where $\mu_B$ denotes the conditional probability distribution induced on $B$ by $\mu$.
\end{Prop}

\begin{Proo}
The proof only involves straightforward calculations, but due to the importance
of the result for our study, we give the details.
First, assume that $\pi'$ is countable.
For $\alpha\neq1$, by definition of $\pi'$,
\begin{align*}
  H_\alpha(\mu~|~\pi')
  &= \frac{1}{1-2^{1-\alpha}}
  \left( 1-\sum_{A\in\pi\setminus\{B\}}\mu(A)^\alpha-\sum_{C\in\pi_B}\mu(C)^\alpha \right) \\
  &= H_\alpha(\mu~|~\pi) + \frac{1}{1-2^{1-\alpha}}
  \left( \mu(B)^\alpha-\sum_{C\in\pi_B}\mu(C)^\alpha \right) \\
  &= H_\alpha(\mu~|~\pi)
  + \frac{\mu(B)^\alpha}{1-2^{1-\alpha}}
  \left( 1-\sum_{C\in\pi_B}\left(\frac{\mu(C)}{\mu(B)}\right)^\alpha \right) \\
  &= H_\alpha(\mu~|~\pi) +
  \frac{\mu(B)^\alpha}{1-2^{1-\alpha}}
  \left( 1-\sum_{C\in\pi_B}\mu_B(C)^\alpha \right) \\
  &= H_\alpha(\mu~|~\pi)+\mu(B)^\alpha H_\alpha(\mu_B~|~\pi_B).
\end{align*}
The case $\alpha=1$ follows by continuity.
Finally, the equality for arbitrary measurable partitions
follows by taking the supremum over countable measurable coarsenings in
Definition~\ref{defentropy02}.
\end{Proo}

This proposition shows that the contribution of a refinement of a block $B$ is
weighted by $\mu(B)^\alpha$, so that finer partitions are given relatively more
weight when $\alpha$ is small.
One therefore expects convergence under successive refinements to become harder
to obtain as $\alpha$ approaches 0.
The following proposition formalizes this idea.

\begin{Prop} \label{cvgHalpha}
Let $(\Omega, \mathcal{A}, \mu)$ be a probability space and let $\pi$ be a
countable measurable partition of~$\Omega$. Let $(\mu_i)_{i \geqslant 1}$ be
a sequence of probability measures on $(\Omega, \mathcal{A})$ satisfying: 
\begin{equation*}
\forall A \in \pi, \quad \mu_i(A) \underset{i \rightarrow \infty}{\longrightarrow} \mu(A).
\end{equation*}
Then:
\begin{enumerate}[label = \roman*.]
\item For any $\alpha > 1$, $H_\alpha(\mu_i~|~\pi) \rightarrow H_\alpha(\mu~|~\pi)$.
\item For any  $0 < \alpha < \alpha ' \leqslant 1$, if $H_\alpha(\mu_i~|~\pi) \rightarrow H_\alpha(\mu~|~\pi)$ then $H_{\alpha'}(\mu_i~|~\pi) \rightarrow H_{\alpha'}(\mu~|~\pi)$. \qedhere
\end{enumerate}
\end{Prop}

\begin{Proo}
Let us start with the first point.
Let $\alpha > 1$ and set
$f_\alpha : x \in [0;1] \mapsto \frac{x(1-x^{\alpha -1})}{1-2^{1-\alpha}}$.
Note that $f_\alpha$ is continuous and nonnegative.
Thus, by Fatou's lemma we have
\begin{eqnarray*}
  H_\alpha(\mu~|~\pi) &=& \sum_{A \in \pi} \liminf_i f_\alpha(\mu_i(A)) \\
  &\leqslant & \liminf_i H_\alpha(\mu_i~|~\pi).
\end{eqnarray*} 
Moreover, since $\alpha$ is greater than $1$, $1/ (1-2^{1-\alpha})$ is positive
and thus, once again using Fatou's lemma,
\begin{eqnarray*}
 H_\alpha(\mu~|~\pi) &=& \frac{1}{1-2^{1-\alpha}} \left(1- \sum_{A \in \pi}
  \liminf_i \mu_i(A)^\alpha \right) \\ &\geqslant &
  \limsup_i H_\alpha(\mu_i~|~\pi),
\end{eqnarray*}
Hence the fact that $H_\alpha(\mu_i~|~\pi) \rightarrow H_\alpha(\mu~|~\pi)$. 

Let us now turn to the second point.
Let $0 < \alpha < \alpha' < 1$, and assume that
\begin{equation*}
\frac{1}{1-2^{1-\alpha}} \left( 1- \sum_{A \in \pi} \mu_i(A)^\alpha \right)
\underset{i \rightarrow \infty}{\longrightarrow}
\frac{1}{1-2^{1-\alpha}} \left( 1- \sum_{A \in \pi} \mu(A)^\alpha \right).
\end{equation*}
Since $\pi$ is countable, we can write $\pi = \{A_0, A_1, \cdots \}$. Thus, for
all $i \geqslant 1$, we can write $u_i = (u_i^{(n)})_{n \in \mathds{N}}$
where $u_i^{(n)} = \mu_i(A_n)$ and, similarly, $u = (u^{(n)})_n$ where
$u^{(n)} = \mu(A_n)$. Thus, our assumptions can be re-written as:
\begin{itemize}
\item[\textbullet] $ \displaystyle \sum_{n \geqslant 0} (u_i^{(n)})^\alpha
  \underset{i \rightarrow \infty}{\longrightarrow}
  \sum_{n \geqslant 0} (u^{(n)})^\alpha$;
\item[\textbullet] $\forall n \geqslant 0, \quad u_i^{(n)} \underset{i
  \rightarrow \infty}{\longrightarrow} u^{(n)}$.
\end{itemize}
By Scheffé's lemma, this is equivalent to
\begin{equation*}
\sum_{n \geqslant 0} \big|(u_i^{(n)})^\alpha - (u^{(n)})^\alpha\big|
\underset{i \rightarrow \infty}{\longrightarrow} 0.
\end{equation*}

Now set $\beta = \alpha'/\alpha > 1$. Since
$x \mapsto x^\beta$ is Lipschitz on $[0,1]$,
there exists a constant $L_\beta > 0$ such that for all $x,y \in [0,1]$,
\begin{equation*}
|x^\beta - y^\beta| \leqslant L_\beta |x-y|.
\end{equation*}
Applying this with $x = (u_i^{(n)})^\alpha$ and $y = (u^{(n)})^\alpha$, we obtain
\begin{equation*}
\big|(u_i^{(n)})^{\alpha'} - (u^{(n)})^{\alpha'}\big|
=
\big|((u_i^{(n)})^\alpha)^\beta - ((u^{(n)})^\alpha)^\beta\big|
\leqslant
L_\beta \big|(u_i^{(n)})^\alpha - (u^{(n)})^\alpha\big|.
\end{equation*}
Summing over $n$ gives
\begin{equation*}
\sum_{n \geqslant 0} \big|(u_i^{(n)})^{\alpha'} - (u^{(n)})^{\alpha'}\big|
\leqslant
L_\beta
\sum_{n \geqslant 0} \big|(u_i^{(n)})^\alpha - (u^{(n)})^\alpha\big|
\underset{i \rightarrow \infty}{\longrightarrow} 0.
\end{equation*}
Thus,
\begin{equation*}
\sum_{n \geqslant 0} (u_i^{(n)})^{\alpha'}
\underset{i \rightarrow \infty}{\longrightarrow}
\sum_{n \geqslant 0} (u^{(n)})^{\alpha'},
\end{equation*}
which implies that
\begin{equation*}
H_{\alpha'}(\mu_i~|~\pi)
\underset{i \rightarrow \infty}{\longrightarrow}
H_{\alpha'}(\mu~|~\pi).
\end{equation*}

Now, all that is left to prove is that if $\alpha < 1$ and
$H_\alpha(\mu_i~|~\pi) \rightarrow H_\alpha(\mu~|~\pi)$, then
$H_1(\mu_i~|~\pi) \rightarrow H_1(\mu~|~\pi)$.
Reasoning as before, define
\begin{equation*}
\phi_\alpha : t \in [0,1] \mapsto
- t^{1/\alpha} \log_2\!\big(t^{1/\alpha}\big),
\end{equation*}
with the convention $\phi_\alpha(0)=0$.
Since $0<\alpha<1$, the function $\phi_\alpha$  is Lipschitz on $[0,1]$, so
there exists $L_\alpha>0$ such that for all $s,t \in [0,1]$,
\begin{equation*}
|\phi_\alpha(s)-\phi_\alpha(t)| \leqslant L_\alpha |s-t|.
\end{equation*}
Now note that $- x \log_2(x) = \phi_\alpha(x^\alpha)$ for all $x \in [0,1]$.
Therefore, for all $n \geqslant 0$,
\begin{equation*}
\big|-u_i^{(n)} \log_2(u_i^{(n)}) + u^{(n)} \log_2(u^{(n)})\big|
=
\big|\phi_\alpha((u_i^{(n)})^\alpha)-\phi_\alpha((u^{(n)})^\alpha)\big|
\leqslant
L_\alpha \big|(u_i^{(n)})^\alpha-(u^{(n)})^\alpha\big|.
\end{equation*}
Summing over $n$, we get
\begin{equation*}
\sum_{n \geqslant 0}
\big|-u_i^{(n)} \log_2(u_i^{(n)}) + u^{(n)} \log_2(u^{(n)})\big|
\leqslant
L_\alpha
\sum_{n \geqslant 0} \big|(u_i^{(n)})^\alpha-(u^{(n)})^\alpha\big|
\underset{i \rightarrow \infty}{\longrightarrow} 0.
\end{equation*}
Hence $H_1(\mu_i~|~\pi) \to H_1(\mu~|~\pi)$ as $i\to\infty$,
concluding the proof.
\end{Proo}

We end this section by stating a result about the convergence of refining
sequences of partitions.
Recall that, given a sequence $(\pi_n)_{n \geqslant 0}$ of partitions of a set
$E$, we say that $\pi_n$ converges to $\pi$ if, for
all $x, y \in E$, there exists $N$ such that for all $n \geq N$,
$x \sim_{\pi_n} y$ if and only if $x \sim_\pi y$ (where $x \sim_\pi
y$ indicates that $x$ and $y$ are in the same block in $\pi$).

\begin{Lemm} \label{cvgSingleton}
Let $(\pi_n)$ be a sequence of measurable partitions on a probability
space $(\Omega, \mathcal{A}, \mu)$. If $(\pi_n)$ converges to a partition
$\pi$ such that $\pi \preccurlyeq \pi_n$ for all $n \geqslant 0$, then $\pi$
is measurable and 
\begin{equation*}
H_\alpha(\mu~|~\pi_n) \underset{n \rightarrow \infty}{\longrightarrow} H_\alpha(\mu~|~\pi).
\end{equation*}
In particular, if $\pi$ is the partition into singletons, the limit is $H_\alpha(\mu)$.
\end{Lemm}

This lemma is a straightforward adaptation of
\cite[Lemma A.14]{bienvenu$B_2$IndexGalled2024}, therefore we omit the proof.

\subsection{The boundary and \texorpdfstring{$H_\alpha$}{H-alpha} index of an infinite phylogenetic network}

We now return to the definition of the $H_\alpha$ index of an infinite
phylogenetic network.
To do so, the first step is to define a suitable notion of \textit{boundary}
for phylogenetic networks.
This is done in~\cite{bienvenu$B_2$IndexGalled2024}, to which we refer
for a formal definition and technical details; here we only recall the idea
behind the definition.
To make things more concrete, consider the infinite phylogenetic
network~$G$ represented in Fig.~\ref{figInfinitePitchfork}.
Since the simple random walk ``escapes to infinity'' along either the left or
the right infinite path (call such infinite paths \textit{rays}), the
$H_\alpha$ index of this network should be the same as that of the cherry --
namely the structural $\alpha$-entropy of the probability distribution
$(1/2, 1/2)$.
In this simple example, the boundary is clear: it should consist of
two ``points at infinity'' that correspond to each of the two rays.

\begin{figure}[H]
    \centering
    \begin{tikzpicture}[main/.style = {draw, circle, minimum size = 7pt, inner sep =1pt}]
    \tikzstyle{dot}=[draw,circle,gray,
                  top color= white, text=gray,minimum size=7pt, inner sep = 1]
    \tikzstyle{phantom}=[opacity = 0, minimum size = 7pt]
    \node[main, label ={[font=\footnotesize] \o} ] (root) at (0, 0) {};
    \node[main, label ={[font=\scriptsize]left:1}] (1) at (-1, -1) {};
    \node[main, label ={[font=\scriptsize]right:1'}] (1bis) at (1, -1) {};
    \node[main, label ={[font=\scriptsize]left:2}] (2) at (-1, -2) {};
    \node[main, label ={[font=\scriptsize]right:2'}] (2bis) at (1, -2) {};
    \node[phantom] (n) at (-1, -4) {};
    \node[phantom] (nbis) at (1, -4) {};
    \draw[->] (root) -- (1);
    \draw[->] (root) -- (1bis);
    \draw[->] (1) -- (2);
    \draw[->] (1bis) -- (2bis);
    \draw[dashed, ->] (2) -- (n);
    \draw[dashed, ->] (2bis) -- (nbis);
\end{tikzpicture}
    \caption{An infinite phylogenetic network $G$ with no leaves and
    two infinite paths along which the simple random walk can escape to
    infinity.}
    \label{figInfinitePitchfork}
\end{figure}
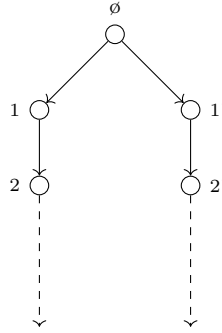

Now suppose we add a directed edge between vertex 1 and 1':
the points at infinity where the simple random walk can
escape do not change, so neither should the boundary
(although the probability distribution on this boundary would differ).
More generally, if two rays are such that the simple random walk
can always switch from one to the other, no matter how far away it has gone
on either of them, then those two rays correspond to the same
point at infinity and should be identified.

Formally, we quotient the space of rays by identifying any two rays that are
intersected infinitely many times by a third (not necessarily different) ray.
Together with the leaves, these equivalent classes of rays form
the \emph{boundary} $\partial G$ of the network.
Moreover, $\partial G$ can be embedded in a suitable compact metric space,
and the limit of the simple random walk is a
well-defined random variable $X_\infty$ taking values in $\partial G$.
In consequence, the $H_\alpha$ index of an arbitrary
phylogenetic network can be defined as the structural $\alpha$-entropy
of the distribution of $X_\infty$.
Finally, because -- by the same argument as for the $B_2$ index --
the resulting $H_\alpha$ index can be expressed as a
pointwise limit of measurable functions on the space 
$\mathbb{G}$ of phylogenetic networks endowed with the local distance,
$H_\alpha$ is measurable from $\mathbb{G}$ to $\mathbb{R}\cup\{+\infty\}$.
As a result, if $G$ is a random phylogenetic network, $H_\alpha(G)$ is a
well-defined random variable.
Again, we refer to \cite[App.~A.2--4]{bienvenu$B_2$IndexGalled2024}
for technical details.

\section{Range of the \texorpdfstring{$\boldsymbol{H_\alpha}$}{H-alpha} index: proofs} \label{appRange}

This appendix is devoted to the proofs of the results listed in \Cref{secRange}.
\Cref{generalBounds} is a well-known result for the structural
$\alpha$-entropy (see
e.g.~\cite[Thm.~2]{havrdaQuantificationMethodClassification1967a}) -- only
rephrased in the context of the probability distribution
$(p_\ell)_{\ell\in\mathcal{L}_G}$ induced by the simple random walk.
The other results in that section are adaptations and generalizations of the
results proved for the $B_2$ index in \cite{bienvenuRevisitingShaoSokals2021}.
Here, we therefore summarize the main ideas of the proof and emphasize
the key differences that arise when replacing the Shannon entropy by
the structural $\alpha$-entropy for $\alpha \neq 1$.

\subsection{Preliminary lemmas} \label{appPrelim}

For any positive $\alpha \neq 1$, set
$f_\alpha: x \mapsto \frac{1}{1-2^{1-\alpha}} x(1-x^{\alpha-1})$ and note
that for any finite phylogenetic network $G$, we have
$H_\alpha(G) = \sum_{\ell \in \mathcal{L}_G} f_\alpha(p_\ell)$.
Also note that $f_\alpha$ is strictly concave for any $\alpha$:
this will be the key to the generalization of
\cite[Prop.~1.13, Lem.~2.9 \&~2.12]{bienvenuRevisitingShaoSokals2021}.

\begin{Lemm} \label{sign}
Let $G$ and $G'$ be two phylogenetic networks that have the
same leaf set and, except for two fixed leaves $\ell_1$ and $\ell_2$, 
satisfy $p_\ell = p'_\ell$ for any leaf $\ell\notin\{\ell_1, \ell_2\}$,
where $p_\ell$ (resp.\ $p_\ell'$) is the probability that the simple random
walk on $G$ (resp.\ $G'$) reaches leaf $\ell$).
Then, 
    \begin{equation*}
        \mathrm{sgn} ( H_\alpha(G') - H_\alpha(G)) = \mathrm{sgn}( (p_{\ell_1}-p_{\ell_1}')(p_{\ell_1}-p_{\ell_2}')),
    \end{equation*}
where $\mathrm{sgn}(x) \in \{-1; 0; 1\}$ denotes the sign of $x$.
\end{Lemm}

\begin{Proo}
To obtain this result, it suffices to note that, as in
\cite{bienvenuRevisitingShaoSokals2021}:
\begin{equation*}
     H_\alpha(G') - H_\alpha(G) = (f_\alpha(p_{\ell_1} + \Delta) - f_\alpha(p_{\ell_1})) - (f_\alpha(p_{\ell_2}' + \Delta) - f_\alpha(p_{\ell_2}')),
\end{equation*}
where $\Delta = p_{\ell_1}' - p_{\ell_1} = p_{\ell_2} - p_{\ell_2}'$,
and then use the strict concavity of $f_\alpha$. 
\end{Proo}

\begin{Lemm} \label{constructionMin}
Let $N$ be a phylogenetic network and let $u, v$ and $w$ be three vertices of
$N$ such that $u$ is a parent of $v$ and $w$, and neither of these two
vertices $v$ and $w$ is an ancestor of the other.
Denote by $N'$ and $N''$ the networks
obtained by adding an edge between $\overset{\rightarrow}{uv}$ and
$\overset{\rightarrow}{uw}$, in one direction for $N'$ and in the other one
for $N''$, as shown in Fig.~\ref{procedure}.
Then, 
\begin{equation*}
H_\alpha(N') + H_\alpha(N'') \leqslant 2 H_\alpha(N).
\end{equation*}
In particular, $\min \{H_\alpha(N') , H_\alpha(N'') \} \leqslant H_\alpha(N).$
Moreover, these inequalities are strict if and only if there exists a leaf
$\ell$ such that $P_v(\ell) \neq P_w(\ell)$, where $P_x(\ell)$ denotes the
probability that the simple random walk started from $x$ ends in $\ell$.
\end{Lemm}

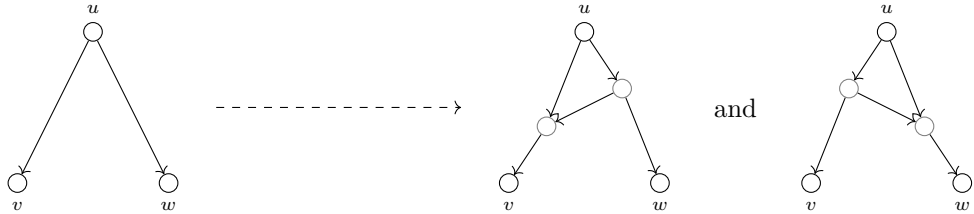
\begin{figure}[H]
    \centering
    \begin{tikzpicture}[main/.style = {draw, circle, minimum size = 7pt, inner sep =1pt}]
    \tikzstyle{dot}=[draw,circle,gray,
                  top color= white, text=gray,minimum size=7pt, inner sep = 1]
    \tikzstyle{phantom}=[opacity = 0, minimum size = 7pt]
    \node[main, label ={[font=\scriptsize] $u$}] (u) at (-6.5, 0) {};
    \node[main, label ={[font=\scriptsize]below: $v$}] (v) at (-7.5, -2) {};
    \node[main, label ={[font=\scriptsize]below: $w$}] (w) at (-5.5, -2) {};
    \draw[->] (u) -- (v);
    \draw[->] (u) -- (w);
    \node[main, label ={[font=\scriptsize] $u$}] (u1) at (0, 0) {};
    \node[main, label ={[font=\scriptsize]below: $v$}] (v1) at (-1, -2) {};
    \node[main, label ={[font=\scriptsize]below: $w$}] (w1) at (1, -2) {};
    \node[dot] (a) at (-0.5, -1.25) {};
    \node[dot] (b) at (0.5, -0.75) {};
    \draw[->] (u1) -- (a);
    \draw[->] (u1) -- (b);
    \draw[->] (a) -- (v1);
    \draw[->] (b) -- (w1);
    \draw[->] (b) -- (a);
    \node[main, label ={[font=\scriptsize] $u$}] (u2) at (4, 0) {};
    \node[main, label ={[font=\scriptsize]below: $v$}] (v2) at (3, -2) {};
    \node[main, label ={[font=\scriptsize]below: $w$}] (w2) at (5, -2) {};
    \node[dot] (c) at (3.5, -0.75) {};
    \node[dot] (d) at (4.5, -1.25) {};
    \draw[->] (u2) -- (c);
    \draw[->] (u2) -- (d);
    \draw[->] (c) -- (v2);
    \draw[->] (d) -- (w2);
    \draw[->] (c) -- (d);
    \node[phantom] (depart) at (-5, -1) {};
    \node[phantom] (arrivee) at (-1.5, -1) {};
    \draw[dashed, ->] (depart) -- (arrivee);
    \node at (2, -1) {and};
\end{tikzpicture}
\caption{Insertion of an edge from
$\vec{uw}$ to $\vec{uv}$ (left) and from $\vec{uv}$ to $\vec{uw}$ (right).}
\label{procedure}
\end{figure}

\begin{Proo}
Letting $p_\ell$, $p'_\ell$ and $p_\ell''$ be the probabilities of reaching a
fixed leaf $\ell$ in $N$, $N'$ and $N''$, respectively, we have
\begin{equation*}
    p_\ell' + p_\ell'' = 2 p_\ell.
\end{equation*}
Using the concavity of $f_\alpha$ and summing these inequalities over the set
of leaves yields
\begin{equation*}
    H_\alpha(N') + H_\alpha(N'') \leqslant 2 H_\alpha(N),
\end{equation*}
and the equality case follows from the strict concavity of $f_\alpha$.
\end{Proo}

Finally, the next result is a direct adaptation of
\cite[Lem.~2.12]{bienvenuRevisitingShaoSokals2021}; since the proof is
exactly the same, we do not recall it.

\begin{Lemm}
Let $N$ be a tree-child network with $n$ leaves. If $N$ has a reticulation whose child is not a leaf, then there exists a tree-child network $N^\star$ with $n$ leaves such that $H_\alpha(N^\star) < H_\alpha(N)$.
\end{Lemm}

\subsection{Proofs of \Cref{binaryBound,,boundsTemporal}}

We begin with the proof of \Cref{binaryBound}, which we recall here
for convenience.

\begin{Theonn}
    Let $T$ be a rooted binary tree with $n$ leaves. Then,
    \begin{equation*}
        H_\alpha(\Cat(n)) \leqslant H_\alpha(T) \leqslant H_\alpha(\mathrm{CB}(\lfloor \log_2(n) \rfloor)) + \left(n-2^{\lfloor \log_2(n)\rfloor} \right) \cdot 2^{-\alpha \lfloor \log_2(n) \rfloor},
    \end{equation*}
    where we recall that $\Cat(n)$ denotes the caterpillar tree with $n$ leaves
    and $\mathrm{CB}(h)$ denotes the complete binary tree with height $h$.
    Moreover, for $\alpha > 0$ these bounds are sharp and:
    \begin{enumerate}[label=\roman*.]
        \item The caterpillar tree $\Cat(n)$ is the only rooted binary tree
          with $n$ leaves that minimizes $H_\alpha$.
        \item The rooted binary trees with $n$ leaves that maximize
          $H_\alpha$ are exactly the trees such that the difference
          between the height of any two leaves is at most $1$. \qedhere
    \end{enumerate}
\end{Theonn}

\begin{Proo}
The proof is a direct adaptation of that of
\cite[Thm.~2.3]{bienvenuRevisitingShaoSokals2021}.
Indeed, in a binary tree $T$ with $n$ leaves, consider a cherry (formed by two
leaves and their parent $v$) and a leaf $\ell$ that is not in the cherry.
Denoting by $T'$ the tree resulting from transferring the cherry to $\ell$,
by \Cref{graftingcherry},
\begin{equation*}
     H_\alpha(T') - H_\alpha(T) =
     2^{- \alpha \delta_\ell} - 2^{- \alpha \delta_v}.
\end{equation*}
Since we can turn any binary tree into another by successively transferring
cherries, we deduce that:
\begin{itemize}
    \item Trees minimizing $H_\alpha$ are trees that do not have a leaf
      whose depth is greater than that of the parent of a cherry; only the
      caterpillar tree satisfies that condition.
    \item Trees maximizing $H_\alpha$ are trees that do not have
      a leaf whose depth is smaller than that of the parent of a cherry.
\end{itemize}
Note that trees that match the second condition are trees that are built as the
complete binary tree with $2^{\lfloor \log_2(n) \rfloor}$ leaves where
cherries are grafted on $n - 2^{\lfloor \log_2(n) \rfloor}$ of those leaves.
Thus, \Cref{graftingcherry} yields the computation of the upper bound.
\end{Proo}

We now recall the statement of \Cref{boundsTemporal}.
\begin{Theonn}
    Assume that $\alpha \leqslant 1$. 
    For every temporal tree-child network $N$ with
    $n$ leaves, \begin{equation*} H_\alpha(\Cat(n)) \leqslant H_\alpha(N)
      \leqslant H_\alpha(\mathrm{CB}(\lfloor \log_2(n) \rfloor)) +
      \left(n-2^{\lfloor \log_2(n)\rfloor} \right) \cdot 2^{-\alpha \lfloor
      \log_2(n) \rfloor}. \qedhere \end{equation*}
\end{Theonn}

\begin{Proo}
This proof is a direct adaptation of the proof of
\cite[Thm.~2.7]{bienvenuRevisitingShaoSokals2021}, which covers the
case $\alpha = 1$.
The idea is to show that for any temporal tree-child network $N$ one can find
two binary trees $T'$ and $T''$ such that
\[
  H_\alpha(T') \leqslant H_\alpha(N) \leqslant H_\alpha(T'').
\]
To do so, we remove each reticulation successively, starting from the ones
with maximal temporal labeling and working our way up towards the root.
At each step, there are two ways to remove the reticulation, depending
on which reticulated edge is deleted: these correspond to
two temporal tree-child networks $N'$ and $N''$. To conclude the proof,
it suffices to show that $N'$ and $N''$ can be chosen such
that $H_\alpha(N') \leqslant H_\alpha(N) \leqslant H_\alpha(N'')$.

Let us start by building $N'$.
Let $r$ be a reticulation and denote the siblings of $r$
by $u$ and $v$, and its child by~$w$. Write $N_u$, $N_v$ and $N_w$ for the
trees subtended by $u$, $v$ and $w$.
Finally, let $p_u, p_v$ and $p_w = p_u +p_v$ be the probabilities that the
simple random walk goes through $u$, $v$ and $w$, respectively.
Assume without loss of generality that $p_u \leqslant p_v$.
Also assume without loss of generality that
$H_\alpha(N_w) \leqslant H_\alpha(N_v)$
(if not, swap $N_w$ and $N_v$: by the grafting property, this
yields a network $\tilde{N}$ such that
$H_\alpha(\tilde{N}) \leqslant H_\alpha(N)$, and we can carry on with
$\tilde{N}$ instead of $N$ in the rest of the proof).

Now, build $N'$ from $N$ by: ungrafting $N_v$ and $N_w$; removing the
reticulation; and regrafting $N_v$ and $N_w$.
By \Cref{graftingproperty}, we have
\begin{equation*}
  H_\alpha(N') - H_\alpha(N) = H_\alpha (G') - H_\alpha(G) +
  H_\alpha(N_w) (2^\alpha p_v^\alpha - p_w^\alpha) - H_\alpha(N_v) (p_v^\alpha
  - p_u^\alpha),
\end{equation*}
where $G$ (resp. $G'$) is the subnetwork of $N$ (resp. $N'$) that remains
unchanged under the grafting operations. Write
$\Delta_{\text{remove}} = H_\alpha(G') - H_\alpha(G)$  and
$\Delta_{\text{graft}} =
H_\alpha(N_w) (2^\alpha p_v^\alpha - p_w^\alpha) - H_\alpha(N_v) (p_v^\alpha -
  p_u^\alpha) $.
Since, by \Cref{sign}, {$\Delta_{\text{remove}}
\leqslant 0$}, we need to show that $\Delta_{\text{graft}} \leqslant 0$. 
Now, because $2p_v \geqslant p_w$, we have that
$H_\alpha(N_w)(2^\alpha p_v^\alpha -
p_w^\alpha) \leqslant H_\alpha(N_v)(2^\alpha p_v^\alpha - p_w^\alpha)$.
Thus,
\begin{equation*}
  p_v^\alpha-p_u^\alpha \geqslant 2^\alpha p_v^\alpha -p_w^\alpha
  \;\implies\;
  H_\alpha(N_w)(2^\alpha p_v^\alpha - p_w^\alpha) \leqslant
  H_\alpha(N_v)(p_v^\alpha - p_u^\alpha).
\end{equation*}
Recalling that $p_w = p_u +p_v$, 
\begin{eqnarray*}
  p_v^\alpha - (p_w -p_v)^\alpha \geqslant 2^\alpha p_v^\alpha - p_w^\alpha &
  \iff & p_v^\alpha (1-2^\alpha) \geqslant (p_w-p_v)^\alpha - p_w^\alpha \\&
  \iff & 1-2^\alpha \geqslant \left( \frac{p_w}{p_v} -1 \right)^\alpha -
  \left( \frac{p_w}{p_v} \right)^\alpha \\ & \iff & \left(\frac{p_w}{p_v} -1
  \right)^\alpha - \left(\frac{p_w}{p_v} \right)^\alpha + 2^\alpha -1
  \leqslant 0.
\end{eqnarray*}
Now, let $h_\alpha : x \mapsto (x-1)^\alpha - x^\alpha + 2^\alpha -1$, and note
that $h_\alpha$ is increasing for $\alpha < 1$.
Since $1 < \frac{p_w}{p_v} \leqslant 2$ (because $0<p_v<p_w$ and $p_w
\leqslant 2p_v$) and $h_\alpha(2) =0$, we have that $h_\alpha \leqslant 0$
on $[1; 2]$. This proves that
$\Delta_{\text{graft}} \leqslant 0$, and therefore that
$H_\alpha(N') \leqslant H_\alpha(N)$ when $\alpha < 1$.

The network $N''$ such that $H_\alpha(N) \leqslant H_\alpha(N'')$
is obtained by the same construction,
with the roles of  $u$ and $v$ exchanged.
This concludes the proof.
\end{Proo}

We do not present the proof of \Cref{boundTreeChild}, since it is identical to
that of \cite[Thm.~2.8]{bienvenuRevisitingShaoSokals2021}.

%\begin{Theonn}
%Let $N$ be a tree-child network with $n$ leaves. Then,
%\begin{equation*}
%H_\alpha(N) \geqslant H_\alpha(\FCat(n)),
%\end{equation*}
%where $\FCat(n)$ is the fat caterpillar with $n$ leaves
%(see Section~\ref{basicProperties}).
%\end{Theonn}
%
%\begin{proof}
%We end this section by giving a short explanation of this result. Once again it
%is a straightforward adaptation of . Indeed all the key elements for the
%proof were given in \Cref{appPrelim}. As in
%\cite{bienvenuRevisitingShaoSokals2021} we can show that $N$ minimizes
%$H_\alpha$ under the condition of not having two reticulations $r$ and $r'$
%such that one parent of $r$ is an ancestor of $r'$ and vice versa. Only the Fat
%Caterpillar with $n$ leaves $\FCat(n)$ satisfies this condition, concluding the
%proof.
%\end{proof}

\section{Galton--Watson trees: proofs} \label{appGW}

In this section, we prove the results that were stated without proof
in \Cref{secGW} -- namely, \Cref{VarGW}, \Cref{momentoforderm} and
\Cref{exponentialMoment}.

\subsection[Variance and \texorpdfstring{$m$-th}{m-th} moment of \texorpdfstring{$H_\alpha$}{H-alpha}]{%
Variance and \texorpdfstring{$\boldsymbol m$-th}{m-th} moment of \texorpdfstring{$\boldsymbol H_{\boldsymbol \alpha}$}{H-alpha}}

Let us start with \Cref{VarGW}, whose statement we recall here for convenience. 

\begin{Theonn}
Let $T$ be a Galton--Watson tree with offspring distribution $\xi$, where
$\mathds{P}(\xi = 1) < 1$.
Then, $\Var(H_\alpha(T))$ is finite if and only if
$\mathds{E}[\xi^{1-\alpha}\mathds{1}_{\{\xi \neq 0\}}] < 1$ and
$\mathds{E}[\xi^{2(1-\alpha)}\mathds{1}_{\{\xi \neq 0\}}] < + \infty$.
In this case, 
\begin{equation*}
    \mathrm{Var}(H_\alpha(T)) = 
    \frac{\mathds{P}(\xi = 0)}{(1-2^{1-\alpha})^2
    \left(1-\mathds{E}\left[ \xi^{1-2\alpha}
    \mathds{1}_{\{\xi \neq 0\}} \right] \right)}
    \left(1+ \frac{\left( \mathds{E}\left[ \xi^{2(1-\alpha)}
    \mathds{1}_{\{\xi \neq 0\}} \right] -1 \right)
    \mathds{P}(\xi = 0)}{\left(1- \mathds{E}\left[ \xi^{1-\alpha}
    \mathds{1}_{\{\xi \neq 0\}} \right] \right)^2} \right). \qedhere
\end{equation*}
\end{Theonn}

\begin{Proo}
Recall from Eq.~\eqref{recurrenceGWdistri} that,
letting $\xi$ denote the number of children of the root, one has
\begin{equation*}
  H_\alpha([T]_k) \;\overset{d}{=}\;
  \left(\frac{1}{1-2^{1-\alpha}}(1-\xi^{1-\alpha}) +
  \xi^{-\alpha} \sum_{j = 1}^{\xi} H_\alpha([T(j)]_{k-1}) \right)
  \1_{\{\xi \neq 0\}},
\end{equation*}
where $(T(j))_{j \geqslant 1}$ are independent copies of $T$ that are also
independent of $\xi$.
Thus,
\begin{align*}
    H_\alpha([T]_k)^2 \;\overset{d}{=}\;& \;
    \frac{1}{(1-2^{1-\alpha})^2} (1-\xi^{1-\alpha})^2 \1_{\{\xi \neq 0\}} +
    \frac{2}{1-2^{1-\alpha}} (1-\xi^{1-\alpha})\xi^{-\alpha}\1_{\{\xi \neq 0\}}
    \sum_{j=1}^\xi H_\alpha([T(j)]_{k-1}) \\
    & +\xi^{-2\alpha}\1_{\{\xi \neq 0\}} \left( \sum_{j=1}^\xi
    H_\alpha([T(j)]_{k-1})^2 + \sum_{j_1 \neq j_2} H_\alpha([T(j_1)]_{k-1})
    \cdot H_\alpha([T(j_2)]_{k-1}) \right).
\end{align*}
Hence, taking expectations and using the independence,
\begin{align*}
    \E[H_\alpha([T]_k)^2] =& \frac{1}{(1-2^{1-\alpha})^2}
    \E[(1-\xi^{1-\alpha})^2 \1_{\{\xi \neq 0\}}] + \frac{2}{1-2^{1-\alpha}}
    \E[(1-\xi^{1-\alpha})\xi^{1-\alpha}\1_{\{\xi \neq 0\}}] \cdot
    \E[H_\alpha([T]_{k-1})] \\
    & + \E[\xi^{1-2\alpha}(\xi -1)\1_{\{\xi \neq 0\}}] \cdot
    \E[H_\alpha([T]_{k-1})]^2 + \E[\xi^{1-2\alpha}\1_{\{\xi \neq 0\}}] \cdot
    \E[H_\alpha([T]_{k-1})^2].
\end{align*}
Writing $u_k = \E[H_\alpha([T]_k)]$ and $v_k = \E[H_\alpha([T]_k)^2]$,
this recurrence is of the form
\[
  v_k = a\, v_{k-1} + b\, u_{k-1} + c\, u_{k-1}^2 + d.
\]
Thus, the sequence $(v_k)_{k\geqslant 0}$ converges if and only if:
\begin{enumerate}
    \item the coefficients $a$, $b$, $c$ and $d$ are finite, i.e.\ 
      $\mathds{E}[\xi^{2(1-\alpha)}\mathds{1}_{\{\xi \neq 0\}}] < + \infty$;
    \item $u_k = \E[H_\alpha([T]_{k})]$ converges;
    \item $a = \E[\xi^{1-2\alpha}\1_{\{\xi \neq 0\}}] < 1$;
\end{enumerate}
By \Cref{expecGW}, $u_k$ converges if and only if
$\E[\xi^{1-\alpha}\1_{\{\xi \neq 0\}}] < 1$, which implies
$\E[\xi^{1-2\alpha}\1_{\{\xi \neq 0\}}] < 1$.
Hence the first part of the theorem.
After standard calculations, we get:
\begin{align*}
    \Var(H_\alpha(T)) &= \E[H_\alpha(T)^2] - \E[H_\alpha(T)]^2 \\
    &= \frac{\mathds{P}(\xi = 0)}{(1-2^{1-\alpha})^2 \left(1-\mathds{E}\left[
      \xi^{1-2\alpha} \mathds{1}_{\{\xi \neq 0\}} \right] \right)} \left(1+
    \frac{\left( \mathds{E}\left[ \xi^{2(1-\alpha)} \mathds{1}_{\{\xi \neq 0\}}
    \right] -1 \right) \mathds{P}(\xi = 0)}{\left(1- \mathds{E}\left[
      \xi^{1-\alpha} \mathds{1}_{\{\xi \neq 0\}} \right] \right)^2} \right),
\end{align*}
concluding the proof.
\end{Proo}

Let us now prove \Cref{momentoforderm}, that is:

\begin{Theonn}
Let $T$ be a Galton--Watson tree with offspring distribution $\xi$, where
$\mathds{P}(\xi = 1) < 1$.
Then, for any positive integer $m$,
$\mathds{E}[(H_\alpha(T))^m] < +\infty$ if and only if
$\mathds{E}[\xi^{(1-\alpha)m}\mathds{1}_{\{\xi \neq 0\}}] < + \infty$
and $\mathds{E}[\xi^{1-\alpha}\mathds{1}_{\{\xi \neq 0 \}}] < 1$.
\end{Theonn}

\begin{Proo}
Let us prove the result by induction on $m$. The case $m =1$ is dealt with by
\Cref{expecGW}.
Since $\mathds{E}[\xi^{(1-\alpha)m}\mathds{1}_{\{\xi \neq
0\}}] < + \infty$, we have
$\mathds{E}[\xi^{(1-\alpha)k}\mathds{1}_{\{\xi \neq 0\}}] < + \infty$
for all $k < m$; thus, by the induction hypothesis,
$\mathds{E}[H_\alpha(T)^k]$ is finite.
Let $C$ be such that $\mathds{E}[H_\alpha(T)^k] \leqslant C$ for all
$k = 1, \ldots, m-1$.
By Eq.~\eqref{recurrenceGWdistri}, for any fixed $n \geq 1$, 
\begin{equation} \label{recrelorderm}
  H_\alpha([T]_n)^m \;\overset{d}{=}\;
  \left( \frac{1}{1-2^{1-\alpha}} (1-\xi^{1-\alpha}) +
  \xi^{-\alpha}  \sum_{k=1}^{\xi} H_\alpha\big([T(k)]_{n-1}\big) \right)^m
  \mathds{1}_{\{\xi \neq 0\}},
\end{equation}
where $(T(k))_{k \geqslant 1}$ is a family of independent random
trees distributed as $T$ that is independent of $\xi$.
Let $A = \frac{1}{1-2^{1-\alpha}} (1-\xi^{1-\alpha})$ and
$B= \xi^{-\alpha} \sum_{k=1}^{\xi} H_\alpha([T(k)]_{n-1})$.
Eq.~\eqref{recrelorderm} then gives:
\begin{equation*}
  H_\alpha([T]_n)^m \;\overset{d}{=}\;
  B^m \mathds{1}_{\{\xi \neq 0\}} + \left(
  \sum_{k=0}^{m-1} \binom{m}{k} B^k A^{m-k} \right) \mathds{1}_{\{\xi \neq 0\}}.
\end{equation*}
Thus,
\begin{equation} \label{eqproofintegermoments01}
  \mathds{E}[H_\alpha([T]_n)^m] = \mathds{E}\left[ \left( \sum_{k=1}^\xi
  H_\alpha([T(k)]_{n-1})\xi^{-\alpha} \right)^m \mathds{1}_{\{\xi \neq 0\}}
  \right] + \sum_{k=0}^{m-1} \binom{m}{k} \mathds{E}[B^k A^{m-k}
  \mathds{1}_{\{\xi \neq 0\}}].
\end{equation}
To start, let us rewrite the first term of
Eq.~\eqref{eqproofintegermoments01}:
\begin{eqnarray*}
& &\mathds{E}\left[ \left( \sum_{k=1}^\xi H_\alpha\big([T(k)]_{n-1}\big)
  \xi^{-\alpha}
  \right)^m \mathds{1}_{\{\xi \neq 0\}} \right] \\
& = & \mathds{E} \left[ \xi^{-\alpha m}\, \mathds{1}_{\{\xi \neq 0\}}\,
  \mathds{E} \left[ \sum_{\substack{k_1, \cdots, k_\xi \geqslant
  0 \\ k_1 + \cdots + k_\xi = m}} \!\! \binom{m}{k_1, \cdots, k_\xi}\, 
  H_\alpha\big([T(1)]_{n-1}\big)^{k_1} \cdots
  H_\alpha\big([T(\xi)]_{n-1}\big)^{k_\xi} ~\Bigg|~ \xi \right] \right] \\
& = & \mathds{E} \left[ \xi^{-\alpha m}\, \mathds{1}_{\{\xi \neq 0\}}\!\!\!
  \sum_{\substack{k_1, \cdots, k_\xi \geqslant 0\\ k_1 + \cdots + k_\xi =m }}
  \!\!\!\binom{m}{k_1, \cdots, k_\xi} \prod_{i =1}^\xi \mathds{E} \left[
    H_\alpha\big([T]_{n-1}\big)^{k_i} \right] \right] \\
& =& \mathds{E}\Big[\xi^{- \alpha m}\, 
  \mathds{1}_{\{ \xi \neq 0\}} \cdot \xi \cdot
  \mathds{E}\big[H_\alpha\big([T]_{n-1}\big)^m\big]\Big] \\
  & & +\; \mathds{E} \left[ \xi^{-\alpha m} \,
  \mathds{1}_{\{\xi \neq 0\}} \!\!\!
  \sum_{\substack{k_j \geqslant 0, k_j \neq m\\
  k_1 + \cdots + k_\xi = m}} \!\!\! \binom{m}{k_1, \cdots, k_\xi}
  \prod_{i =1}^\xi \mathds{E} \left[
    H_\alpha\big([T]_{n-1}\big)^{k_i} \right] \right].
\end{eqnarray*}
Now, note that, in the second term of this last equation,
since $k_j < m$ for all $j\in \{1, \ldots, \xi\}$,
$\mathds{E} \left[ H_\alpha([T]_{n-1})^{k_i} \right]\leqslant
\mathds{E} \left[H_\alpha(T)^{k_i} \right] \leqslant C^{k_i}$.
Thus, we can bound this term by
\begin{equation}
  \sum_{\substack{k_j \geqslant 0, k_j \neq m\\ k_1 + \cdots + k_\xi = m}}
  \!\!\! \binom{m}{k_1, \cdots, k_\xi}
  \prod_{i=1}^\xi C^{k_i} \;\leqslant
  \sum_{\substack{k_1, \cdots, k_\xi \geqslant 0\\ k_1 + \cdots + k_\xi =m }}
  \!\!\! \binom{m}{k_1, \cdots, k_\xi}
  \prod_{i=1}^\xi C^{k_i} = (C \xi)^m, 
\end{equation} 
and as a result, 
\begin{align} \label{eqproofintegermoments02}
  & \mathds{E} \left[ \left( \sum_{k=1}^ \xi H_\alpha([T(k)]_{n-1})\xi^{-\alpha}
  \right)^m \mathds{1}_{\{ \xi \neq 0\}} \right] \nonumber \\
 &\qquad\qquad \qquad \leqslant\;\; 
  \mathds{E} \big[\xi ^{1-\alpha m} \mathds{1}_{\{ \xi \neq 0\}}\big]
  \cdot
  \mathds{E}\big[H_\alpha([T]_{n-1})^m\big] + \mathds{E}\big[\xi ^{-\alpha m}
  \mathds{1}_{\{\xi \neq 0\}} (C \xi)^m\big] \nonumber \\
 & \qquad\qquad \qquad \leqslant\;\;
  \mathds{E}\big[\xi^{1-\alpha} \mathds{1}_{\{\xi \neq 0\}}\big] \cdot
  \mathds{E}\big[H_\alpha([T]_{n-1})^m\big] +
  C^m \mathds{E}\big[\xi ^{(1-\alpha) m}
  \mathds{1}_{\{\xi \neq 0\}}\big]. 
\end{align}
Now let us turn to the second term of Eq.~\eqref{eqproofintegermoments01}.
Using the same argument:
\begin{eqnarray} \label{eqproofintegermoments03}
\mathds{E}[B^k A^{m-k} \mathds{1}_{\{\xi \neq 0\}}] &=&
  \mathds{E}\left[\left(\xi^{-\alpha} \sum_{k=1}^\xi H_\alpha([T]_{n-1})
  \right)^k \cdot \left( \frac{1}{1-2^{1-\alpha}} (1-\xi^{(1-\alpha)})
  \right)^{m-k} \mathds{1}_{\{\xi \neq 0\}} \right] \nonumber \\
&=& \mathds{E} \left[ \xi^{-\alpha k} \left( \frac{1}{1-2^{1-\alpha}}
  \big(1-\xi^{(1-\alpha)}\big) \right)^{m-k} \mathds{1}_{\{\xi \neq 0\}}\, 
  \mathds{E} \left[ \left(\sum_{k=1}^\xi H_\alpha([T]_{n-1}) \right)^k ~
  \Bigg |~ \xi \right] \right] \nonumber \\
& \leqslant & \frac{C^k}{(1-2^{1-\alpha})^{m-k}}\,
  \mathds{E}\big[\xi^{(1-\alpha)k}
  (1- \xi^{(1-\alpha)})^{m-k} \mathds{1}_{\{\xi \neq 0\}}\big].
\end{eqnarray}
Finally, plugging Eqs.~\eqref{eqproofintegermoments02} and
\eqref{eqproofintegermoments03} in~\eqref{eqproofintegermoments01}, we get
\begin{align*}
\mathds{E}\big[H_\alpha([T]_n)^m\big] \leqslant&
  ~\mathds{E}\big[\xi^{1-\alpha}
  \mathds{1}_{\{\xi \neq 0\}}\big] \cdot
  \mathds{E}\big[H_\alpha([T]_{n-1})^m\big] + C^m
  \mathds{E}\big[\xi ^{(1-\alpha) m} \mathds{1}_{\{\xi \neq 0\}}\big] \\
&+ \sum_{k=0}^{m-1} \binom{m}{k} \frac{C^k}{(1-2^{1-\alpha})^{m-k}}
  \mathds{E}\big[\xi^{(1-\alpha)k} (1- \xi^{(1-\alpha)})^{m-k}
  \mathds{1}_{\{\xi \neq 0\}}\big].
\end{align*}
Since we have assumed that
$\mathds{E}[\xi^{(1-\alpha)m}\mathds{1}_{\{\xi \neq 0\}}] < +\infty$
and $\mathds{E}[\xi^{1-\alpha}\mathds{1}_{\{\xi \neq 0 \}}] < 1$,
this implies that the sequence
$(\mathds{E}[H_\alpha([T]_n)^m])_{n \geqslant 0}$ is
bounded, concluding the proof.
\end{Proo}

\subsection[Exponential moments of \texorpdfstring{$H_\alpha$}{H-alpha} for binary Galton--Watson trees]{%
  Exponential moments of \texorpdfstring{$\boldsymbol H_{\boldsymbol \alpha}$}{H-alpha} for binary Galton--Watson trees}

Finally, let us finally recall and prove \Cref{exponentialMoment}.

\begin{Propnn}
Let $T$ be a Galton--Watson tree with offspring distribution
$\xi \sim 2\,\mathrm{Ber}(p)$. The following are equivalent:
\begin{enumerate}[label=(\roman*)]
  \item $\E[H_\alpha(T)] < +\infty$.
  \item $p < 2^{\alpha -1}$.
  \item $H_\alpha(T)$ has exponential moments. \qedhere
\end{enumerate}
\end{Propnn}

\begin{Proo}
It is already proved in \Cref{expecGW} that (i) and (ii) are equivalent,
and the implication $\text{(iii)} \implies \text{(i)}$ is trivial;
so to conclude the proof it suffices to
prove $\text{(ii)} \implies \text{(iii)}$.
Let us write $\lambda =2p$ for the expectation of $\xi$ and
$(Z_k)$ for the underlying Galton--Watson process -- that is,
for all $k \geqslant 0$, $Z_k$ is the number of vertices at height~$k$ in $T$.
 \Cref{HalphaBinaire} yields
    \begin{equation*}
        H_\alpha(T) = 2^{\alpha -1} \sum_{k \geqslant 1} 2^{-\alpha k} Z_k =  2^{\alpha -1} \sum_{k \geqslant 1} (2^{1-\alpha}p)^k W_k,
    \end{equation*}
where $W_k = Z_k / \lambda^k$ is the canonical martingale associated with the
Galton--Watson process $(Z_k)$.
    
Let us now fix an integer $r \geqslant 3$.
To prove the desired implication, we show the existence of an
``exponentially summable'' upper bound on
$\E[H_\alpha(T)^r]$ -- that is, of $u_r \geqslant \E[H_\alpha(T)^r]$ such 
that there exists
$\theta > 0$ for which $\sum_{r \geqslant 3} u_r \frac{\theta^r}{r!} < + \infty$.
Note that, since $p < 2^{\alpha -1}$, one has, for any fixed $n \geqslant 1$,
\begin{align*}
    \left( \sum_{k \geqslant 1}^n (2^{1-\alpha}p)^k W_k \right)^r &\leqslant
    \left( \sum_{k \geqslant 1}^n
    (2^{1-\alpha}p)^k \right)^r \sup_{k \leqslant n} W_k^r  \\
    &\leqslant \left( \frac{1}{1-2^{1-\alpha}p} \right)^r \sup_{k \leqslant n}
    W_k^r.
\end{align*}
Moreover, since $W_k$ is a martingale, Doob's maximal inequality yields
\begin{equation*}
    \E \left[ \sup_{k \leqslant n} W_k^r \right] \leqslant
    \left( \frac{r}{r-1}\right)^r \E[W_n^r] \,.
\end{equation*}
As a result,
\begin{equation} \label{eqUpperBoundProofExpmoments}
  \E[H_\alpha(T)^r] \leqslant \left( \frac{2^{\alpha -1}}{1-2^{1-\alpha}p}
\right)^r \left( \frac{r}{r-1}\right)^r \sup_n \E[W_n^r] \,,
\end{equation}
and, since $((\frac{r}{r-1})^r)_{r\geqslant 3}$ is bounded,
to finish the proof it suffices to find a suitable upper bound on
$\sup_n \E[W_n^r]$.

For this, we will use a slightly refined version of a proof by
\cite{jansonMOMENTSLIMITINGRANDOM} of the following classic result
of Bingham and Doney:
first, recall that if $(Z_k)$ is a supercritical Galton--Watson process, the
Kesten--Stigum theorem states that the almost sure limit of the nonnegative
martingale $W_n$ is non-degenerate if and only if the offspring distribution
$\xi$ satisfies $\E[\xi \log^+\xi] < +\infty$
\cite{kestenLimitTheoremMultidimensional1966}.
In that context, 
\cite{binghamAsymptoticPropertiesSupercritical1974}
proved that $W_n$ converges in $L^r$ if and only if $\E[\xi^r] < +\infty$
(note that since $W_n^r$ is a submartingale, we then have
$\sup_n \E[W_n^r] =  \E[W_\infty^r]$).
In the unpublished manuscript \cite{jansonMOMENTSLIMITINGRANDOM}, 
Janson gave a simple proof of
$\E[\xi^r] <+\infty \iff \E[W_\infty^r] < +\infty$; what follows is
a straightforward adaptation of his proof.

First, recall that $\lambda = 2p$ and assume that $p > 1/2$, so that we are in
the supercritical regime.
Note that, since $\xi$ is bounded, $\E[\xi^r] < +\infty$ for all $r \geqslant 1$.
Now, note that
\begin{equation*}
    Z_{n+1}-\lambda Z_n = \sum_{i =1}^{Z_n}\xi'_i,
\end{equation*}
where $(\xi'_i)_{i \geqslant 1}$ are independent copies of
$\xi' = \xi - \lambda$ and
$(\xi'_i)_{i \geqslant 1}$ is independent of $Z_n$.
Applying \cite[Chap.~3, Cor.~8.2]{gutProbabilityGraduateCourse2013}, we get
\begin{equation*}
    \E\big[|Z_{n+1}-\lambda Z_n|^r ~\big|~Z_n\big]
    \leqslant C_r \cdot Z_n^{r/2}\cdot \E\big[|\xi'|^r\big],
\end{equation*}
where $C_r$ is a constant that depends on $r$ only. Note that, 
although the statement of the corollary
in \cite{gutProbabilityGraduateCourse2013} does not say
so explicitly, it follows from its proof that one can in fact take
$C_r = 2^r B_r$, where $B_r$
is the best constant for the upper bound in Khintchine's inequality
(see \cite{haagerupBestConstantsKhintchine1981}).
Since $\E[|\xi'|^r] \leqslant 2^r$, it follows that
\begin{equation*}
    \E\big[|Z_{n+1}-\lambda Z_n|^r\big]
    \leqslant 2^r C_r  \E[Z_n^{r/2}]
    \leqslant 2^r C_r  \E[Z_n^{r}]^{1/2},
\end{equation*}
and, therefore, that
\begin{equation*}
    \norm{Z_{n+1}-\lambda Z_n}_r \leqslant 2 \, C_r^{1/r} \norm{Z_n}_r^{1/2}.
\end{equation*}
Now, notice that since $Z_n = \lambda^n W_n$, one has $W_{n+1}-W_n =
\lambda^{-n-1} (Z_{n+1}-\lambda Z_n)$, so that the last inequality can be
rewritten as
\begin{equation*}
    \norm{W_{n+1}- W_n}_r \leqslant \lambda^{-\frac{n}{2} -1}\, 2\, C_r^{1/r}
    \norm{W_n}_r^{1/2}.
\end{equation*}
So the triangle inequality gives
\begin{equation*}
    \norm{W_{n+1}}_r \leqslant \norm{W_n}_r + \left( \frac{1}{\sqrt{\lambda}}
    \right)^n \cdot \frac{2\, C_r^{1/r}}{\lambda} \norm{W_n}_r^{1/2}.
\end{equation*}
Setting $v_n = \norm{W_n}_r^{1/2}$, we see that $(v_n)_{n\geqslant 0}$
satisfies $v_0 = 1$ and
\[
  v_{n+1}^2 \leqslant v_n^2 + a q^n v_n, 
\]
where $a = 2 C_r^{1/r} / \lambda > 0$ and $q = \lambda^{-1/2} \in \;]0; 1[$.
Since $(v_n + aq^n / 2)^2 \geqslant v_n^2 + aq^n v_n$, we have
$v_{n+1} \leqslant v_n + a q^n / 2$ and, therefore,
\[
  v_n \leqslant 1 + \frac{a}{2} \sum_{k = 0}^{n-1} q^k 
  \leqslant 1 + \frac{a}{2} \frac{1}{1-q} \,.
\]
As a result,
\begin{equation*}
    \norm{W_n}_r \leqslant \left( 1 +
    \frac{C_r^{1/r}}{\lambda(1-\lambda^{-1/2})} \right)^2.
\end{equation*}
Now, recall that we have chosen $C_r = 2^r B_r$, where
$B_r$ is the optimal upper bound in Khintchine's inequality, that is:
\begin{equation*}
    B_r = \sqrt{2} \left( \Gamma \left( \frac{r+1}{2} \right) \cdot
    \frac{1}{\sqrt{\pi}}  \right)^{1/r},
\end{equation*}
where $\Gamma$ is the gamma function -- see
\cite{haagerupBestConstantsKhintchine1981}.
As a result, noting that $\Gamma((r+1)/2) \leqslant r^r$ for $r \geqslant 3$, 
we have the crude upper bound
\[
  C_r^{1/r} = 
   2^{1 + \frac{1}{2r}} \left( \Gamma \left( \frac{r+1}{2} \right) \cdot
    \frac{1}{\sqrt{\pi}}  \right)^{1/r^2}
    \leqslant  2^{1 + \frac{1}{2r}}\cdot r^{\frac{1}{r}}
    \leqslant M\,, 
\]
for some fixed constant $M$.
As a result, there exists a constant $K > 1$ such that,
for all $n\geqslant 0$ and all $r$,
$\|W_n\|_r \leqslant K$, i.e.\ $\E[W_n^r] \leqslant K^r$.
Plugging in Eq.~\eqref{eqUpperBoundProofExpmoments}, 
we see that there exists $\theta > 0$ such that 
$\theta^r \E[H_\alpha(T)^r]/r!$ is summable, i.e.\ such that 
\begin{equation*}
    \E\big[\exp(\theta H_\alpha(T))\big] < +\infty.
\end{equation*}

We have thus proved that $\text{(ii)} \implies \text{(iii)}$ when
$p > 1/2$. To conclude the proof, 
note that when $\Tilde{p} < p$, $\xi\sim 2\,\mathrm{Ber}(p)$ stochastically
dominates $\Tilde{\xi}\sim 2\,\mathrm{Ber}(\Tilde{p})$.
As a result, the associated Galton--Watson processes can be coupled in
such a way that $\Tilde{Z_k} \leqslant Z_k$ for all $k \geqslant 0$.
Thus, letting $T$ and~$\Tilde{T}$ denote the associated
Galton--Watson trees, it follows from \Cref{HalphaBinaire}
that $H_\alpha(\Tilde{T}) \leqslant H_\alpha(T)$.
As a result, the existence of exponential moments for $H_\alpha(T)$ implies the
existence of exponential moments for $H_\alpha(\Tilde{T})$,
which concludes the proof.
\end{Proo}

\begin{Rema}
As can be deduced from our remark at the end of
Section~\ref{secGW}, $\tilde{\xi} \preccurlyeq \xi$
does not imply $H_\alpha(\mathrm{GW}(\tilde{\xi})) \preccurlyeq
H_\alpha(\mathrm{GW}(\xi))$. Similarly, 
let us stress that \emph{unless $\tilde{\tau}$ and $\tau$ are binary trees},
$\tilde{\tau} \subset \tau$ does not imply
$H_\alpha(\tilde{\tau}) \leqslant H_\alpha(\tau)$.
\end{Rema}

\section{The Yule model: asymptotics of the variance of \texorpdfstring{$\boldsymbol{H_\alpha}$}{H-alpha}} \label{appYule}

In this appendix, we prove \Cref{asymptoticVariance}.
First, let us recall its statement.

\begin{Coronn}
Let $T_n$ be a Yule tree with $n$ leaves, and let
$c = -\log_{2}(1-\frac{\sqrt{2}}{2})\approx 1.7716$.
Then, for $\alpha \neq 1$, as $n\to\infty$,
\begin{equation*}
  \Var(H_{\alpha}(T_n)) \sim_{n \rightarrow + \infty}
    \begin{cases}
    Q_\alpha\, n^{2(2^{1-\alpha}-1)} & \text{if } \alpha < c, \\[1.2em]
    \dfrac{1}{\Gamma(2^{1-2\alpha})}\,\log(n)\, n^{\,2(2^{1-\alpha}-1)}
    & \text{if } \alpha = c, \\[1.2em]
    \dfrac{1}{\Gamma(2^{1-2\alpha})(1+2^{1-2\alpha}-2^{2-\alpha})}\;
      n^{\,2^{1-2\alpha}-1}
    & \text{if } \alpha > c.
    \end{cases}
\end{equation*}
where $Q_\alpha > 0$ is a constant.
\end{Coronn}

\begin{Proo}
Recall from the main text that
$X_{n, \alpha} = \sum_{\ell \in \mathcal{L}_{T_n}} p_\ell^\alpha$ and that
$Y_{n, \alpha} = 
X_{n, \alpha} \prod_{k=1}^n \big( 1 + \frac{2^{1-\alpha} -1}{k} \big)^{-1}$
is a martingale.
Next, note that
\begin{equation}
    \Var(H_\alpha(T_n)) = \frac{1}{(1-2^{1-\alpha})^2} \prod_{k=1}^n \left( 1+ \frac{2^{1-\alpha}-1}{k} \right)^2 \Var(Y_{n, \alpha}). \label{relationHZ}
\end{equation}
Moreover, since $Y_{n, \alpha}$ is a martingale and $\Var(Y_{1, \alpha}) = 0$, 
\begin{equation*}
  \Var (Y_{n, \alpha}) =
  \sum_{k=1}^{n-1} \E\big[(Y_{k+1, \alpha} - Y_{k, \alpha})^2\big].
\end{equation*}
We thus focus on expressing $\E[(Y_{k+1, \alpha} - Y_{k, \alpha})^2]$. To ease
notation, as previously we set $\beta = 1 - 2^{1-\alpha}$.
First, note that
\begin{equation*}
    Y_{k+1, \alpha} - Y_{k, \alpha} =
    \left(X_{k+1, \alpha}- \big(1 - \tfrac{\beta}{k+1} \big)
    X_{k, \alpha}\right) \,
    \prod_{j=1}^{k+1} \left(1 - \frac{\beta}{j} \right)^{-1} \,.
\end{equation*}
Thus,
\begin{equation*}
    \E\big[(Y_{k+1, \alpha} - Y_{k, \alpha})^2\big] =
    \E \left[\left( X_{k+1, \alpha}- \big(1 - \tfrac{\beta}{k+1} \big)
    X_{k, \alpha} \right)^2 \right] \,
    \prod_{j=1}^{k+1} \left( 1 - \frac{\beta}{j} \right)^{-2} \,.
\end{equation*}
In this expression, 
\begin{equation*}
    \E \left[\left( X_{k+1, \alpha}- \big(1 - \tfrac{\beta}{k+1} \big)
    X_{k, \alpha} \right)^2 \right]
    = \E[X_{k+1, \alpha}^2] - 2 \big(1 - \tfrac{\beta}{k+1} \big)
    \E[X_{k+1, \alpha} X_{k, \alpha}] + \big(1 - \tfrac{\beta}{k+1} \big)^2
    \E[X_{k, \alpha}^2].
\end{equation*}
Moreover, by Eq.~\eqref{cond},
\begin{equation*}
    \E[X_{k+1, \alpha} X_{k, \alpha}]
    = \E\big[X_{k, \alpha} \E[X_{k+1, \alpha} ~|~T_k]\big] =
    \left(1 - \frac{\beta}{k} \right) \,\E\left[X_{k, \alpha}^2  \right],
\end{equation*}
and, by Eq.~\eqref{reccurence2moment},
\begin{equation*}
    \E[X_{k+1, \alpha}^2] =
    \left(1 - \frac{2\beta}{k} \right) \E[X_{k, \alpha}^2]  +
    \frac{\beta^2}{k} \E[X_{k, 2\alpha}].
\end{equation*}
So, using elementary algebra we get:
\begin{align*}
    \E \left[\left( X_{k+1, \alpha}- \left(1 - \tfrac{\beta}{k+1} \right)
    X_{k, \alpha} \right)^2 \right] &= - \frac{(k+2)\beta^2}{k(k+1)^2}
    \E[X_{k, \alpha}^2] + \frac{\beta^2}{k} \E[X_{k, 2\alpha}]. \\
    &=  \frac{\beta^2}{k} \left( \E[X_{k, 2\alpha}] - \frac{k+2}{(k+1)^2}
    \E[X_{k, \alpha}^2] \right).
\end{align*}
We now study the asymptotic behavior of this expression as $k\to\infty$.
First, by \Cref{expectYule} and Eq.~\eqref{asymp}, we have:
\begin{equation*}
    \E[X_{k, 2\alpha}] \sim \frac{k^{2^{1-2\alpha}-1}}{\Gamma(2^{1-2\alpha})}.
\end{equation*}
Second, we use Eq.~\eqref{moment2} to determine an asymptotic equivalent
of $\E[X_{k, \alpha}^2]$.
For this, note that, by Eq.~\eqref{asymp},
\begin{equation*}
  u_m := \frac{g_m}{\prod_{k=1}^m f_k} \sim_{m \rightarrow + \infty}
    \frac{\Gamma(2^{2-\alpha}-1) \beta^2}{\Gamma(2^{1-2\alpha})}\,
    m^{2^{1-2\alpha}-2^{2-\alpha}}.
\end{equation*}
A simple analysis of the function $x \in \mathbb{R}_+ \mapsto 2^{1-2x} - 2^{2-x}$
shows that $2^{1-2\alpha} - 2^{2-\alpha} \leqslant -1$ if and only if $\alpha
\leqslant c := - \log_2 \big( 1-\frac{\sqrt{2}}{2} \big)$. Thus,
\begin{itemize}
    \item If $\alpha > c$, the series $\sum u_m$ is divergent and
      by comparison with a $p$-series we have
    \begin{equation*}
        1+ \sum_{m=1}^{n-1} \frac{g_m}{\prod_{k=1}^m f_k} \sim_{n \rightarrow
        +\infty} \frac{\Gamma(2^{2-\alpha}-1)
        \beta^2}{\Gamma(2^{1-2\alpha})(1+2^{1-2\alpha} - 2^{2-\alpha})}\, 
        n^{1+2^{1-2\alpha}-2^{2-\alpha}}.
    \end{equation*}
    Therefore,
    \begin{equation*}
        \prod_{k=1}^{n-1} f_k \cdot \left( 1+ \sum_{m=1}^{n-1}
        \frac{g_m}{\prod_{k=1}^m f_k}  \right) \sim_{n \rightarrow +\infty}
        \frac{\beta^2}{\Gamma(2^{1-2\alpha})(1+2^{1-2\alpha} - 2^{2-\alpha})}\,
        n^{2^{1-2\alpha}-1}.
    \end{equation*}
    \item If $\alpha = c$, the series $\sum u_m$ 
      is also divergent and by equivalence with the harmonic series we have
    \begin{equation*}
         1+ \sum_{m=1}^{n-1} \frac{g_m}{\prod_{k=1}^m f_k}
         \sim_{n \rightarrow +\infty} \frac{\Gamma(2^{2-\alpha}-1)
         \beta^2}{\Gamma(2^{1-2\alpha})} \, \log(n),
    \end{equation*}
    and, therefore,
    \begin{equation*}
        \prod_{k=1}^{n-1} f_k \cdot \left( 1+ \sum_{m=1}^{n-1}
        \frac{g_m}{\prod_{k=1}^m f_k}  \right) \sim_{n \rightarrow +\infty}
        \frac{\beta^2}{\Gamma(2^{1-2\alpha})} \cdot \log(n) n^{-2\beta}.
    \end{equation*}
    \item If $\alpha < c$, the
      series $\sum u_m$ is convergent so there exists a positive constant
      $K_\alpha$ such that 
    \begin{equation*}
        \sum_{m=1}^{n-1} \frac{g_m}{\prod_{k=1}^m f_k}
        \sim_{n \rightarrow +\infty} K_\alpha.
    \end{equation*}
    Once again using Eq.~\eqref{asymp}, we thus have 
\begin{equation*}
    \prod_{k=1}^{n-1} f_k \cdot \left( 1+ \sum_{m=1}^{n-1}
    \frac{g_m}{\prod_{k=1}^m f_k}  \right)
    \sim_{n \rightarrow + \infty}
    \frac{1+ K_\alpha}{\Gamma(2^{2-\alpha}-1)} \, n^{-2\beta}.
\end{equation*}
\end{itemize}
To summarize, 
\begin{equation*}
  \E\left[X_{k,\alpha}^2\right] \sim
  \begin{cases}
  \dfrac{\beta^2}{\Gamma\left(2^{1-2\alpha}\right)
  \left(1+2^{1-2\alpha}-2^{2-\alpha}\right)} \, 
  k^{2^{1-2\alpha}-1}
  & \text{if } \alpha > c,
  \\[1.2em]
  \dfrac{\beta^2}{\Gamma\left(2^{1-2\alpha}\right)}\, 
  \log(k) \, k^{-2\beta}
  & \text{if } \alpha = c,
  \\[1.2em]
  \dfrac{1+K_\alpha}{\Gamma\left(2^{2-\alpha}-1\right)}\,
  k^{-2\beta}
  & \text{if } \alpha < c.
  \end{cases}
\end{equation*}
Since $\frac{k+2}{(k+1)^2} \sim k^{-1}$, we deduce that for any $\alpha > 0$,
as $k\to\infty$,
\begin{equation*}
     \E[X_{k, 2\alpha}] - \frac{k+2}{(k+1)^2} \E[X_{k, \alpha}^2]
     \sim \frac{k^{2^{1-2\alpha}-1}}{\Gamma(2^{1-2\alpha})},
\end{equation*}
so that 
\begin{equation*}
    \frac{\beta^2}{k} \left( \E[X_{k, 2\alpha}] - \frac{k+2}{(k+1)^2}
    \E[X_{k, \alpha}^2] \right)
    \sim \frac{\beta^2}{\Gamma(2^{1-2\alpha})}\, k^{2^{1-2\alpha}-2}.
\end{equation*}
Since,
\begin{equation*}
    \prod_{j=1}^{k+1} \left(1 - \frac{\beta}{j} \right)^{-2}
    \sim \Gamma(2^{1-\alpha})^2\, k^{2\beta}  \,, 
\end{equation*}
we deduce that
\begin{equation*}
     \E[(Y_{k+1, \alpha} - Y_{k, \alpha})^2]
     \sim \frac{\Gamma(2^{1-\alpha})^2\cdot \beta^2}{\Gamma(2^{1-2\alpha})}
     \cdot k^{2^{1-2\alpha}-2^{2-\alpha}}.
\end{equation*}
As before, by comparison with a $p$-series we have that, as $n\to\infty$, 
\begin{equation*}
    \Var(Y_{n,\alpha}) \sim
    \begin{cases}
    K'_\alpha 
    & \text{if } \alpha < c,
    \\[1.2em]
    \dfrac{\Gamma(2^{1-\alpha})^2 \cdot
      (2^{1-\alpha} - 1)^2}{\Gamma(2^{1-2\alpha})}\cdot\log(n)
    & \text{if } \alpha = c,
    \\[1.2em]
    \dfrac{\Gamma(2^{1-\alpha})^2 \cdot \beta^2}
    {\Gamma(2^{1-2\alpha})\cdot(1 + 2^{1-2\alpha} - 2^{2-\alpha})}
      \cdot n^{1+2^{1-2\alpha}-2^{2-\alpha}}
    & \text{if } \alpha > c.
    \end{cases}
\end{equation*}
where $K'_\alpha=\sum_{k=1}^{+\infty} \E[(Y_{k+1, \alpha} - Y_{k, \alpha})^2]$. 

Moreover, since 
\begin{equation*}
    \prod_{k=1}^n \left(1 - \frac{\beta}{k} \right)^2 \sim
    \frac{n^{-2\beta}}{\Gamma(2^{1-\alpha})^2},
\end{equation*}
we finally deduce that:
\begin{equation*}
    \Var(H_{n,\alpha}) \sim_{n \rightarrow + \infty}
    \begin{cases}
    \dfrac{K'_\alpha}{\Gamma(2^{1-\alpha})^2\cdot \beta^2}\, n^{-2\beta}
    & \text{if } \alpha < c,
    \\[1.2em]
    \dfrac{1}{\Gamma(2^{1-2\alpha})}\, \log(n)\,n^{-2\beta}
    & \text{if } \alpha = c,
    \\[1.2em]
    \dfrac{1}{\Gamma(2^{1-2\alpha})(1+2^{1-2\alpha}-2^{2-\alpha})}\, 
      n^{2^{1-2\alpha}-1}
    & \text{if } \alpha > c.
    \end{cases}
\end{equation*}
This concludes the proof.
\end{Proo}

\section{Bijection between leaf-pointed ordered binary trees and grand Dyck paths}
\label{appPDA}

In this appendix, we prove \Cref{lemGrandDyck} by giving a bijection between
leaf-pointed ordered binary trees and grand Dyck paths.
Let us start by recalling the definition of the latter.

\begin{Defi}
A \emph{grand Dyck path with semilength $n$} --
also known as a \emph{free Dyck path} or as a \emph{Dyck bridge} -- is a
lattice path from $(0, 0)$ to $(2n, 0)$ consisting of steps $(+1, +1)$
and $(+1, -1)$.
A point $(t, 0)$ of the path with $t > 0$ is called a 
\emph{return to the $x$-axis}.
\end{Defi}

The combinatorics of grand Dyck paths is well-studied, see e.g.\
\cite{Ferrari2010,SapounakisTsikourasTasoulasManes2012}.
In particular, the number of grand Dyck paths of semilength $n$ with
$k$ returns to the $x$-axis is
\[
  \frac{k\,2^k}{2n-k} \binom{2n-k}{n}  \,,
\]
see entry A108747 of the OEIS~\cite{OEISA108747}.
Therefore, to finish the proof of \Cref{lemGrandDyck}, it suffices to give
a bijection between the set of ordered binary trees with $n$ leaves having
a distinguished leaf at distance $k$ from the root and the set of grand Dyck
paths of semilength $n-1$ with $k$ returns to the $x$-axis.

\begin{Proo}[Bijection between leaf-pointed ordered binary trees and grand Dyck paths]

Consider an ordered binary tree $\tau$ with $n$ leaves having a distinguished
leaf $\ell$ at distance $k$ from the root, as in
Figure~\ref{TreeToEncode}.

By ``cutting'' $\tau$ at each of the $k-1$ vertices inside the path from the
root to $\ell$, we obtain $k$ ordered binary trees such that, for each of these
trees, one of the two children of the root is distinguished (namely,
the child that corresponds to the next vertex on the path from the root of
$\tau$ to $\ell$) and is a leaf: see Figure~\ref{FirstStep}, top row.
For some of these $k$ trees, the distinguished child of the root might be
its right child: in that case, ``flip'' the tree horizontally to ensure
that the distinguished child of the root is always its left child
(Figure~\ref{FirstStep}, bottom row).

Each of the $k$ ordered trees obtained as described above can be encoded
bijectively by a Dyck path: for this, traverse the edges of the tree in
depth-first order, starting with left edges and interpreting these
as $(+1, +1)$ steps and right edges as $(+1, -1)$ steps, as illustrated
in Figure~\ref{Dyckpath}.
For each of these trees, because the left child of the root is a leaf,
the corresponding Dyck path is guaranteed to hit the $x$-axis exactly
twice (i.e.\ to return to it exactly once).
Finally, to obtain a grand Dyck path:
\begin{enumerate}
  \item ``flip'' vertically each of the Dyck paths that correspond to trees
    that had to be flipped horizontally during the previous step;
  \item concatenate all the resulting paths.
\end{enumerate}
See Figure~\ref{Dyckpath} for an illustration. Note that the resulting
grand Dyck path has exactly $k$ returns to the $x$-axis: one for each
of the $k$ Dyck paths used to construct it.
Because this construction is reversible and its inverse can be applied to any
grand Dyck path, it is bijective. This concludes the proof. \qedhere

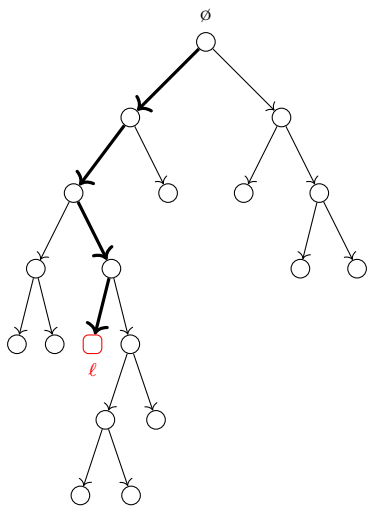
\begin{figure}[H]
  \centering
  \begin{tikzpicture}[main/.style = {draw, circle, minimum size = 7pt, inner sep =1pt}]
    \tikzstyle{dot}=[draw,circle,gray,
                  top color= white, text=gray,minimum size=7pt, inner sep = 1]
    \tikzstyle{phantom}=[opacity = 0, minimum size = 7pt]
    \tikzstyle{leaf}=[draw,rounded corners=2pt,red,
                  top color= white, text=red,minimum size=7pt]
    \node[main, label ={[font=\footnotesize] \o}] (root) at (0, 0) {};
    \node[main] (1) at (-1, -1) {};
    \node[main] (2) at (1, -1) {};
    \node[main] (21) at (0.5, -2) {};
    \node[main] (22) at (1.5, -2) {};
    \node[main] (221) at (1.25, -3) {};
    \node[main] (222) at (2, -3) {};
    \node[main] (11) at (-1.75, -2) {};
    \node[main] (12) at (-0.5, -2) {};
    \node[main] (111) at (-2.25, -3) {};
    \node[main] (112) at (-1.25, -3) {};
    \node[main] (1111) at (-2.5, -4) {};
    \node[main] (1112) at (-2, -4) {};
    \node[leaf, label ={[font=\scriptsize]below:\textcolor{red}{$\ell$}}] (1121) at (-1.5, -4) {};
    \node[main] (1122) at (-1, -4) {};
    \node[main] (11221) at (-1.33, -5) {};
    \node[main] (11222) at (-0.66, -5) {};
    \node[main] (112211) at (-1.66, -6) {};
    \node[main] (112212) at (-0.99, -6) {};
    \draw[->, line width=1.2pt] (root) -- (1);
    \draw[->] (root) -- (2);
    \draw[->] (2) -- (21);
    \draw[->] (2) -- (22);
    \draw[->] (22) -- (222);
    \draw[->] (22) -- (221);
    \draw[->, line width=1.2pt] (1) -- (11);
    \draw[->] (1) -- (12);
    \draw[->] (11) -- (111);
    \draw[->, line width=1.2pt] (11) -- (112);
    \draw[->] (111) -- (1111);
    \draw[->] (111) -- (1112);
    \draw[->, line width=1.2pt] (112) -- (1121);
    \draw[->] (112) -- (1122);
    \draw[->] (1122) -- (11221);
    \draw[->] (1122) -- (11222);
    \draw[->] (11221) -- (112211);
    \draw[->] (11221) -- (112212);
  \end{tikzpicture}
  \caption{An ordered binary tree $\tau$ with $n = 10$ leaves and a
  distinguished leaf $\ell$ at distance $k = 4$ from the root.}
  \label{TreeToEncode}
\end{figure}

\begin{figure}[H]
  \centering
  \begin{tikzpicture}[main/.style = {draw, circle, minimum size = 7pt, inner sep =1pt}]
    \tikzstyle{dot}=[draw,circle,gray,
                  top color= white, text=gray,minimum size=7pt, inner sep = 1]
    \tikzstyle{phantom}=[opacity = 0, minimum size = 7pt]
    \tikzstyle{leaf}=[draw,rounded corners=2pt,red,
                  top color= white, text=red,minimum size=7pt]
    \node[main, label ={[font=\footnotesize] \o}] (root) at (-1, 0) {};
    \node[main] (1) at (-2, -1) {};
    \node[main] (2) at (0, -1) {};
    \node[main] (21) at (-0.5, -2) {};
    \node[main] (22) at (0.5, -2) {};
    \node[main] (221) at (0.25, -3) {};
    \node[main] (222) at (1, -3) {};
    \draw[->, line width=1.2pt] (root) -- (1);
    \draw[->] (root) -- (2);
    \draw[->] (2) -- (21);
    \draw[->] (2) -- (22);
    \draw[->] (22) -- (222);
    \draw[->] (22) -- (221);
    \node[phantom] (x) at (1.5, 0) {};
    \node[phantom] (y) at (1.5, -1) {};
    \draw[] (x) -- (y);
    \node[main] (rootTree2) at (3, 0) {};
    \node[main] (1Tree2) at (2.5, -1) {};
    \node[main] (2Tree2) at (3.5, -1) {};
    \draw[->, line width=1.2pt] (rootTree2) -- (1Tree2);
    \draw[->] (rootTree2) -- (2Tree2);
    \node[phantom] (x2) at (4.5, 0) {};
    \node[phantom] (y2) at (4.5, -1) {};
    \draw[] (x2) -- (y2);
    \node[main] (rootTree3) at (6, 0) {};
    \node[main] (1Tree3) at (5.5, -1) {};
    \node[main] (2Tree3) at (6.5, -1) {};
    \node[main] (11Tree3) at (5, -2) {};
    \node[main] (12Tree3) at (6, -2) {};
    \draw[->, line width=1.2pt] (rootTree3) -- (2Tree3);
    \draw[->] (rootTree3) -- (1Tree3);
    \draw[->] (1Tree3) -- (11Tree3);
    \draw[->] (1Tree3) -- (12Tree3);
    \node[phantom] (x3) at (7.5, 0) {};
    \node[phantom] (y3) at (7.5, -1) {};
    \draw[] (x3) -- (y3);
    \node[main] (rootTree4) at (9, 0) {};
    \node[leaf, label ={[font=\scriptsize]below:\textcolor{red}{$\ell$}}] (1Tree4) at (8.5, -1) {};
    \node[main] (2Tree4) at (9.5, -1) {};
    \node[main] (21Tree4) at (9, -2) {};
    \node[main] (22Tree4) at (10, -2) {};
    \node[main] (211Tree4) at (8.5, -3) {};
    \node[main] (212Tree4) at (9.5, -3) {};
    \draw[->, line width=1.2pt] (rootTree4) -- (1Tree4);
    \draw[->] (rootTree4) -- (2Tree4);
    \draw[->] (2Tree4) -- (21Tree4);
    \draw[->] (2Tree4) -- (22Tree4);
    \draw[->] (21Tree4) -- (211Tree4);
    \draw[->] (21Tree4) -- (212Tree4);
  \end{tikzpicture}
  \\[3ex]
  \begin{tikzpicture}[main/.style = {draw, circle, minimum size = 7pt, inner sep =1pt}]
    \tikzstyle{dot}=[draw,circle,gray,
                  top color= white, text=gray,minimum size=7pt, inner sep = 1]
    \tikzstyle{phantom}=[opacity = 0, minimum size = 7pt]
    \tikzstyle{leaf}=[draw,rounded corners=2pt,red,
                  top color= white, text=red,minimum size=7pt]
    \node[main, label ={[font=\footnotesize] \o}] (root) at (-1, 0) {};
    \node[main] (1) at (-2, -1) {};
    \node[main] (2) at (0, -1) {};
    \node[main] (21) at (-0.5, -2) {};
    \node[main] (22) at (0.5, -2) {};
    \node[main] (221) at (0.25, -3) {};
    \node[main] (222) at (1, -3) {};
    \draw[->, line width=1.2pt] (root) -- (1);
    \draw[->] (root) -- (2);
    \draw[->] (2) -- (21);
    \draw[->] (2) -- (22);
    \draw[->] (22) -- (222);
    \draw[->] (22) -- (221);
    \node[phantom] (x) at (1.5, 0) {};
    \node[phantom] (y) at (1.5, -1) {};
    \draw[] (x) -- (y);
    \node[main] (rootTree2) at (3, 0) {};
    \node[main] (1Tree2) at (2.5, -1) {};
    \node[main] (2Tree2) at (3.5, -1) {};
    \draw[->, line width=1.2pt] (rootTree2) -- (1Tree2);
    \draw[->] (rootTree2) -- (2Tree2);
    \node[phantom] (x2) at (4.5, 0) {};
    \node[phantom] (y2) at (4.5, -1) {};
    \draw[] (x2) -- (y2);
    \node[main] (rootTree3) at (6, 0) {};
    \node[main] (1Tree3) at (5.5, -1) {};
    \node[main] (2Tree3) at (6.5, -1) {};
    \node[main] (21Tree3) at (6, -2) {};
    \node[main] (22Tree3) at (7, -2) {};
    \draw[->, line width=1.2pt] (rootTree3) -- (1Tree3);
    \draw[->] (rootTree3) -- (2Tree3);
    \draw[->] (2Tree3) -- (21Tree3);
    \draw[->] (2Tree3) -- (22Tree3);
    \node[phantom] (x3) at (7.5, 0) {};
    \node[phantom] (y3) at (7.5, -1) {};
    \draw[] (x3) -- (y3);
    \node[main] (rootTree4) at (9, 0) {};
    \node[leaf, label ={[font=\scriptsize]below:\textcolor{red}{$\ell$}}] (1Tree4) at (8.5, -1) {};
    \node[main] (2Tree4) at (9.5, -1) {};
    \node[main] (21Tree4) at (9, -2) {};
    \node[main] (22Tree4) at (10, -2) {};
    \node[main] (211Tree4) at (8.5, -3) {};
    \node[main] (212Tree4) at (9.5, -3) {};
    \draw[->, line width=1.2pt] (rootTree4) -- (1Tree4);
    \draw[->] (rootTree4) -- (2Tree4);
    \draw[->] (2Tree4) -- (21Tree4);
    \draw[->] (2Tree4) -- (22Tree4);
    \draw[->] (21Tree4) -- (211Tree4);
    \draw[->] (21Tree4) -- (212Tree4);
  \end{tikzpicture}

  \caption{Top row, the $k$ trees obtained by ``cutting'' $\tau$ at each of the
  $k$ internal vertices on the path from the root to $\ell$.
  Note that, for each of these trees, one child of the root is distinguished
  (materialized by the bold edge), and that this child is a leaf. However,
  this distinguished child can be the left or the right child of the root.
  Bottom row, the $k$ trees obtained after ``flipping'' some of them
  horizontally to ensure that the distinguished child of the root is always
  its left child.}
  \label{FirstStep}
\end{figure}
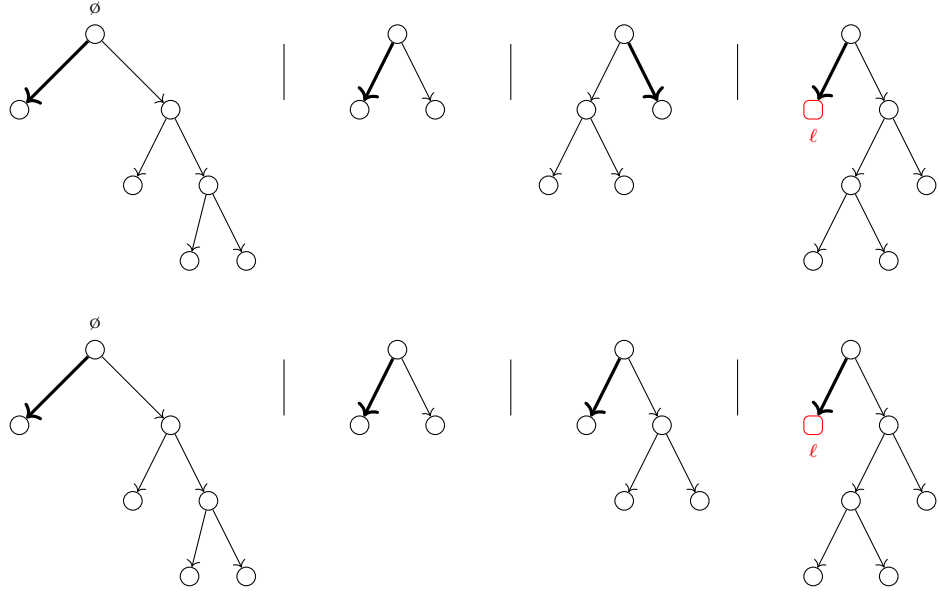

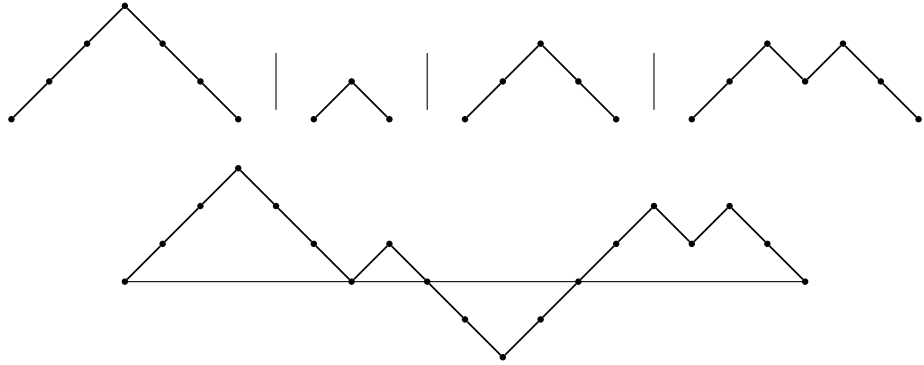
\begin{figure}[H]
  \centering
  \begin{tikzpicture}[main/.style = {draw, circle, minimum size = 10pt, inner sep =1pt}]
    \tikzstyle{leaf}=[draw,rectangle,red,
                  top color= white, text=red,minimum width=10pt]
    \tikzstyle{dot}=[draw,circle,gray,
                  top color= white, text=gray,minimum width=10pt]
    \tikzstyle{phantom}=[opacity = 0]
    \tikzstyle{full}=[draw, circle, fill, minimum width=2pt, inner sep =0.5pt]
    \node[full] (1) at (0,0) {};
    \node[full] (2) at (0.5,0.5) {};
    \node[full] (3) at (1,1) {};
    \node[full] (4) at (1.5,1.5) {};
    \node[full] (5) at (2,1) {};
    \node[full] (6) at (2.5,0.5) {};
    \node[full] (7) at (3,0) {};
    \draw[line width=0.66pt] (1) -- (2);
    \draw[line width=0.66pt] (2) -- (3);
    \draw[line width=0.66pt] (3) -- (4);
    \draw[line width=0.66pt] (4) -- (5);
    \draw[line width=0.66pt] (5) -- (6);
    \draw[line width=0.66pt] (6) -- (7);
    \node[phantom] (x) at (3.5, 0) {};
    \node[phantom] (y) at (3.5, 1) {};
    \draw[] (x) -- (y);
    \node[full] (1b) at (4,0) {};
    \node[full] (2b) at (4.5,0.5) {};
    \node[full] (3b) at (5,0) {};
    \draw[line width=0.66pt] (1b) -- (2b);
    \draw[line width=0.66pt] (2b) -- (3b);
    \node[phantom] (x2) at (5.5, 0) {};
    \node[phantom] (y2) at (5.5, 1) {};
    \draw[] (x2) -- (y2);
    \node[full] (1c) at (6,0) {};
    \node[full] (2c) at (6.5,0.5) {};
    \node[full] (3c) at (7,1) {};
    \node[full] (4c) at (7.5,0.5) {};
    \node[full] (5c) at (8,0) {};
    \draw[line width=0.66pt] (1c) -- (2c);
    \draw[line width=0.66pt] (2c) -- (3c);
    \draw[line width=0.66pt] (3c) -- (4c);
    \draw[line width=0.66pt] (4c) -- (5c);
    \node[phantom] (x3) at (8.5, 0) {};
    \node[phantom] (y3) at (8.5, 1) {};
    \draw[] (x3) -- (y3);
    \node[full] (1d) at (9,0) {};
    \node[full] (2d) at (9.5,0.5) {};
    \node[full] (3d) at (10,1) {};
    \node[full] (4d) at (10.5,0.5) {};
    \node[full] (5d) at (11,1) {};
    \node[full] (6d) at (11.5,0.5) {};
    \node[full] (7d) at (12,0) {};
    \draw[line width=0.66pt] (1d) -- (2d);
    \draw[line width=0.66pt] (2d) -- (3d);
    \draw[line width=0.66pt] (3d) -- (4d);
    \draw[line width=0.66pt] (4d) -- (5d);
    \draw[line width=0.66pt] (5d) -- (6d);
    \draw[line width=0.66pt] (6d) -- (7d);
  \end{tikzpicture}
  \\[3ex]
  \begin{tikzpicture}[main/.style = {draw, circle, minimum size = 10pt, inner sep =1pt}]
    \tikzstyle{leaf}=[draw,rectangle,red,
                  top color= white, text=red,minimum width=10pt]
    \tikzstyle{dot}=[draw,circle,gray,
                  top color= white, text=gray,minimum width=10pt]
    \tikzstyle{phantom}=[opacity = 0]
    \tikzstyle{full}=[draw, circle, fill, minimum width=2pt, inner sep =0.5pt]
    \node[full] (1) at (0,0) {};
    \node[full] (2) at (0.5,0.5) {};
    \node[full] (3) at (1,1) {};
    \node[full] (4) at (1.5,1.5) {};
    \node[full] (5) at (2,1) {};
    \node[full] (6) at (2.5,0.5) {};
    \node[full] (7) at (3,0) {};
    \draw[line width=0.66pt] (1) -- (2);
    \draw[line width=0.66pt] (2) -- (3);
    \draw[line width=0.66pt] (3) -- (4);
    \draw[line width=0.66pt] (4) -- (5);
    \draw[line width=0.66pt] (5) -- (6);
    \draw[line width=0.66pt] (6) -- (7);
    \node[full] (2b) at (3.5,0.5) {};
    \node[full] (3b) at (4,0) {};
    \draw[line width=0.66pt] (7) -- (2b);
    \draw[line width=0.66pt] (2b) -- (3b);
    \node[full] (2c) at (4.5,-0.5) {};
    \node[full] (3c) at (5,-1) {};
    \node[full] (4c) at (5.5,-0.5) {};
    \node[full] (5c) at (6,0) {};
    \draw[line width=0.66pt] (3b) -- (2c);
    \draw[line width=0.66pt] (2c) -- (3c);
    \draw[line width=0.66pt] (3c) -- (4c);
    \draw[line width=0.66pt] (4c) -- (5c);
    \node[full] (2d) at (6.5,0.5) {};
    \node[full] (3d) at (7,1) {};
    \node[full] (4d) at (7.5,0.5) {};
    \node[full] (5d) at (8,1) {};
    \node[full] (6d) at (8.5,0.5) {};
    \node[full] (7d) at (9,0) {};
    \draw[line width=0.66pt] (5c) -- (2d);
    \draw[line width=0.66pt] (2d) -- (3d);
    \draw[line width=0.66pt] (3d) -- (4d);
    \draw[line width=0.66pt] (4d) -- (5d);
    \draw[line width=0.66pt] (5d) -- (6d);
    \draw[line width=0.66pt] (6d) -- (7d);
    \draw[] (1) -- (7d);
  \end{tikzpicture}
  \caption{Top row, the Dyck paths associated to the trees in the bottom
  row of Figure~\ref{FirstStep}. Note that each of these Dyck paths has
  exactly one return to the $x$-axis. Bottom row, the grand Dyck path
  obtained by concatenating the Dyck paths in the top row. Note that the
  third of these Dyck paths had to be ``flipped'' below the $x$-axis to
  record the information that the corresponding tree had to be flipped
  horizontally at the previous step of the construction. Also note that
  this grand Dyck path has exactly $k$ returns to the $x$-axis.}
  \label{Dyckpath}
\end{figure}
\end{Proo}

\section{Behavior of the \texorpdfstring{$\boldsymbol{H_\alpha}$}{H-alpha} index for the local topology}
In this appendix, we will describe the behavior of the $H_\alpha$ indices under the local topology. We will begin with a discussion on measurability and semi-continuity before giving conditions for the convergence of the $H_\alpha$ index of a sequence of networks.

\subsection{First properties} \label{appNonContinuity}
\begin{Prop}
The function $H_\alpha$ is a measurable random variable for the local topology.
\end{Prop}

The proof of this fact is a straightforward adaptation of \cite[Prop.~A.16]{bienvenu$B_2$IndexGalled2024}. Indeed, we can show that $H_\alpha$ is the (pointwise) limit of continuous (for the local topology) functions.
Following the notation in \cite{bienvenu$B_2$IndexGalled2024}
we consider the functions
\begin{equation*}
    H_\alpha^{k,n} : G \longmapsto \frac{1}{1-2^{1-\alpha}} \left( \sum_{A \in [G]_k} p_{A, n}(1-p_{A,n}^{\alpha-1}) \right),
\end{equation*}
instead of the functions $B_2^{k,n}$ and conclude with the same arguments.

However the $H_\alpha$ indices are not continuous for this topology so one can not state that if a sequence $(G_n)_n$ converges towards $G$ in the local topology then $(H_\alpha(G_n))_n$ would converge towards $H_\alpha(G)$. The following proposition states a more precise result.

\begin{Prop}
The $H_\alpha$ indices are not lower nor upper semi continuous for the local topology.
\end{Prop}

\begin{Proo}
    We can consider the following counter examples.
\begin{itemize}
    \item For the non lower semi continuity, one can consider the sequence of phylogenetic networks $(G_n)_{n \geqslant 1}$ and its local limit $G$ depicted in (respectively) \Cref{seqNonLowersc,locLimitNonLowersc}.
        \begin{figure}[H]
        \centering
        \begin{subfigure}[t]{.4\textwidth}
        \centering
    \begin{tikzpicture}[main/.style = {draw, circle, minimum size = 7pt, inner sep =1pt}]
    \tikzstyle{dot}=[draw,circle,gray,
                  top color= white, text=gray,minimum size=7pt, inner sep = 1]
    \tikzstyle{phantom}=[opacity = 0, minimum size = 7pt]
    \tikzstyle{leaf}=[draw,rounded corners=2pt,red,
                  top color= white, text=red,minimum size=7pt]
    \node[main, label ={[font=\footnotesize] \o}] (root) at (0, 0) {};
    \node[main, label ={[font=\scriptsize]left:1}] (1) at (0, -1) {};
    \node[main, label ={[font=\scriptsize]left:$n$}] (2) at (0, -3) {};
    \node[main] (3) at (-1, -4) {};
    \node[main] (4) at (1, -4) {};
    \draw[->] (root) -- (1);
    \draw[dashed, ->] (1) -- (2);
    \draw[->] (2) -- (3);
    \draw[->] (2) -- (4);
\end{tikzpicture}
    \caption{Representation of the $n$-th term of the sequence $(G_n)_{n \geqslant 1}$.}
    \label{seqNonLowersc}
\end{subfigure}
~
\begin{subfigure}[t]{.4\textwidth}
\centering
    \begin{tikzpicture}[main/.style = {draw, circle, minimum size = 7pt, inner sep =1pt}]
    \tikzstyle{dot}=[draw,circle,gray,
                  top color= white, text=gray,minimum size=7pt, inner sep = 1]
    \tikzstyle{phantom}=[opacity = 0, minimum size = 7pt]
    \tikzstyle{leaf}=[draw,rounded corners=2pt,red,
                  top color= white, text=red,minimum size=7pt]
    \node[main, label ={[font=\footnotesize] \o}] (root) at (0, 0) {};
    \node[main, label ={[font=\scriptsize]left:1}] (1) at (0, -1) {};
    \node[main, label ={[font=\scriptsize]left:$n$}] (2) at (0, -3) {};
    \node[phantom] (3) at (0, -4) {};
    \draw[->] (root) -- (1);
    \draw[dashed, ->] (1) -- (2);
    \draw[dashed, ->] (2) -- (3);
\end{tikzpicture}
    \caption{Representation of $G$.}
    \label{locLimitNonLowersc}
\end{subfigure}
    \caption{A counter example to the non lower semi-continuity.}
    \label{contreExToLower}
    \end{figure}
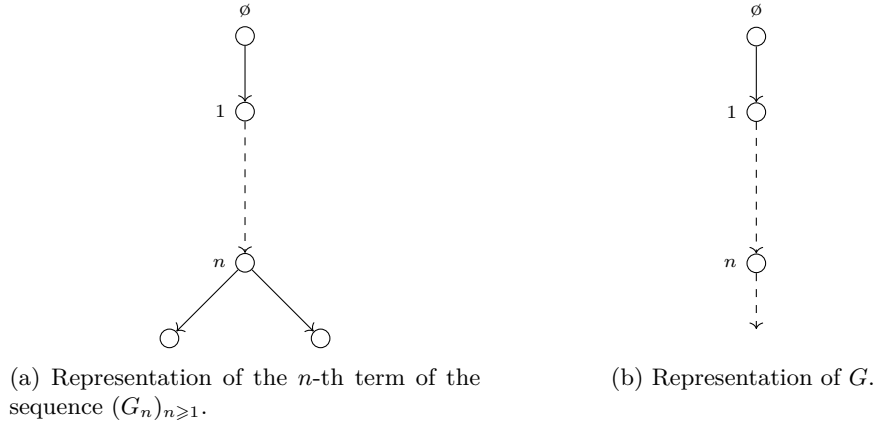
    
One can easily compute the $H_\alpha$ indices of $G$ and of $G_n$ for all integers $n$. Indeed,
\begin{equation*}
    H_\alpha(G) = 0 \text{ and for all } n \in \mathbb{N}, H_\alpha(G_n) = 1.
\end{equation*}
Hence the non lower semi-continuity.
\item For the non upper semi continuity, one can consider the sequence of phylogenetic networks $(G_n)_{n \geqslant 1}$ and its local limit $G$ depicted in (respectively) \Cref{seqNonUppersc,locLimitNonUppersc}.
        \begin{figure}[H]
        \centering
        \begin{subfigure}[t]{.4\textwidth}
        \centering
    \begin{tikzpicture}[main/.style = {draw, circle, minimum size = 7pt, inner sep =1pt}]
    \tikzstyle{dot}=[draw,circle,gray,
                  top color= white, text=gray,minimum size=7pt, inner sep = 1]
    \tikzstyle{phantom}=[opacity = 0, minimum size = 7pt]
    \tikzstyle{leaf}=[draw,rounded corners=2pt,red,
                  top color= white, text=red,minimum size=7pt]
    \node[main, label ={[font=\footnotesize] \o}] (root) at (0, 0) {};
    \node[main, label ={[font=\scriptsize]left:1}] (1) at (-1, -1) {};
    \node[main, label ={[font=\scriptsize]right:1'}] (1bis) at (1, -1) {};
    \node[main, label ={[font=\scriptsize]left:2}] (2) at (-1, -2) {};
    \node[main, label ={[font=\scriptsize]right:2'}] (2bis) at (1, -2) {};
    \node[main, label ={[font=\scriptsize]left:$n$}] (n) at (-1, -4) {};
    \node[main, label ={[font=\scriptsize]right:$n'$}] (nbis) at (1, -4) {};
    \node[main, label ={[font=\scriptsize]left:$n+1$}] (n+1) at (-1, -5) {};
    \node[main, label ={[font=\scriptsize]right:$(n+1)'$}] (n+1bis) at (1, -5) {};
    \node[phantom] (c) at (-1, -6) {};
    \node[phantom] (cbis) at (1, -6) {};
    \node[phantom] (p) at (-1, -7) {};
    \node[phantom] (pbis) at (1, -7) {};
    \draw[->] (root) -- (1);
    \draw[->] (root) -- (1bis);
    \draw[->] (1) -- (2);
    \draw[->] (1bis) -- (2bis);
    \draw[dashed, ->] (2) -- (n);
    \draw[dashed, ->] (2bis) -- (nbis);
    \draw[->] (n) -- (nbis);
    \draw[->] (n) -- (n+1);
    \draw[->] (nbis) -- (n+1bis);
    \draw[->] (n+1) -- (n+1bis);
    \draw[dashed, ->] (n+1) -- (p);
    \draw[dashed, ->] (n+1bis) -- (pbis);
    \draw[dashed, ->] (c) -- (cbis);
\end{tikzpicture}
    \caption{Representation of the $n$-th term of the sequence $(G_n)_{n \geqslant 1}$.}
    \label{seqNonUppersc}
\end{subfigure}
~
\begin{subfigure}[t]{.4\textwidth}
\centering
    \begin{tikzpicture}[main/.style = {draw, circle, minimum size = 7pt, inner sep =1pt}]
    \tikzstyle{dot}=[draw,circle,gray,
                  top color= white, text=gray,minimum size=7pt, inner sep = 1]
    \tikzstyle{phantom}=[opacity = 0, minimum size = 7pt]
    \tikzstyle{leaf}=[draw,rounded corners=2pt,red,
                  top color= white, text=red,minimum size=7pt]
    \node[main, label ={[font=\footnotesize] \o}] (root) at (0, 0) {};
    \node[main, label ={[font=\scriptsize]left:1}] (1) at (-1, -1) {};
    \node[main, label ={[font=\scriptsize]right:1'}] (1bis) at (1, -1) {};
    \node[main, label ={[font=\scriptsize]left:2}] (2) at (-1, -2) {};
    \node[main, label ={[font=\scriptsize]right:2'}] (2bis) at (1, -2) {};
    \node[phantom] (n) at (-1, -4) {};
    \node[phantom] (nbis) at (1, -4) {};
    \draw[->] (root) -- (1);
    \draw[->] (root) -- (1bis);
    \draw[->] (1) -- (2);
    \draw[->] (1bis) -- (2bis);
    \draw[dashed, ->] (2) -- (n);
    \draw[dashed, ->] (2bis) -- (nbis);
\end{tikzpicture}
    \caption{Representation of $G$.}
    \label{locLimitNonUppersc}
\end{subfigure}
    \caption{A counter example to the non upper semi-continuity.}
    \label{contreExToUpper}
    \end{figure}
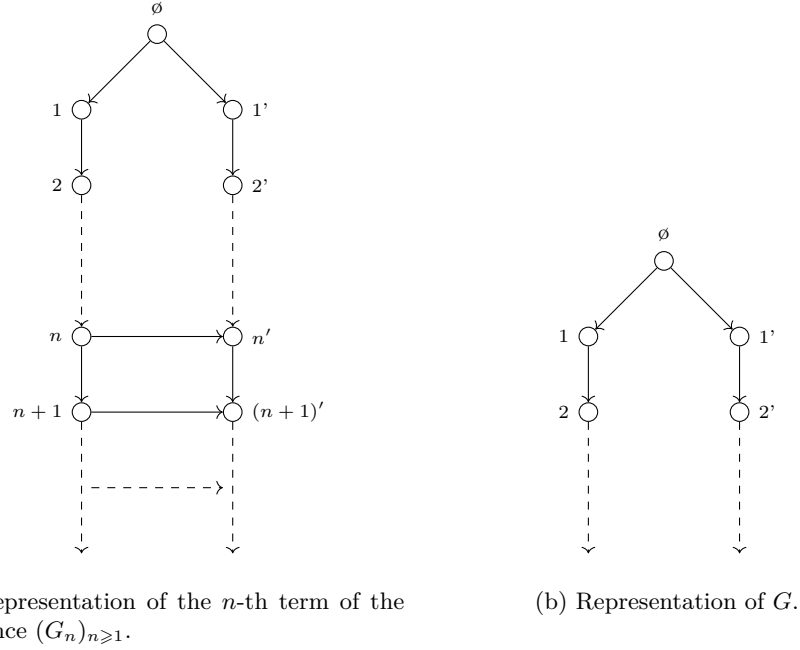
One can compute the $H_\alpha$ indices of $G$ and of $G_n$ for all integers $n$ using for instance the stub sequence $(S_{n_k})$ being $G_n$ cut at height $kn$ (see \Cref{Stubsequence}). It yields
\begin{equation*}
   H_\alpha(G) = 1 \text{ and for all } n \in \mathbb{N}, H_\alpha(G_n) = 0.
\end{equation*}
Hence the non upper semi continuity. \qedhere
\end{itemize}
\end{Proo}

\subsection[Conditions for the convergence of the \texorpdfstring{$H_\alpha$}{H-alpha} index]{%
Conditions for the convergence of the \texorpdfstring{$\boldsymbol{H_\alpha}$}{H-alpha} index} \label{appCondContinuity}

One can overcome the non-continuity difficulty in some cases. We already saw that for a tree $T$,  the $H_\alpha$ index being increasing, the sequence $(T_n)_n$ defined by $T_n = [T]_n$ (that obviously converges towards $T$ in the local topology) is a sequence that verifies $H_\alpha(T_n) \rightarrow H_\alpha(T)$ as $n$ goes to infinity. When considering networks, things are not as easy and one needs to be more careful when truncating the network. First, recall the definition of a \textit{stub sequence} from \cite{bienvenu$B_2$IndexGalled2024}.

\begin{Defi}
Let $G$ be a phylogenetic network. Select a set of cut-vertices $\Xi$ in $G$ and ungraft the networks subtended by each $v \in \Xi$. The remaining network $S$ is called a stub and is denoted by $S \sqsubset G$. An increasing sequence of stubs $S_n \sqsubset S_{n+1} \sqsubset G$ that converges toward $G$ is called a stub sequence.
\end{Defi}

\begin{Prop} \label{Stubsequence}
Let $(S_n)$ be a stub sequence of $G$. Then, 
\begin{equation*}
H_\alpha(S_n) \underset{n \rightarrow \infty}{\longrightarrow} H_\alpha(G). \qedhere
\end{equation*}
\end{Prop}

The proof of this proposition is a straightforward adaptation of
\cite[Prop.~A.20]{bienvenu$B_2$IndexGalled2024}. Indeed, we can define
a sequence of partitions associated with $(S_n)$ that converges
towards the partition in singletons of $G$ and so \Cref{cvgSingleton}
concludes the proof.

To end this section, we state a generalization of
\cite[Prop.~A.17~\&~A.18]{bienvenu$B_2$IndexGalled2024} that provides
us with a sufficient condition for the convergence of the $H_\alpha$
index of a sequence in the local topology.

\begin{Prop} \label{localcvg}
Let $(G_n)$ be a sequence of phylogenetic networks converging locally to $G$. Assume that the directed random walk on $G$ ends in a leaf almost surely. Then,
\begin{itemize}
\item[\textbullet] If $\alpha > 1$, $\lim_n H_\alpha(G_n) = H_\alpha(G)$.
\item[\textbullet] If $0 < \alpha <1$, then $\liminf_n H_\alpha(G_n) \geqslant H_\alpha$.

Moreover, if there exists a sequence $(k_n)$ such that $[G_n]_{k_n} = [G]_{k_n}$ and
\begin{equation*}
\mathds{P}(h(X_\infty) > k_n) (|\mathcal{L}_{G_n}|)^{\frac{1}{\alpha}-1} \underset{n \rightarrow \infty}{\longrightarrow} 0,
\end{equation*}
(where $h(X_\infty)$ is the height of the limit of the directed random walk on $G$) then, $\lim_n H_\alpha(G_n) = H_\alpha(G)$. \qedhere
\end{itemize}
\end{Prop}

\begin{Rema}
  As in \Cref{cvgHalpha}, this proposition illustrates that it is ``harder''
  for the $H_\alpha$ index to converge when $\alpha$ is small, with a
  critical change in $\alpha = 1$ (i.e.\ for the $B_2$ index).
\end{Rema}

\begin{Proo}
Let us start with showing that $\liminf_n H_\alpha(G_n) \geqslant H_\alpha$. We refer the reader to \cite{bienvenu$B_2$IndexGalled2024} and rapidly recall the main ideas. By definition of local convergence we can pick $(k_n)$ a sequence verifying $[G_n]_{k_n} = G_{k_n}$. Set $\mathcal{L}_G^k$ the set of leaves of $G$ of height at most $k$ and as usual write $p_\ell$ for the probability that the random walk on $G$ ends in $\ell$. Assume that $\mathds{P}(X_\infty \in \mathcal{L}_G) = 1$, we thus have
\begin{equation*}
\frac{1}{1-2^{1-\alpha}} \sum_{\ell \in \mathcal{L}_G^k} p_\ell (1-p_\ell^{\alpha -1}) \underset{k \rightarrow \infty}{\longrightarrow} H_\alpha(G).
\end{equation*} 
Taking $b$ smaller than $H_\alpha(G)$ and $k$ such that $\frac{1}{1-2^{1-\alpha}} \sum_{\ell \in \mathcal{L}_G^k} p_\ell (1-p_\ell^{\alpha -1}) > b$ we have
\begin{equation*}
H_\alpha(G_n) \geqslant \frac{1}{1-2^{1-\alpha}} \sum_{\ell \in \mathcal{L}_k} p_\ell (1-p_\ell^{\alpha -1}) > b.
\end{equation*}
Thus, $\liminf_n H_\alpha(G_n) \geqslant b$ and letting $b \rightarrow H_\alpha(G)$ yields $\liminf_n H_\alpha(G_n) \geqslant H_\alpha$.

Assume now that $\alpha > 1$. To prove the first point, we now only have to show that $\limsup_n H_\alpha(G_n) \leqslant H_\alpha(G)$. Let $q_n = \mathds{P}(h(X_\infty) > k_n)$. By hypothesis, $q_n \underset{n \rightarrow \infty}{\longrightarrow} 0$. Denote by $\mathcal{L}_{G_n}^{k_n}$ the set of leaves of $G_n$ of height at most $k_n$ and $\mathcal{L}_{G_n}^{>k_n}$ the set of leaves of $G_n$ of height at least $k_n +1$. Since $[G_n]_{k_n} = [G]_{k_n}$, 
\begin{equation*}
\frac{1}{1-2^{1-\alpha}} \sum_{\ell \in \mathcal{L}_{G_n}^{k_n}} p_{n, \ell}(1- p_{n, \ell}^{\alpha -1}) = \frac{1}{1-2^{1-\alpha}} \sum_{\ell \in \mathcal{L}_{G}^{k_n}} p_\ell (1-p_\ell^{\alpha -1}) \underset{n \rightarrow \infty}{\longrightarrow} H_\alpha(G)
\end{equation*}
Moreover, 
\begin{equation*}
0 < \frac{1}{1-2^{1-\alpha}} \sum_{\ell \in \mathcal{L}_{G_n}^{>k_n}} p_{n, \ell} (1-p_{n, \ell}^{\alpha -1}) = \frac{1}{1-2^{1-\alpha}} \left( q_n - \sum_{\ell \in \mathcal{L}_{G_n}^{>k_n}
} p_{n, \ell}^\alpha \right) < \frac{q_n}{1-2^{1-\alpha}} \underset{n \rightarrow \infty}{\longrightarrow} 0.
\end{equation*}
Thus, $\limsup_n H_\alpha(G_n) \leqslant H_\alpha(G)$ and so $H_\alpha(G_n) \underset{n \rightarrow \infty}{\longrightarrow} H_\alpha(G)$.

Now, let us turn to point two. Assume that $0 < \alpha < 1$ and that one can pick a sequence $(k_n)$ such that $[G_n]_{k_n} = [G]_{k_n}$ and 
\begin{equation*}
\mathds{P}(h(X_\infty) > k_n) (|\mathcal{L}_{G_n}|)^{\frac{1}{\alpha}-1} \underset{n \rightarrow \infty}{\longrightarrow} 0.
\end{equation*}
We will show that under this condition $\limsup_n H_\alpha(G_n) \leqslant H_\alpha(G)$. As before, note that there exists $n_0 \in \mathbb{N}$ such that for all $n \geqslant n_0$,
\begin{equation*}
H_\alpha(G_n) = \frac{1}{1-2^{1-\alpha}} \sum_{\ell \in \mathcal{L}_G^{k_n}} p_{\ell} (1- p_{\ell}^{\alpha -1}) + \frac{1}{1-2^{1-\alpha}} \sum_{\ell \in \mathcal{L}_{G_n}^{> k_n}} p_{n, \ell} (1- p_{n, \ell}^{\alpha -1}).
\end{equation*}
Since the first term tends to $H_\alpha(G)$ as $n \rightarrow \infty$ we shall bound the second term. Note that $(p_{n, \ell} / q_n : \ell \in \mathcal{L}_{G_n}^{>k_n})$ is a probability distribution. Thus by Jensen inequality (applied to the concave function $f_\alpha : x >0 \mapsto \frac{1-x^{1-\alpha}}{1-2^{1-\alpha}}$) we have that:
\begin{eqnarray*}
\sum_{\ell \in \mathcal{L}_{G_n}^{>k_n}} \frac{p_{n, \ell}}{q_n} \cdot \frac{1-p_{n, \ell}^{\alpha -1}}{1-2^{1-\alpha}}
& \leqslant & \frac{1}{1-2^{1-\alpha}} \left( 1- \left( \sum_{\ell \in \mathcal{L}_{G_n}^{>k_n}} \frac{p_{n, \ell}}{q_n} \cdot \frac{1}{p_{n, \ell}} \right)^{1-\alpha} \right) \\
& \leqslant & \frac{1}{1-2^{1-\alpha}} - \frac{1}{1-2^{1-\alpha}} \left(\frac{1}{q_n} |\mathcal{L}_{G_n}| \right)^{1-\alpha}.
\end{eqnarray*}
Thus, the second term can be bounded by:
\begin{equation*}
\frac{1}{1-2^{1-\alpha}} \sum_{\ell \in \mathcal{L}_{G_n}^{> k_n}} p_{n, \ell} (1- p_{n, \ell}^{\alpha -1}) \leqslant \frac{q_n}{1-2^{1-\alpha}} + \frac{1}{1-2^{1-\alpha}} q_n^\alpha |\mathcal{L}_{G_n}|^{1-\alpha},
\end{equation*}
and this quantity tends to $0$ as $n \rightarrow \infty$ since we assumed that $q_n|\mathcal{L}_{G_n}|^{\frac{1}{\alpha}-1} \underset{n \rightarrow \infty}{\longrightarrow} 0$.
\end{Proo}

\section{Blowups of Galton--Watson trees: proofs} \label{appBlowups}

We end this document with the proofs of the results of
\Cref{secBlowup}, related to the blowups of Galton--Watson trees.

\subsection{The \texorpdfstring{$H_\alpha$}{H-alpha} index of Kesten trees: proof of \Cref{expecKesten}} \label{appHalphaKesten}

For the sake of clarity, let us recall the statement of \Cref{expecKesten}. Also, recall that $T$ is a Galton--Watson tree with a
critical offspring distribution $\xi$, $T_\star$ is the associated
Kesten tree and $\hat{\xi}$ is the size-biased distribution.
\begin{Theonn}
  Assume that $\mathds{E}[\xi]=1$ and $\mathds{P}(\xi = 1) < 1$.
  Then $\mathds{E}[H_\alpha(T_\star)]$ is finite if and only if
  $\mathds{E}[\hat{\xi}^{1-\alpha}] < + \infty$. In this case, we have
  the following expression:
\begin{equation*}
\mathds{E}[H_\alpha(T_\star)] = \frac{\frac{1}{1-2^{1-\alpha}}(1-\mathds{E}[\hat{\xi}^{1-\alpha}])+\mathds{E}[\hat{\xi}^{-\alpha}(\hat{\xi}-1)] \cdot \mathds{E}[H_\alpha(T)]}{1- \mathds{E}[\hat{\xi}^{-\alpha}]},
\end{equation*}
where $\mathds{E}[H_\alpha(T)]$ is explicit in view of \Cref{expecGW}.
\end{Theonn}

As in the proof of \Cref{momentoforderm}, we use the recursive
construction of the Kesten tree by applying \Cref{graftingproperty} to
the structure close to the root.

\begin{Proo}
  Let us fix $n\geq 2$. By construction of the Kesten tree, letting
  $\hat{\xi}$ denote the number of children of the root, we have
\begin{equation*}
  H_\alpha([T_\star]_n) \overset{d}{=}  \frac{1}{1-2^{1-\alpha}} (1- \hat{\xi}^{1-\alpha}) + \hat{\xi}^{-\alpha} \left( \sum_{k=1}^{\hat{\xi} -1} H_\alpha([T(k)]_{n-1}) + H_\alpha([T_\star']_{n-1}) \right).
\end{equation*}
where $(T(k))_{k\geq 1}$ are copies of $T$, the random tree $T_\star'$
is distributed as $T_\star$, all of these variables are independent
and independent of $\xi$. Thus, conditioning on $\hat{\xi}$ and taking
expectations we get:
\begin{equation*}
  \mathds{E}[H_\alpha([T_\star]_n)] = \frac{1}{1-2^{1-\alpha}} (1-\mathds{E}[\hat{\xi}^{1-\alpha}]) + \mathds{E}[\hat{\xi}^{-\alpha}(\hat{\xi} -1)] \cdot \mathds{E}[H_\alpha([T]_{n-1})] + \mathds{E}[\hat{\xi}^{-\alpha}] \cdot \mathds{E}[H_\alpha([T_\star]_{n-1})].
\end{equation*}
Note that if $\mathds{E}[\hat{\xi}^{1-\alpha}]=\infty$, then the above
expression is infinite for every $n$. Therefore we now assume that
this expectation is finite. Setting
$v_n = \mathds{E}[H_\alpha([T_\star]_n)]$,
$c = \mathds{E}[\hat{\xi}^{-\alpha}]$ and
$d_n = \frac{1}{1-2^{1-\alpha}} (1-\mathds{E}[\hat{\xi}^{1-\alpha}]) +
\mathds{E}[\hat{\xi}^{-\alpha}(\hat{\xi} -1)] \cdot
\mathds{E}[H_\alpha([T]_{n-1})]$, the above induction is simply
rewritten as
\[
  v_n = d_n + cv_{n-1}
\]
The term $\mathds{E}[H_\alpha([T]_{n-1})]$ converges to
$\mathds{E}[H_\alpha(T)]$ as $n$ goes to infinity, and the limit is
finite if and only if
$\mathds{E}[\xi^{1-\alpha}\1_{\{\xi\neq 0\}}] < 1$, which is
automatically satisfied under our assumptions on $\xi$, so $d_n$ tends
to a finite value. Solving the induction explicitly, we get:
\begin{equation*}
v_n = c^n \cdot \sum _{m=0}^{n-1} \frac{d_m}{c^m} = \sum_{m=1}^{n} d_m \cdot c^{n-m}
\end{equation*}
Since $d_n$ is bounded, the series converges if and only if $c <1$.
Because $\hat{\xi}\geq 1$ almost surely, we have
$c=\mathds{E}[\hat{\xi}^{-\alpha}] \leqslant 1$, and $c$ cannot be
equal to $1$ since the trivial case $\mathds{P}(\xi=1)=1$ is excluded
by our assumptions. We conclude using the monotone convergence theorem
that $v_n$ converges to $\frac{d}{1-c}$, the only fixed point of
$f : x \mapsto cx +d$, and so $\mathds{E}[H_\alpha(T_\star)]$ is
finite and is equal to $\frac{d}{1-c}$.
\end{Proo}

\subsection{Total variation distance of truncated Galton--Watson trees} \label{appPrelimBlowups}

We begin this appendix by recalling a truncation of the Kesten tree
and the matching construction for Galton--Watson trees. Once again we
refer the reader to Section~A.5 of
\cite{bienvenu$B_2$IndexGalled2024} for the technical details. We will
also recall without any proofs some results from
\cite{bienvenu$B_2$IndexGalled2024} necessary to our proof of
\Cref{theoprincipal}. Recall that $T$ is a Galton--Watson tree with a
critical offspring distribution $\xi$, $T_\star$ is the associated
Kesten tree and $\hat{\xi}$ is the size-biased distribution. Moreover
we denote by $T_n$ a Galton--Watson tree conditioned on having $n$
leaves (for all $n$ such that this event occurs with non zero
probability). The considered truncation of $T_\star$ is the following.
For all $k$, we denote by $T_\star^k$ the tree obtained from $T_\star$
by deleting all descendants of $v_k$, the $k$-th vertex on the spine.
Moreover, we keep track of this vertex and so, when referring to
$T_\star^k$, we implicitly consider the pair $(T_\star^k, v_k)$. We
now use a similar construction for $T_n$. Let $\ell_n$ be a leaf of
$T_n$ chosen uniformly at random, and recall that $\delta_{\ell_n}$
denotes the depth of $\ell_n$. For all $k \leqslant \delta_{\ell_n}$
consider $v_{n,k}$ the $k$-th vertex on the unique path between the
root and $\ell_n$. Let then $T_n^k$ be the tree obtained from $T_n$ by
deleting all descendants of $v_{n,k}$. As before we consider $T_n^k$
equipped with $v_{n,k}$. For completeness, although this is of little importance, we let $T_n^k$ be the tree reduced to a single
vertex for any other $k$. We now recall Prop.~A.21
from \cite{bienvenu$B_2$IndexGalled2024}.

\begin{Prop} \label{distanceTV}
    If $\xi$ has a finite third moment, then for any sequence of integers $(k_n)$ satisfying $k_n = o(\sqrt{n})$, we have
    \begin{equation*}
        d_{\mathrm{TV}}(T_n^{k_n}, T_\star^{k_n}) = \Theta \left( \frac{k_n}{\sqrt{n}} \right),
    \end{equation*}
    where $d_{\mathrm{TV}}$ is the distance in total variation.
\end{Prop}

We can define in an obvious way the subgraphs $G_n^k \subset G_n$ and
$G_\star^k \subset G_\star$ that are blowups of $T_n^k$
and~$T_\star^k$ with respect to some fixed family of distributions
$\nu$. Because there exists a coupling of $T_n^k$ and $T_\star^k$ such
that
$\mathds{P}(T_n^k = T_\star^k) = d_{\mathrm{TV}}(T_n^k,T_\star^k)$,
and because we can also couple the blowup procedure, we readily get,
for all $n,k\geq 1$:
\begin{equation}
  \label{eq:dTV-blowups}
  d_{\mathrm{TV}}(G_n^{k}, G_\star^{k}) \leq d_{\mathrm{TV}}(T_n^{k}, T_\star^{k}),
\end{equation}
with an equality if for all $k\geq 1$, the random networks
$\Gamma_k\sim\nu_k$ are almost surely $2$-connected. So in
\Cref{distanceTV}, we can replace $T_n$ and $T_\star$ by $G_n$ and
$G_\star$ respectively and keep the same upper bound.

Before closing this section with a reminder of Lemma~A.24
from \cite{bienvenu$B_2$IndexGalled2024}, let us recall some notation:
for any fixed phylogenetic network $G$ and node $v\in G$, we let
$p_{(G, v)}$ denote the probability that the random walk on $G$ reaches
$v$.

\begin{Lemm} \label{vanish}
With the previous notation, we have:
    \begin{enumerate}[label = (\roman*)]
        \item $\E[p_{(G_\star, v_k)}] = \P(\xi \neq 0)^k$.
        \item $p_{(G_\star, v_k)} = O( \P(\xi \neq 0)^k)$ almost surely. \qedhere
    \end{enumerate}
\end{Lemm}

\subsection[Moments of \texorpdfstring{$H_\alpha$}{H-alpha} for blowups of Galton--Watson trees: proof of \Cref{theoprincipal}]{%
Moments of \texorpdfstring{$\boldsymbol{H_\alpha}$}{H-alpha} for blowups of Galton--Watson trees:\newline proof of \Cref{theoprincipal}}

This section is dedicated to the proof of \Cref{theoprincipal}, whose
statement we recall here.
\begin{Theonn}
  Assuming that the offspring distribution $\xi$ is critical and satisfies $\P(\xi = 0) > 0$, we have:
  \begin{itemize}[label = \textbullet]
    \item For all $\alpha > 1$,
    \begin{enumerate}[label = (\roman*)]
      \item $H_\alpha(G_n) \rightarrow H_\alpha(G_\star)$ in distribution and,
      \item for all $ m \geq 1, ~ \mathds{E}[H_\alpha(G_n)^m] \rightarrow \mathds{E}[H_\alpha(G_\star)^m]$, and all these moments are finite.
    \end{enumerate}
    \item For all $\alpha \leq 1$, if $\xi$ has a finite third moment,
    \begin{enumerate}[label = (\roman*)]
      \item $H_\alpha(G_n) \rightarrow H_\alpha(G_\star)$ in distribution and,
      \item for all $m \geq 1$ such that $m ( 1-\alpha) < 1/2, ~\mathds{E}[H_\alpha(G_n)^m] \rightarrow \mathds{E}[H_\alpha(G_\star)^m]$, and all these moments are finite. \qedhere
    \end{enumerate}
  \end{itemize}
\end{Theonn}

\begin{Proo} In the case $\alpha=1$, this result was already proven in
  \cite[Theorem~3.7]{bienvenu$B_2$IndexGalled2024}. Since we adapt
  the method of proof to our broader setting, let us first recall
  briefly the facts that we borrow from this previous work. The key is
  the existence of a coupling, via Skorokhod's representation theorem,
  of the $(G_n)_{n\geq 0}$ and of $G_\star$ such that
  $G_n \to G_\star$ almost surely in the local topology. On this
  probability space, there must exist a random sequence
  $(k_n)_{n\geq 0}$ of integers that tends to $+\infty$, such that for
  all $0\leq n$,
  \begin{equation}
    \label{eq:proof-cvdist1}
    G_{n}^{k_n} = G_{\star}^{k_n} \qquad a.s,
  \end{equation}
  where $G_{n}^{k}$ and $G_{\star}^{k}$ are the truncations considered
  in \Cref{appPrelimBlowups}. This will be enough to tackle the case
  $\alpha > 1$, but let us first finish gathering the necessary facts
  from~\cite{bienvenu$B_2$IndexGalled2024}. Under the assumption of a
  third moment on $\xi$, for any deterministic sequence
  $(k_n)_{n\geq 1}$ satisfying
  \[
    \log n \ll k_n \ll \sqrt{n},
  \]
  for any $A\subset \mathbb{N}$, by
  \Cref{distanceTV},~\eqref{eq:dTV-blowups} and a Borel--Cantelli
  argument, we can find a deterministic subset $A'\subset A$ such
  that~\eqref{eq:proof-cvdist1} holds a.s.\ for all $n\in A'$ large
  enough. The final ingredient of the proof is due to the grafting
  property of \Cref{graftingproperty}: for any $n$
  satisfying~\eqref{eq:proof-cvdist1} we have
  \begin{equation}
    \label{eq:proof-cvdist2}
    H_\alpha(G_\star^{k_{n}}) \leq H_\alpha(G_{n})
    \leq H_\alpha(G_\star^{k_{n}}) + p_{(G_\star,
      v_{k_{n}})}^\alpha \cdot \frac{1}{\abs{2^{1-\alpha}-1}} \cdot
    \begin{cases}
      1 & \text{ if } \alpha > 1, \\
      n^{(1-\alpha)} & \text{ if } \alpha < 1.
    \end{cases}
  \end{equation}
  The proof is almost over in the case $\alpha >1$, since by
  \Cref{vanish}, the term $p_{(G_\star, v_{k_{n}})}^\alpha$ vanishes
  almost surely. By \Cref{Stubsequence},
  $H_\alpha(G_\star^{k_{n}}) \rightarrow H_\alpha(G_\star)$ almost
  surely, so $H_\alpha(G_n) \rightarrow H_\alpha(G_\star)$ in our
  coupling, which implies the convergence in distribution~(i). Since
  the random variables $H_\alpha(G_n)$ are uniformly bounded, we also
  immediately get convergence of all moments.

  We now focus on the case $\alpha < 1$. Using \Cref{vanish} and the
  fact that $k_n \gg \log n$, we get
  $H_\alpha(G_n) \rightarrow H_\alpha(G_\star)$ almost surely along
  sequences taking values in $A'$. To summarize, we have shown in
  particular that for any increasing sequence of integers
  $(\phi(n))_{n\geq 1}$, there exists a subsequence
  $(\phi(\psi(n))_{n\geq 1}$ such that
  $H_\alpha(G_{\phi(\psi(n))}) \rightarrow H_\alpha(G_\star)$ in
  distribution. This implies the convergence in distribution~(i), see
  e.g.\ Theorem~2.6 in
  \cite{billingsleyConvergenceProbabilityMeasures1999}.
  
  To show~(ii) in the case $\alpha < 1$, we will show that for all
  $m < 1/(2(1-\alpha))$, the quantity $\E\left[H_\alpha(G_n)^m\right]$
  is bounded. This fact implies that for each $m < 1/(2(1-\alpha))$,
  the sequence $(H_\alpha(G_n)^m)_{n\geq 1}$ is uniform integrable,
  and since we have already proved its convergence in distribution,
  the desired result follows (see e.g.\ Lemma~5.11 in
  \cite{kallenbergFoundationsModernProbability2002}).

  Let us now fix $m<1/(2(1-\alpha))$. Recall that the bounds
  from~\eqref{eq:proof-cvdist2} holds on the event
  $G_n^{k_n} = G_\star^{k_n}$, and otherwise we can bound
  $H_\alpha(G_n)$ by $\frac{n^{1-\alpha}}{2^{1-\alpha}-1}$. Therefore
  we can always write
  \begin{equation*}
    H_\alpha(G_n)^m \leqslant \mathds{1}_{\{G_n^{k_n} \neq
      G_\star^{k_n} \}} \frac{n^{m(1-\alpha)}}{(2^{1-\alpha}-1)^m}  +
    \mathds{1}_{\{G_n^{k_n} = G_\star^{k_n} \}} \left(
      H_\alpha(G_\star^{k_n}) + p_{(G_\star, v_{k_n})}^\alpha
      \frac{n^{(1-\alpha)}}{2^{1-\alpha}-1} \right)^m. 
  \end{equation*}
  Taking expectation and using the inequality
  $(a+b)^m \leqslant 2^{m-1}(a^m+b^m)$ yields
  \begin{equation*}
    \mathds{E}[H_\alpha(G_n)^m] \leqslant \frac{1}{(2^{1-\alpha}-1)^m}
    \left( d_{\mathrm{TV}}(T_n^{k_n}, T_\star^{k_n}) n^{m(1-\alpha)} +
      2^{m-1} \left( \mathds{E}[H_\alpha(G_\star^{k_n})^m] +
        \mathds{E}[p_{(G_\star, v_{k_n})}^{\alpha m}] \cdot
        n^{m(1-\alpha)} \right) \right).
  \end{equation*}
  Fix $\varepsilon>0$ such that
  $m(1-\alpha) \leqslant 1/2 - \varepsilon$. Taking for instance
  $k_n = n^{\varepsilon/2}$, \Cref{distanceTV} implies that
  $d_{\mathrm{TV}}(T_n^{k_n}, T_\star^{k_n}) n^{m(1-\alpha)}$ vanishes
  as $n \rightarrow +\infty$.
  By \Cref{vanish}, we also have
  \[
    \mathds{E}[p_{(G_\star, v_{k_n})}^{\alpha m}] \leq
    \begin{cases}
      \mathds{E}[p_{(G_\star, v_{k_n})}]^{\alpha m} \leq
      \mathds{P}(\xi>0)^{\alpha m k_n} & \text{ if } \alpha m < 1
                                         \text{ by Jensen's inequality}, \\
      \mathds{E}[p_{(G_\star, v_{k_n})}] \leq
      \mathds{P}(\xi>0)^{k_n} & \text{ if } \alpha m \geq 1,
    \end{cases}
  \]
  which readily implies that
  $\mathds{E}[p_{(G_\star, v_{k_n})}^{\alpha m}] \cdot
  n^{m(1-\alpha)}$ vanishes as well since $k_n$ is a positive power of
  $n$. Since $H_\alpha(G_\star^{k_n}) \leqslant H_\alpha(G_\star)$ it
  remains only to prove that $\mathds{E}[H_\alpha(G_\star)^m]$ is
  finite.

  Let us recall some notation from \Cref{sec:blowup-notation} and
  introduce some new one for our current blowup of $T_\star$. We
  denote by $\Gamma_{v}$ the network that replaces $v\in T_\star$
  during the blowup procedure. We also write $\hat{\xi}_k$ for the
  number of children of $v_k$ in $T_\star$ and for
  $i \leqslant \hat{\xi}_k$, let us write $q_{k,i}$ for the
  probability that the directed random walk in $G_\star$ reaches the
  $i$-th leaf of $\Gamma_{v_k}$. By the leaf exchangeability property,
  we can assume without loss of generality that the leaves of
  $\Gamma_{v_k}$ are labeled so that
  $q_{k, 1} = p_{(G_\star, v_{k+1})}$. Finally, for $i \geqslant 2$
  (when it makes sense) $G_{k,i}$ will be the network grafted on the
  $i$-th leaf of $\Gamma_{v_k}$ in $G_\star$. Once again using
  \Cref{graftingproperty} repeatedly, we have:
  \begin{equation*}
    H_\alpha(G_\star) = \sum_{k \geqslant 0} \left( p_{(G_\star, v_k)}^\alpha H_\alpha(\Gamma_{v_k}) + \sum_{i = 2}^{\hat{\xi}_k}  q_{k, i}^\alpha H_\alpha(G_{k, i}) \right).
  \end{equation*}
  Using this expression we can bound $\E[ H_\alpha(G_\star)^m]$ by
  \begin{equation} \label{majEsp}
    \mathds{E}[H_\alpha(G_\star)^m]
    \leqslant 2^{m-1} \Bigg( \mathds{E}\bigg[\bigg(\sum_{k \geqslant
            0} p_{(G_\star, v_k)}^\alpha H_\alpha
          (\Gamma_{v_k})\bigg)^m \bigg] +
      \mathds{E}\bigg[\bigg(\sum_{k \geqslant 0} \sum_{i
            =2}^{\hat{\xi}_k} q_{k, i}^\alpha H_\alpha(G_{k,
            i})\bigg)^m \bigg] \Bigg).
  \end{equation}
  Using the independence of $p_{(G_\star, v_k)}$ and $\Gamma_{v_k}$, and
  bounding $H_\alpha(\Gamma_{v_k})$ by its maximum possible value, we obtain
  \begin{eqnarray*}
    \mathds{E}\Big[\big(p_{(G_\star, v_k)}^{\alpha} H_\alpha(\Gamma_{v_k})\big)^m\Big] &=& \mathds{E}[p_{(G_\star, v_k)}^{\alpha m}] \cdot \mathds{E}[H_\alpha(\Gamma_{u_k})^m] \\
                                                                                     & \leqslant & \mathds{E}[p_{(G_\star, v_k)}^{\alpha m}] \cdot \mathds{E}[\hat{\xi}^{(1-\alpha)m}] \frac{1}{(1-2^{1-\alpha})^m} \\
                                                                                     & \leqslant & \mathds{E}[p_{(G_\star, v_k)}^{\alpha m}] \cdot \mathds{E}[\hat{\xi}^{1/2}] \frac{1}{(1-2^{1-\alpha})^m}.
  \end{eqnarray*}
  Note that $\mathds{E}[\hat{\xi}^{1/2}]$ is finite because $\xi$ is
  assumed to have a third moment, and by
  \Cref{vanish} we can write
  $\E[p_{(G_\star, v_k)}^{\alpha m}] = O(\zeta^{ k})$, with
  $\zeta = \P(\xi \neq 0)^{1\wedge (\alpha m)} < 1$. Thus
  $\lVert p_{(G_\star, v_k)}^\alpha H_\alpha(\Gamma_{v_k})\rVert_m =
  O(\zeta^{k/m})$, so the triangle inequality yields
  \begin{equation*}
    \bigg\lVert\sum_{k \geqslant 0} p_{(G_\star, v_k)}^\alpha
    H_\alpha(\Gamma_{v_k}) \bigg\rVert_m \; \leq \; \sum_{k \geqslant
      0} \norm{ p_{(G_\star, v_k)}^\alpha
      H_\alpha(\Gamma_{v_k})}_m < + \infty,
  \end{equation*}
  meaning that the first term on the right-hand side of~\eqref{majEsp}
  is bounded. Now, let us bound the second term. Letting
  $q_\alpha = \sum_{k \geqslant 0} \sum_{i =2}^{\hat{\xi}_k} q_{k,
    i}^\alpha$, Jensen's inequality gives :
  \begin{eqnarray*}
    \mathds{E}\bigg[\bigg(\sum_{k \geqslant 0} \sum_{i =2}^{\hat{\xi}_k} q_{k, i}^\alpha H_\alpha(G_{k, i})\bigg)^m \bigg] & = & \mathds{E}\bigg[q_\alpha ^m \bigg(\sum_{k \geqslant 0} \sum_{i =2}^{\hat{\xi}_k} \frac{q_{k, i}^\alpha}{q_\alpha} H_\alpha(G_{k, i})\bigg)^m \bigg] \\
                                                                                                                             & \leqslant & \mathds{E}\bigg[q_\alpha ^{m-1} \sum_{k \geqslant 0} \sum_{i =2}^{\hat{\xi}_k} q_{k, i}^\alpha H_\alpha(G_{k, i})^m \bigg].
  \end{eqnarray*}
  Letting $G$ be a blowup of a Galton--Watson tree with offspring
  distribution $\xi$, using the independence of the $G_{k,i}$'s and
  the $q_{k,i}$'s and noticing that each $G_{k,i}$ is distributed as
  $G$, we have
  \begin{equation*}
    \mathds{E}\bigg[q_\alpha ^{m-1} \sum_{k \geqslant 0} \sum_{i
      =2}^{\hat{\xi}_k} q_{k, i}^\alpha H_\alpha(G_{k, i})^m \bigg] =
    \mathds{E}[q_\alpha ^m H_\alpha(G)^m].
  \end{equation*}
  Moreover,
  \begin{eqnarray*}
    \mathds{E}[q_\alpha^m H_\alpha(G)^m] & \leqslant & \mathds{E}[q_\alpha^m |\mathcal{L}_T|^{(1-\alpha)m} ] \frac{1}{(1-2^{1-\alpha})^m} \\
                                         & \leqslant & \mathds{E}[q_\alpha^m |\mathcal{L}_T|^{1/2 - \varepsilon} ] \frac{1}{(1-2^{1-\alpha})^m} \\
                                         & = & \mathds{E}[q_\alpha^m] \cdot \mathds{E}[|\mathcal{L}_T|^{1/2 - \varepsilon} ] \frac{1}{(1-2^{1-\alpha})^m}.
  \end{eqnarray*}
  Since $\mathds{E}[\xi] = 1$ and $\mathds{E}[\xi^2]$ is finite, it is
  classic that
  $\mathds{P}(|\mathcal{L}_T| \geqslant n) = \Theta(n^{-1/2})$, see
  e.g.\ \cite[Proposition~24]{Ald93}
  or~\cite[Theorem~3.1]{Kor12}. This implies easily
  $\mathds{E}[|\mathcal{L}_T|^{1/2 - \varepsilon} ] < + \infty$. To
  conclude, it suffices to show that $\mathds{E}[q_\alpha^m]$ is
  finite. By definition of $q_{k, i}$ we have
\begin{equation*}
\mathds{E}\bigg[\Big( \sum_{k \geqslant 0} \sum_{i=2}^{\hat{\xi}_k} q_{k, i}^\alpha \Big)^m \bigg] \leqslant \mathds{E}\bigg[ \Big( \sum_{k \geqslant 0} \hat{\xi}_k^{1-\alpha} p_{(G_\star, v_k)}^\alpha \Big) ^m\bigg],
\end{equation*}
and we can show exactly as above, using the triangle inequality, that $\mathds{E}[q_\alpha^m]$ is finite.
\end{Proo}

\end{document}